\documentclass[preprint,3p]{amsart}

\usepackage{amsmath, amssymb}

\usepackage{enumerate}

\usepackage{graphicx}

\newcommand{\aut}{\operatorname{Aut}}

\newcommand{\sym}{\operatorname{Sym}}
\newcommand{\alt}{\operatorname{Alt}}
\newcommand{\dih}{\operatorname{Dih}}
\newcommand{\sdih}{\operatorname{SDih}}
\newcommand{\syl}{\operatorname{Syl}}

\newcommand{\SL}{\mathrm{SL} }
\newcommand{\PSL}{\mathrm{PSL}}

\newcommand{\GL}{\mathrm{GL}}
\newcommand{\GF}{\mathrm{GF}}

\newcommand{\HN}{\mathrm{HN}}
\newcommand{\HS}{\mathrm{HS}}

\newcommand{\wt}[1]{\widetilde{#1}}

\def \<{\langle }
\def \>{\rangle }
\def \bs {\backslash }
\renewcommand{\bar}{\overline}
\def \inv {^{-1}}
\def \GF {\mathrm{GF} }

\def \GL {\mathrm{GL} }
\def \sym {\mathrm{Sym} }
\def \alt {\mathrm{Alt} }
\def \dih {\mathrm{Dih} }
\def \syl {\mathrm{Syl} }
\def \bs {\backslash }
\renewcommand{\bar}{\overline}
\renewcommand{\leq}{\leqslant}
\renewcommand{\geq}{\geqslant}

\def \b{\beta }
\def \a{\alpha }

\input cyracc.def

\renewcommand{\hat}{\widehat}

\newtheorem{thm}{Theorem}[section]
\newtheorem{lemma}[thm]{Lemma}
\newtheorem{cor}[thm]{Corollary}
\newtheorem{hyp}[thm]{Hypothesis}
\newtheorem{notation1}[thm]{Notation}
\begin{document}

%\begin{frontmatter}

\title{A 3-Local Characterization of the Harada--Norton Sporadic Simple Group}

\author{Sarah Astill}

\maketitle

\address{Department of Mathematics, University of Bristol,
University Walk,
Bristol, BS8 1TW}
%
%
%\ead{Sarah.Astill@bristol.ac.uk}
%
\begin{abstract}
We provide $3$-local characterizations of the Harada--Norton sporadic simple group and its automorphism group. Both groups are examples of groups of parabolic characteristic three and we identify them from the structure of the normalizer of the centre of a Sylow $3$-subgroup.
\end{abstract}
%
%\begin{keyword}
%Groups \sep simple groups
%\end{keyword}
%
%\end{frontmatter}

\section{Introduction}

In \cite{HaradaHN} in 1975, Harada introduced a  new simple group. He proved that a group with an
involution whose centralizer is a double cover of the automorphism group of the Higman--Sims
sporadic simple group is simple of order $2^{14}.3^6.5^6.7.11.19$. In 1976, in his PhD thesis,
Norton proved such a group exists and thus we have the Harada--Norton sporadic simple group, $\HN$.
The simple group was not proved to be unique until 1992. In \cite{SegevHN}, Segev proves that there
is a unique group $G$ (up to isomorphism) with two involutions $u$ and $t$ such that $C_G(u)\sim (2
^{\cdot} \HS) : 2$ and $C_G(t)\sim 2_+^{1+8}.(\alt(5)\wr 2)$ with $C_G(O_2(C_G(t)))\leq O_2(C_G(t))$. We
can therefore define the group $\HN$ by the structure of two involution centralizers in this way.

In the ongoing project to understand groups of local and parabolic characteristic $p$ (see for example \cite{MSS-overview}) it has been observed that both $\HN$ and $\aut(\HN)$ are example of groups of parabolic characteristic $3$. This is to say that every $3$-local subgroup, $H$, containing a Sylow $3$-subgroup satisfies $C_H(O_p(H))\leq O_p(H)$. The aim of this paper is therefore to characterize $\HN$ and $\HN :2$ in terms of their $3$-structure. The hypothesis we consider and the theorem we prove are as follows.
\begin{hyp}\label{mainhyp} Let $G$ be a group and let $Z$ be the centre of a Sylow $3$-subgroup of $G$
with $Q:=O_3(N_G(Z))$. Suppose that
\begin{enumerate}[$(i)$]
\item $Q\cong 3_+^{1+4}$;
\item $C_G(Q)\leq Q$;
\item $Z \neq Z^x \leq Q$ for some $x \in G$; and
\item $N_G(Z)/Q $ has shape $4^{.}\alt(5)$ or $4^{.}\sym(5)$.
\end{enumerate}
\end{hyp}
\begin{thm}
If $G$ satisfies Hypothesis \ref{mainhyp} then $G \cong \HN$ or $G \cong \aut(\HN)$.
\end{thm}

For a large part of this proof we work under the hypothesis that $N_G(Z)/Q $ has shape either $4^{.}\alt(5)$ or $4^{.}\sym(5)$. We will refer to these two possibilities as Case I and Case II respectively. In Section \ref{HN-Section-3Local}, we determine the possibilities for the structure of certain $3$-local subgroups of
$G$. This allows us to see the fusion of elements of order three
in $G$. In particular, it allows us to identify a distinct conjugacy class of elements of order
three. In $3$-local recognition results, it is often necessary to determine $C_G(x)$ for each
element of order three in $G$. In this case, we have just one further centralizer to determine
which is isomorphic to $3 \times \alt(9)$ or $3 \times \sym(9)$ and we use a theorem due to Prince \cite{PrinceSym9} to recognize this centralizer.

In Section \ref{HN-Section-CG(t)} we determine the  structure of $C_G(t)$ where $t$ is a
$2$-central involution. This requires a great deal of $2$-local analysis, in particular, we must take full advantage of our knowledge of the $2$-local subgroups in $\alt(9)$ and use a theorem due to
Goldschmidt about $2$-subgroups with a strongly closed abelian subgroup. The determination
of $C_G(t)$ is much more difficult than similar recognition results (in \cite{AstillO8+2.3} for example). A reason for this may
be that the $3$-rank of $C_G(t)/O_2(C_G(t))$ is just two whilst the $2$-rank is four. An
easier example may have greater $3$-rank and lesser $2$-rank. We also show in Section \ref{HN-Section-CG(t)} that in Case II of the hypothesis, $G$ has a proper normal subgroup which satisfies Case I. Once have made this observation, our calculations are simplified significantly as we can reduce to a Case I hypothesis only.

One conjugacy class of involution centralizer is not enough to recognize $\HN$ and so in Section
\ref{HN-Section-CG(u)} we prove that $G$ also has an involution centralizer which has shape $(2^{\cdot}\HS):2$  by making use of a theorem of Aschbacher. The results of Sections
\ref{HN-Section-CG(t)} and \ref{HN-Section-CG(u)} allow us to apply the uniqueness theorem by Segev
to prove that in Case I $G\cong \HN$. It then follows easily that in Case II, $G\cong \aut(\HN)$.

All groups in this article are finite.  We note that $\sym(n)$ and $\alt(n)$ denote the symmetric and alternating groups of degree $n$ and $\dih(n)$ and $\mathrm{Q}(n)$ denote the dihedral group and quaternion groups of order $n$. Notation for classical groups follows \cite{Aschbacher}. All other groups and notation for group extensions follows the {\sc Atlas} \cite{atlas} conventions. In particular we mention that the shape of a group is some description of its normal structure and we use the symbol $\sim$ (for example if $G\cong \sym(4)$, we may choose to write $G\sim 2^2.\sym(3)$).   If $H$ is a
group acting on a set containing $x$ then $x^H$ is the orbit of $x$ under $H$. If a group $A$ acts on a group $B$ and $a \in A$ and $b \in B$ then $[b,a]=b\inv b^a$.
Further group theory notation and terminology is standard as in \cite{Aschbacher} and
\cite{stellmacher} except that $\mathbf{Z}(H)$ denotes the centre of a group $H$.

\section{Preliminary Results}

\begin{thm}[Aschbacher]\cite{AschbacherHS}\label{Aschbacher-HS}
Let $G$ be a group with an involution $t$ and set $H:=C_G(t)$. Let $V \leq G$ such that  $V\cong 2 \times 2
\times 2$ and set $M:=N_G(V)$. Suppose that
 \begin{enumerate}[$(i)$]
 \item $O_2(H)\cong 4 *2_+^{1+4}$ and $H/O_2(H) \cong \sym(5)$; and
 \item $V \leq O_2(H)$, $O_2(M)\cong 4 \times 4 \times 4$ and $M/O_2(M)\cong \GL_3(2)$.
\end{enumerate} Then $G\cong\HS$.
\end{thm}

\begin{thm}[Segev]\cite{SegevHN}\label{Segev-HN}
Let G be a finite group containing two involutions $u$ and $t$ such that $C_G(u)\cong (2 ^. \HS) : 2$
and $C_G(t)\sim 2_+^{1+8}.(\alt(5)\wr 2)$ with $C_G(O_2(C_G(t)))\leq O_2(C_G(t))$. Then $G \cong \HN$.
\end{thm}

Recall that given a $p$-group $S$, we set $\Omega(S)=\<x \mid x^p=1\>$. \begin{thm}[Goldschmidt]\cite[p370]{stellmacher} \label{goldschmidt}
Let $S$ be a Sylow 2-subgroup of a group $G$ and let $A$ be an abelian subgroup of $S$ such that $A$ is
strongly closed in $S$ with respect to $G$. Suppose that $G = \<A^G\>$ and $O_{2'}(G) = 1$.
Then $G = F^*(G)$ and $A = O_2(G)\Omega(S)$.
\end{thm}

\begin{thm}[Hayden]\cite[3.3, p545]{HaydenPSp43}.\label{Hayden}
Let $G$ be a finite group and let $T \in \syl_3(G)$ be elementary
abelian of order nine. Assume that
\begin{enumerate}[$(i)$]
    \item $N_G(T)/C_G(T)\cong 2\times 2$;
    \item $C_G(T)=T$; and
    \item $C_G(t)\leq N_G(T)$ for each $t \in T^\#$.
\end{enumerate}
Then $G=N_G(T)$.
\end{thm}

\begin{thm}[Feit--Thompson] \cite{FeitThompson}\label{Feit-Thompson}
Let $G$ be a finite group containing a subgroup, $X$, of order three such that $C_G(X)=X$. Then one
of the following holds:
\begin{enumerate}[$(i)$]
    \item $G$ contains a nilpotent normal subgroup, $N$, such that $G/N\cong \sym(3)$ or
    $C_3$;
    \item $G$ contains an elementary abelian normal 2-subgroup, $N$, such that
    $G/N\cong \mathrm{Alt}(5)$; or
    \item $G\cong\mathrm{PSL}_2(7)$.
\end{enumerate}
\end{thm}
The result can be found in  \cite{FeitThompson} however the additional information in conclusion
$(ii)$ that $N$ is elementary abelian uses a theorem of Higman \cite{Higman}.

\begin{thm}[Prince]\cite{PrinceSym9}\label{princeSym9}
Let $G$ be a group and suppose $x \in G$ has order 3 such that $C_G(x)\cong
C_{\sym(9)}(\hspace{0.5mm} (1,2,3)(4,5,6)(7,8,9) \hspace{0.5mm} )$ and there exists $J\leq C_G(x)$
which is elementary abelian of order 27 and normalizes no non-trivial $3'$-subgroup of $G$. Then
either $J\vartriangleleft G$ or $G \cong \sym(9)$.
\end{thm}

\begin{lemma}\label{Prelims-sym9}
Let $G$ be a group of order $3^42$ with $S \in \syl_3(G)$ and $T \in \syl_2(G)$ and
$J\vartriangleleft G$ elementary abelian of order 27. Suppose that $Z:=\mathbf{Z}(S)$ has order three and
$Z \leq C_S(T)\neq S$. Then $G\cong C_{\sym(9)}(\hspace{0.5mm} (1,2,3)(4,5,6)(7,8,9) \hspace{0.5mm} )$.
\end{lemma}
\begin{proof}
We have that $T$ normalizes $Z$ and $J$ and so by Maschke's Theorem, there exists a subgroup $K \leq J$ such that $K$ is a $T$-invariant complement to $Z$ in $J$. Set $L:=KT$ then $K\trianglelefteq L$ and $[G:L]=9$. Suppose that $N \leq L$ and that $N$ is normal in $G$. If $3\mid |N|$ then $N\cap \mathbf{Z}(S) \neq 1$ which is a contradiction since $Z \nleq K$. So $N$ is a $2$-group which implies $N=1$ otherwise $G$ has a central involution. Hence there is an injective homomorphism from $G$ into $\sym(9)$. Moreover there is a map from $G$ into the centralizer in $\sym(9)$ of the centre of a Sylow $3$-subgroup. Since  $| C_{\sym(9)}(\hspace{0.5mm} (1,2,3)(4,5,6)(7,8,9) \hspace{0.5mm}|=|G|$, we have an isomorphism.
\end{proof}

\begin{thm}[Parker--Rowley]\cite{ParkerRowleyA8}\label{ParkerRowleyA8}
Let $G$ be a finite group with $R :=\<a, b\>$ an elementary abelian Sylow 3-subgroup of $G$ of order
nine. Assume the following hold. \begin{enumerate}[$(i)$]
 \item  $C_G(R)=R$ and $N_G(R)/C_G(R)\cong Dih(8)$.
 \item $C_G(a)\cong 3\times \alt(5)$ and $N_G(\<a\>)$ is isomorphic to the diagonal subgroup
 of index two in  $\sym(3)\times \sym(5)$.
 \item $C_G(b)\leq N_G(R)$, $C_G(b)/R\cong 2$ and $N_G(\<b\>)/R\cong 2\times 2$.
 \end{enumerate}Then $G$ is isomorphic to $\alt(8)$.
\end{thm}

\begin{cor}\label{Cor-ParkerRowleyA8}
Let $G$ be a group and $\alt(8)\cong H\leq G$ such that for $R \in \syl_3(H)$ and each $r \in
R^\#$, $C_G(r) \leq H$. Then $G=H$.
\end{cor}
\begin{proof}
Suppose $R$ is not a Sylow $3$-subgroup of $G$. Then there exists $R<S\in \syl_3(G)$. Therefore
$R<N_S(R)$ and $1 \neq r \in \mathbf{Z}(N_S(R))\cap R$. Therefore $N_S(R)\leq C_G(r)\leq H$ which is a
contradiction. Thus $R \in \syl_3(G)$. Pick $a,b\in R$ such that $C_H(a)\cong 3 \times \alt(5)$ and
$C_H(b)\leq N_H(R)$. Now we check the hypotheses of Theorem \ref{ParkerRowleyA8}. We have that for
any $r \in R^\#$, $C_G(R)\leq C_G(r) \leq H$ and so $C_G(R)=C_H(R)=R$. So consider $N_G(R)/C_G(R)$
which is isomorphic to a subgroup of $\GL_2(3)$. Since $R \in \syl_3(G)$, $N_G(R)/R$ is a
$2$-group. Also $N_H(R)/R\cong \dih(8)$. Suppose $N_G(R)/R \cong \sdih(16)$. Then $N_G(R)$ is
transitive on $R^\#$ which is a contradiction. Therefore  $N_G(R)=N_H(R)$ and $N_G(R)/C_G(R)\cong
\dih(8)$ so $(i)$ is satisfied. Now $C_G(a)=C_H(a)$ and there exists some $x \in H$ that inverts $a$. Therefore $N_H(\<a\>)=C_H(a)\<x\>\leq H$. Similarly $C_G(b)=C_H(b)$ and there exists some $y \in H$ that  inverts $b$. Therefore $N_H(\<b\>)=C_H(b)\<y\>\leq H$.
Thus $(ii)$ and $(iii)$ are satisfied so $G=H\cong \alt(8)$.
\end{proof}

\begin{lemma}\cite[3.20 $(iii)$]{ParkerRowley-book}\label{Parker-Rowley-SL2(q)-splitting}
Let $X\cong \SL_2(3)$ and $S \in \syl_3(X)$. Suppose that $X$ acts on an elementary  abelian
$3$-group $V$  such that $V=\<C_V(S)^X\>$, $C_V(X)=1$ and $[V,S,S]=1$. Then $V$ is a direct product
of natural modules for $X$.
\end{lemma}

\begin{lemma}\label{conjugation in thompson subgroup}
Let $G$ be a group, $p$ be a prime and $S\in \syl_p(G)$. Suppose $J(S)$ is abelian and suppose $a,b
\in J(S)$ are conjugate in $G$. Then $a$ and $b$ are conjugate in $N_G(J(S))$.
\end{lemma}
\begin{proof}
Suppose $a^g=b$ for some $g \in G$. Notice first that it follows immediately from the definition of
the Thompson subgroup that $J(S)^g=J(S^g)$. Now $J(S),J(S^g)\leq C_G(b)$. Let $P,Q \in
\syl_p(C_G(b))$ such that $J(S)\leq P$ and $J(S^g)\leq Q$. Again, by the definition of the Thompson
subgroup, it is clear that $J(S)\leq P$ implies $J(S)=J(P)$ and similarly $J(S^g)= J(Q)$. By
Sylow's Theorem, there exists $x \in C_G(b)$ such that $Q^x=P$ and so $J(S)=J(P)=
J(Q)^x=J(S)^{gx}$. Thus $gx \in N_G(J(S))$ and $a^{gx}=b^x=b$ as required.
\end{proof}

\begin{lemma}\label{Prelims 2^8 3^2 Dih(8)}
Let $X$ be a group with an elementary abelian subgroup $E\vartriangleleft X$ of order $2^{2n}$ such
that $C_X(E)=E$. Let $S \in \syl_2(X)$ and suppose that whenever $E<R\trianglelefteq S$ with $R/E$
elementary abelian and $|R/E|=2^a$ we have  $|C_E(R)|\leq 2^{2n-a-1}$. Then $E$ is characteristic in
$S$.
\end{lemma}
\begin{proof}
First observe that since $C_X(E)=E$, $X/E$ is a group of outer automorphisms of $E$. Let $\a$ be an
automorphism of $S$ such that $E^\a\neq E$. Then $R:=EE^\a\trianglelefteq S$. Since $E^\a$ is elementary
abelian, we have that $E^\a/(E \cap E^\a)\cong EE^\a/E=R/E$ is elementary abelian and $E \cap E^\a$ is central in
$R$. If $|R/E|=2^a$ then $|E \cap E^\a|=2^{2n-a}$ so $|C_E(R)|\geq |E \cap
E^\a|=2^{2n-a}$ which is a contradiction.
\end{proof}

\begin{lemma}\label{prelims-extraspecial 2^5 in GL(4,3)}
Let $E\leq \GL_4(3)$ such that $|E|=2^5$, and $|\Phi(E)|\leq 2$. Furthermore let $S \leq \GL_4(3)$ be elementary abelian of order nine
such that $S$ acts faithfully on $E$. If $\mathrm{Q}(8) \cong A \cong B$ with $A\neq B$ both $S$-invariant
subgroups of $E$, then $E\cong 2_+^{1+4}$ and $E$ is uniquely determined up to conjugation in
$\GL_4(3)$.
\end{lemma}
\begin{proof}
Note that $E$ is non-abelian since $A,B \leq E$. Therefore $|E/\Phi(E)|=2^4$  is acted on
faithfully by ${S}$. Hence, $S$ is isomorphic to a subgroup of $\GL_4(2)$. Now observe that $\GL_4(2)$ has Sylow $3$-subgroups of order nine which
contain an element of order three which acts fixed-point-freely on the natural module. Thus any
${S}$-invariant subgroup of $E$ properly containing $\Phi(E)$ has order $2^3$ or $2^5$. Since $A$
and $B$ are distinct and normalized by $S$, we have $E=AB$. Suppose $|\mathbf{Z}(E)|>2$. Then
$|\mathbf{Z}(E)|=8$ is ${S}$-invariant. By coprime action,
$\mathbf{Z}(E)=\<C_{\mathbf{Z}(E)}(s)|s \in S^\#\>$. Thus there exists $s \in S^\#$ such that
$C_{\mathbf{Z}(E)}(s)>\Phi(E)$. Since $E=AB$, we find $a\in A$ and $b\in B$ such that $ab \in
C_{\mathbf{Z}(E)}(s)\bs \Phi(E)$. Then, as $s$ normalizes $A$ and $B$,  $s$ must centralize $a$
and $b$. Now $C_E(s)$ is $S$-invariant with $|C_E(s) \cap A|\geq 4$ and $|C_E(s) \cap B|\geq 4$. It
follows that $[E,s]=1$ which is a contradiction. Thus $\mathbf{Z}(E)=\Phi(E)$ and so $E$ is extraspecial and $E\cong 2_+^{1+4}$.

Since $E$ is extraspecial, $[E:E']=2^4$. Therefore there are sixteen 1-dimensional representations of $E$ over $\GF(3)$.
Moreover there is a $4$-dimensional representation of $E$ since $E \leq \GL_4(3)$. Since
$16+4^2=32=|E|$, this accounts for all the irreducible representations of $E$ over $\GF(3)$. Hence there is a unique
$4$-dimensional representation of $E$ and so there is one conjugacy class of such subgroups in
$\GL_4(3)$.
\end{proof}

The following lemma is an application of Extremal Transfer (see \cite[15.15, p92]{GLS2}
that will be needed in Section \ref{HN-Section-CG(u)}.
\begin{lemma}\label{Prelims-4*4*4 transfer}
Let $G$ be a group and $4\times 4 \times 4 \cong A \leq G$ with $C_G(A)=A$. Set $X:=N_G(A)$ and assume that $X\sim
4^3: (2 \times \GL_3(2))$ contains a Sylow $2$-subgroup of $G$. Furthermore suppose that there
exists an involution $u \in X \bs O^2(X)$ such that $C_G(u)\cong 2 \times \sym(8)$. Then $u \notin
O^2(G)$.  In particular, $O^2(G)\neq G$.
\end{lemma}
\begin{proof}
Let $Y:=O^2(X)$ then $Y/A\cong \GL_3(2)$ and $u \notin Y$. We assume for a contradiction that for
some $g \in G$, $r:=u^g \in Y$ and so we apply \cite[15.15, p92]{GLS2} to see that
$C_X(r)$ contains a Sylow $2$-subgroup of $C_G(r)\cong 2 \times \sym(8)$. Observe first that $r
\notin A$ because no element of order four in $G$ squares to $r$.

Set $V:=\Omega(A)\cong 2^3$ and let $S\in \syl_2(C_X(r))$. Then $|S|=2^8$ and therefore $|S \cap A|\geq 2^4$. It
follows that $S \cap A\cong 4 \times 4$ since $Ar\in Y/A\cong \GL_3(2)$ acts faithfully on $V$ and
therefore $|C_V(r)|\leq 2^2$. In particular, $|C_A(r)|=2^4$ and so $SA\in \syl_2(X)$.

Since $X/A\cong 2\times \GL_3(2)$, $2 \times \dih(8)\cong SA/A \cong S /(A \cap S)=S/C_A(r)$. Set
$S_0:=S \cap Y$ then $r \in S_0$ and we have that $\dih(8)\cong S_0A/A \cong S_0 /(A \cap
S_0)=S_0/C_A(r)$. Since $r \in \mathbf{Z}(S)$, $C_A(r)r\in \mathbf{Z}(S_0/C_A(r))$. Therefore $S_0/\<C_A(r),r\>\cong
2 \times 2$. Let $C_A(r)<R<S$ such that $|R/C_A(r)|=2$ and $S=S_0R$ and $[R,S_0]\leq C_A(r)$. This
is possible as $S/C_A(r)\cong 2 \times \dih(8)$. We have therefore that $[R,S_0]$, $S_0 \cap R \leq
C_A(r) \leq \<C_A(r),r\>$ and so $S/\<C_A(r),r\> \cong 2 \times 2 \times 2$. Now
$\<C_A(r),r\>/\<r\>\cong C_A(r)\cong 4 \times 4$. Hence, $S/\<r\>\sim (4 \times 4). (2 \times 2
\times 2)$  which is a subgroup of $C_G(r)/\<r\>\cong \sym(8)$. However, a $2$-subgroup of
$\sym(8)$ has non-abelian derived subgroup which supplies us with a contradiction. Thus $u \notin
O^2(G)$.
\end{proof}

\begin{lemma}\label{Prelims-centralizers of invs on a vspace which invert a 3}
Let $G$ be a group with a normal $2$-subgroup $V$ which is elementary abelian of order
$2^n$. Suppose $t$ and $w$ are in $G$ such that $Vt$ has order two and $Vw$ has order three and
$Vt$ inverts $Vw$. If $|C_V(w)|=2^a$ then $|C_V(t)|\leq 2^{(n+a)/2}$.
\end{lemma}
\begin{proof}
Since $Vt$ inverts $Vw$, we have that $Vw=Vtw^2t$ and so $Vw^2=Vtw^2tw=VtVt^w$. Therefore $C_V(t)
\cap C_V(t^w) \leq C_V(w^2)=C_V(w)$. We have that $|V| \geq
|C_V(t)C_V(t^w)|=|C_V(t)||C_V(t^w)|/|C_V(t) \cap C_V(t^w)|$ and so $2^n \geq |C_V(t)|^2/2^a$
which implies $|C_V(t)|\leq 2^{(n+a)/2}$.
\end{proof}

\begin{lemma}\label{lem-conjinvos}
Let $G$ be a finite group and $V\trianglelefteq G$ be an elementary abelian $2$-group. Suppose that
$r\in G$ is an involution such that $C_{V}(r)=[V,r]$. Then
\begin{enumerate}[$(i)$]
\item every involution in $Vr$ is conjugate to $r$; and
\item
$|C_{G}(r)|=|C_{V}(r)||C_{G/V}(Vr)|$.
\end{enumerate}
\end{lemma}
\begin{proof}
$(i)$ Let $t\in Vr$ be an involution. Then $t=qr$, for some $q\in V$. Since $t^2=1$, we have that
$1=qrqr=[q,r]$ as $r$ and $q$ have order at most two. So $q\in C_{V}(r)=[V,r]$. So
$q=q_{1}rq_{1}r$, for some $q_{1}\in V$, and therefore $t=q_{1}rq_{1}rr=r^{q_{1}}$ and so $t$ is
conjugate to $r$ by an element of $V$.

$(ii)$ Define a homomorphism, $\phi:C_{G}(r)\rightarrow C_{G/V}(Vr)$  by $\phi(x)=Vx$. Then $\ker
\phi=C_{V}(r)$. Moreover, if $Vy\in  C_{G/V}(Vr)$ then $Vr^y=Vr$. Hence, using $(i)$ we see that
there exists $q \in V$ such that $r^y=r^q$. Therefore $q\inv y \in C_G(r)$ and of course $Vq\inv y
=Vy$ and so $\phi(q\inv y)=Vy$. Therefore $\phi$ is surjective. Thus, by an isomorphism theorem,
$C_{G}(r)/C_{V}(r)\cong C_{G/V}(Vr)$ and $|C_{G}(r)|=|C_{V}(r)||C_{G/V}(Vr)|$, as required.
\end{proof}

\section{Determining the 3-Local Structure of $G$}\label{HN-Section-3Local}

%%%%%%%%%%%%%%%%%%%%%%%%%%%%%%%%%%%%%%%%%%%%%%%%%%%%%%%%%%%%%%%%%%%%%%%%%%%%%%%%%%%%%%%%%%%%%%%%%%%%%%%%%%%%%%%%%%%%%%%%%%%%%%%%%%%%%%%%%%%%%

We assume Hypothesis \ref{mainhyp}. Let $x \in G\bs N_G(Z)$ such that $Z^x \leq Q$ and set $Y:=ZZ^x \leq Q$. We begin by making some easy observations in particular noting that $Z \leq Q^x$ and so our hypothesis is symmetric. For a large part of this proof we work under the hypothesis that $N_G(Z)/Q $ has shape either $4^{.}\alt(5)$ or $4^{.}\sym(5)$. We will refer to these two possibilities as Case I and Case II respectively. At the end of Section \ref{HN-Section-CG(t)} we are able to see that in Case II $G$ has an proper normal subgroup which satisfies the hypothesis of Case I and so from that point we consider Case I only.

\begin{lemma}\label{Prelims-EasyLemma}
\begin{enumerate}[$(i)$]
\item $|Z|=3$.
\item $C_{C_G(Z)}(Q/Z)=Q$.
\item $Z \leq Q^x$.
\item $Q \cap Q^x$ is elementary abelian.
\end{enumerate}
\end{lemma}
\begin{proof}
$(i)$ By hypothesis, $Z$ is the centre of a Sylow $3$-subgroup, $S$ say, of $G$, therefore $\syl_3(C_G(Z)) \subseteq \syl_3(G)$. We have that $Q=O_3(C_G(Z))$ and $C_G(Q) \leq Q$. Therefore $[Q,Z]=1$ implies that $Z \leq Q$. Thus $Z \leq \mathbf{Z}(Q)$ and $|\mathbf{Z}(Q)|=3$ since $Q$ is extraspecial. Hence $Z=\mathbf{Z}(Q)=\mathbf{Z}(S)$ has order three.

$(ii)$ Suppose that $p$ is a prime and  $g \in C_G(Z)$ is a $p$-element such that $[Q/Z,g]=1$.
Then if $p \neq 3$ we may apply coprime action to say that $Q/Z=C_{Q/Z}(g)=C_Q(g)Z/Z= C_Q(g)/Z$
and so $[Q,g]=1$ which is a contradiction as $C_G(Q) \leq Q$. Therefore $C_{C_G(Z)/Z}(Q/Z)$ is a
$3$-group and the preimage in  $C_G(Z)$ is a normal $3$-subgroup of $C_G(Z)$ and so must be
contained in $O_3(C_G(Z))=Q$. Therefore $Q \leq C_{C_G(Z)}(Q/Z)\leq Q$.

$(iii)$ Suppose $Z \nleq Q^x$. Notice that $C_Q(Y)$ normalizes $Q$ and $Q^x$ and therefore $Q \cap Q^x
\trianglelefteq C_Q(Y)$. This implies that $Q \cap Q^x=Z^x$ for if we had $Q \cap Q^x>Z^x$ then $|C_Q(Y)/(Q\cap Q^x)|\leq 9$ and so
$Z=C_Q(Y)'\leq Q \cap Q^x$. Therefore $Q^xC_Q(Y)/Q^x\cong C_Q(Y)/Z^x$ which must be non-abelian of
exponent three and order $3^3$ and so $Q^xC_Q(Y)/Q^x$ is a subgroup of $C_G(Z^x)/Q^x$ of order $3^3$ which is a contradiction.

$(iv)$ Since $[Q,Q]=Z \neq Z^x =[Q^x,Q^x]$, we immediately see that $[Q \cap Q^x,Q \cap Q^x] \leq Z\cap Z^x =1$.  Therefore $Q \cap Q^x$ is  abelian and since  $Q$ has exponent three, $Q \cap Q^x$
is elementary abelian.
\end{proof}

%%%%%%%%%%%%%%%%%%%%%%%%%%%%%%%%%%%%%%%%%%%%%%%%%%%%%%%%%%%%%%%%%%%%%%%%%%%%%%%%%%%%%%%%%%%%%%%%%%%%%%%%%%%%%%%%%%%%%%%%%%%%%%%%%%%%%%%%%%%%%

Let $t \in N_G(Z)$ be an involution such that $Qt \in \mathbf{Z}(N_G(Z)/Q)$.

\begin{lemma}\label{HN-AnotherEasyLemma}
\begin{enumerate}[$(i)$]
 \item $N_G(Z)/Q$ acts irreducibly on $Q/Z$ and in Case I, $C_G(Z)/Q\cong 2^{\cdot} \alt(5)\cong \SL(2,5)$ whilst in Case II, it has shape $2^{\cdot}\sym(5)$.

 \item $C_Q(t)=Z=C_Q(f)$ for every element of order five $f \in C_G(Z)$. In particular, in Case I, $C_G(Y)$ is a $3$-group and in Case II, $C_G(Y)$ is a $\{2,3\}$ group with $2$-part at most $2$.

 \item There exists a group $X<C_G(Z)$ with $X/Q\cong 2^{\cdot}\alt(4)\cong \SL_2(3)$ and such that $X/Q$ has  no central
 chief factors on $Q/Z$.
\end{enumerate}
\end{lemma}
\begin{proof}
It is clear that $N_G(Z)/Q$ acts irreducibly on $Q/Z$ (for example, from the fact that $5 \nmid |\GL_3(3)|$) and so  has no non-trivial modules of dimension less than four over
$\GF(3)$. Since $Qt \in \mathbf{Z}(N_G(Z)/Q)$, it follows from coprime action that $C_{Q/Z}(Qt)=C_Q(t)/Z$. Hence $C_Q(t)$ is a normal subgroup $N_G(Z)$ which is not equal to $Q$ since $C_G(Q)\leq Q$. Since $C_Q(t)/Z$ is a proper $N_G(Z)/Q$-submodule of $Q$, we have $C_Q(t)=Z$ and $[Q,t]=Q$. In particular notice that this means that the normal subgroup of order four in $N_G(Z)/Q$ does not centralize $Z$. Thus $C_G(Z)/Q\cong 2^{\cdot} \alt(5)\cong \SL(2,5)$ or has shape $2^{\cdot}\sym(5)$ which proves $(i)$. Now, for $f \in N_G(Z)$ of order five, by coprime action,
$Q/Z=C_{Q/Z}(f) \times [Q/Z,f]$ and $C_{Q/Z}(f)=C_Q(f)/Z\neq Q/Z$. Since $f$ acts fixed-point-freely on $[Q/Z,f]$,  $[Q/Z,f]^\#$ has
order a multiple of five. Therefore $Q/Z=[Q/Z,f]$ and so $C_Q(f)=Z$.

Now $C_G(Y) \leq C_G(Z)$ and $C_G(Y)$ contains no involution or element of order five. In the case that $N_G(Z)/Q\cong 4^{\cdot} \alt(5)$ (and so $C_G(Z)/Q\cong 2^{\cdot} \alt(5)$) since the Sylow $2$-subgroups are quaternion, we see that $C_G(Y)$ is a $3$-group. Otherwise, $Y$ is centralized by a $2$-group of order at most $2$.  This proves $(ii)$.

$(iii)$ Observe (using \cite[33.15, p170]{Aschbacher} for example) that a group of shape $2^{.}\alt(5)$ is uniquely defined up to isomorphism and that in either case $C_G(Z)/Q$ has such a subgroup which necessarily contains $Qt$. Moreover, $2^{.}\alt(5)$ has Sylow $2$-subgroups isomorphic to $\mathrm{Q}(8)$ with normalizer
isomorphic to $\SL_2(3)$. Let $X\leq C_G(Z)$ be such a subgroup. There can be no central chief factor of $X$ on $Q/Z$ because $Qt\in \mathbf{Z}(X/Q)$ inverts $Q/Z$.
\end{proof}

%For the rest of this section we fix a subgroup $X$ of $C_G(Z)$ such that $S<X<C_G(Z)$ and $X/Q\cong \SL_2(3)$.

\begin{lemma}\label{facts about W}
\begin{enumerate}[$(i)$]
\item $W:=C_Q(Y)C_{Q^x}(Y)=O_3(C_G(Y))$ is a group of order $3^5$ and there exists $S\in \syl_3(N_G(Z))$ such that $Y\vartriangleleft S$.

\item $L:=\<Q,Q^x\>\leq N_G(Y)$, $W=C_L(Y)$, $L/W \cong \SL_2(3)$ and $N_G(Y)/C_G(Y) \cong \GL_2(3)$.

\item $Y$ and $W/(Q \cap Q^x)$ are natural $L/C_L(Y)$-modules.

\item $\mathbf{Z}(W)=Y$.

\item $W$ has exponent three.
\end{enumerate}
\end{lemma}
\begin{proof}
$(i)$ Notice that $C_Q(Y)$ normalizes $C_{Q^x}(Y)=Q^x \cap C_G(Z)$ and so $W$ is a $3$-group. Since $Z$ is the centre  of a Sylow $3$-subgroup of $G$ we clearly have that $|W|<3^6$. Now, $Q$ is extraspecial and so $|C_Q(Y)|=3^4$ and since $Y=ZZ^x\leq Q^x$, we similarly have $|C_{Q^x}(Y)|=3^4$. Of course, $C_{Q^x}(Y) \neq C_{Q}(Y)$ as their derived subgroups are not equal and so we have that $|W|=3^5$. Moreover, $|C_G(Y):W| \leq 2$ and so $W=O_3(C_G(Y))$ and since $Q\nleq W$ we have that $Y\lhd QW\in \syl_3(N_G(Z))$.

$(ii)$ We have that $L:=\<Q,Q^x\>\leq N_G(Y)$ and $L/C_L(Y)$ is isomorphic to a subgroup of $\GL_2(3)$. Since $L$ is generated by two distinct three subgroups, $Q$ and $Q$, it follows that $L/C_L(Y)$ is generated by two distinct subgroups of order three and so $L/C_L(Y)\cong \SL_2(3)$. Also, we have seen that $C_Q(t)=Z$, therefore $t$ inverts $Y/Z$ and of course centralizes $Z$. Hence $\<L,t\>/C_L(Y)\cong \GL_2(3)$ and in particular $N_G(Y)/C_G(Y) \cong \GL_2(3)$. If $W\neq C_L(Y)$ then $|C_L(Y)/W|=2$. Notice that $L/W$ has a normal, non-abelian Sylow $2$-subgroup, $P$ say, of order $2^4$ and that any smaller normal $2$-subgroup of $L/W$ must centralizes $QW/W$ (else they would together generate $L/W$). It follows therefore that $P$ is special with centre of order four however there can be no such group. Thus $W=C_L(Y)$.

$(iii)$ We clearly have that $Y$ is a natural $L/W$-module. Now suppose that $Y=Q \cap Q^x=C_Q(Y) \cap C_{Q^x}(Y)$. Then $|W|=3^43^4/3^2=3^6$ which is a contradiction hence $Y<Q \cap Q^x$ and since $Q \cap Q^x$ is elementary abelian, $|Q \cap Q^x|=3^3$ and so $|W/(Q \cap Q^x)|=9$. Notice that $Q \cap Q^x$ is normalized by $L$ and that $L/W$ acts on $W/(Q \cap Q^x)$, which is the direct product of the groups $C_Q(Y)/(Q \cap Q^x)$ and $C_{Q^x}(Y)/(Q \cap Q^x)$. Therefore it must be a natural $L/W$-module.

$(iv)$ Clearly $Y \leq \mathbf{Z}(W)$ so suppose $Y<\mathbf{Z}(W)$ then $\mathbf{Z}(W)$ has index at most nine in $W$. Since the non-abelian group
$C_Q(Y)$ is contained in $W$, $W$ is non-abelian. Therefore $[W:\mathbf{Z}(W)]=9$. Notice
that  $\mathbf{Z}(W)\neq Q \cap Q^x$ otherwise $C_Q(Y)$ is abelian. So we have that $(Q\cap Q^x )\mathbf{Z}(W)/(Q
\cap Q^x)$ is a proper and non-trivial $L/W$ invariant subgroup of the natural $L/W$-module, $W/(Q
\cap Q^x)$. This is a contradiction. Thus $Y=\mathbf{Z}(W)$.

$(v)$ Since $Q$ has exponent three, $Q \cap Q^x$ does also. Choose $a \in Q  \bs (Q \cap Q^x)$ then
$(Q \cap Q^x) a$ is a non-identity element of the natural $N_G(Y)/W$-module, $W/Q \cap Q^x$. Moreover,
every element in the coset has order dividing three since $Q$ has exponent three. Since $N_G(Y)/W$ is
transitive on the non-identity elements of the natural module $W/(Q \cap Q^x)$, every element of
$W$ has order dividing three.
\end{proof}

We continue to use the notation $W$ and $L$ as in the previous lemma. Furthermore, let $s$ be an involution such that $Ws \in \mathbf{Z}(L/W)$ then $s$ inverts $Y$ and so $s \in N_G(Z)$. Since $t \in N_G(Y)$ we may choose $s$ such that $s$ and $t$ are in a Sylow $2$-subgroup of $N_G(Y)$ and therefore $[s,t]=1$. Furthermore set  $J:=[W,s]$ and we also now set $S:=QW \in \syl_3(C_G(Z)) \cap \syl_3(L)$ and let $Z_2$ be the second centre of $S$ (so $Z_2/Z=\mathbf{Z}(S/Z)$).

\begin{lemma}\label{lemma structure of Y and W}
\begin{enumerate}[$(i)$]
 \item We have that $Y \leq Z_2\leq C_G(Y)$ and $Z_2$ is abelian of order $3^3$ but distinct from $Q \cap Q^x$.
 %%%%%%%%%%%%%%%%%%%%%%%%%%%%%%%%%%%%%%%%%%%%%%%%%%%%%%%%%%%%%%%%%%%%%%%%%%%%%%%%%%%%%%%%%%%%%%%%%%%%%%%%%%%%%%%%%%%%%%%%%%%%%%%%%%%%%%%%?????????
 \item $W'=Y$.
 \item $Q \cap Q^x=Y C_{C_S(Y)}(s)$ and $C_{C_S(Y)}(s)\lhd C_L(s)\cong 3. \SL_2(3)$.
 \item $J$ is an elementary abelian subgroup of $C_G(Y)$ of order $3^4$ that is inverted by $s$ and $Q \cap J=Z_2$.
 \item $J=J(S)=J(W)$ and $Y \leq S' \leq Q \cap J$.
 \item $t$ and $st$ are conjugate in $N_G(Y)$.
\end{enumerate}
\end{lemma}
\begin{proof}
$(i)$  By Lemma \ref{Prelims-EasyLemma} $(ii)$, $C_G(Z)/Q$ acts faithfully on $Q/Z$. Therefore $S/Q$ is isomorphic to a cyclic subgroup of $\GL(Q/Z)\cong \GL_4(3)$ of order three, we may consider the Jordan blocks of elements of order three to see that any such cyclic subgroup centralizes a subgroup of $Q/Z$ of order at least $3^2$. Therefore $|Z_2/Z|\geq 9$ and so $|Z_2|\geq 27$.  Since $[Q/Z,Z_2]=1$, $Z_2 \leq Q$ by Lemma \ref{Prelims-EasyLemma} $(ii)$. Now suppose $Z_2 \nleq C_G(Y)$. Then $S=Z_2C_G(Y) \in \syl_3(G)$.
Since $Z_2/Z=\mathbf{Z}(S/Z)$, $[S,Z_2] \leq Z$ and so $[W,Z_2] \leq Z \leq Q \cap Q^x$. Therefore $[W/(Q
\cap Q^x), Z_2]=1$, however this implies that $S/W$ acts trivially on the natural
$L/W$-module $W/(Q \cap Q^x)$ which is a contradiction. So $Z_2 \leq C_G(Y)\cap Q$. Suppose  $Q \cap Q^x
\leq Z_2$. Then $Z^x=[C_{Q^x}(Y), Q \cap Q^x]\leq [C_{Q^x}(Y),Z_2]\leq Z$ which is a contradiction. Therefore $|Z_2|=27$
and in particular, $Z_2 \neq Q \cap Q^x$. Furthermore $Y\vartriangleleft S$ and so $Y/Z$ is
central in $S/Z$. Therefore $Y \leq Z_2$ and since $Y$ is central in $Z_2 \leq C_G(Y)$, $Z_2$ is
abelian.

$(ii)$ Now $Z=C_Q(Y)'\leq W'$ and $Z^x=C_{Q^x}(Y)'\leq W'$ and so $Y \leq W'$. Moreover, we have
just observed that $Z_2 \neq Q \cap Q^x$ and so  $Q \cap Q^x$ and $Z_2$ are distinct  normal
subgroups of $W$ both of index nine. Thus $Y \leq W' \leq Q \cap Q^x \cap Z_2$. It follows from the
group orders that $Y=C_G(Y)'=Q \cap Q^x \cap Z_2$.

$(iii)$ By coprime action on an abelian group, $W/Y=C_{W/Y}(s) \times [W/Y,s]$. By Lemma \ref{facts
about W}  $(ii)$, $Y$ and $W/(Q \cap Q^x)$ are natural $L/W$-modules. Therefore $s$ inverts $Y$ and
$W/(Q \cap Q^x)$. It follows from coprime action that $|C_{W}(s)|=3$ and that $C_{W}(s)\leq Q \cap Q^x$
with $Q \cap Q^x=YC_{W}(s)$. Furthermore, $L/W\cong \SL_2(3)$ and $Ws\in \mathbf{Z}(L/W)$ so it follows from coprime action that $C_{W}(s)\lhd C_L(s)\sim 3.\SL_2(3)$.

$(iv)$ We have that $Y$ is inverted by $s$ and so $Y=[Y,s] \leq J$. We also have that $|[W/Y,s]|=9$ and by coprime action we see, $[W/Y,s]=Y[W,s]/Y=YJ/Y \cong J/(J \cap Y)=J/Y$.  This implies that $J$ has order $3^4$ and furthermore we have that $1=J\cap C_{W}(s)$. This implies that $s$ acts fixed-point-freely on $J$ and so $J$  is abelian and inverted
by $s$. By Lemma \ref{facts about W} $(v)$, $W$ has exponent three and so $J$ is elementary
abelian. Now $Z_2$ is a characteristic subgroup of $S$ and so is normalized by $s$. Moreover $Y \leq Z_2$. If $C_{Z_2}(s)\neq 1$ then $Z_2=Q \cap Q^x=YC_{W}(s)$ which is a contradiction. Thus $Z_2=J \cap Q$.

%Observe that the involution $s$ normalizes $S$. Now $S\leq L=\<Q,Q^x\>$ and so $S$ normalizes
%$W\<s\>$ and therefore $[S,s] \leq S \cap W\<s\>=W$. So by coprime action, $[S,s]=[S,s,s]\leq
%[W,s]=J$. Since $s$ normalizes $Z$ and  $S$, we must have that $s$ normalizes $Z_2$. Moreover, $Z_2
%\leq W$ and so if $C_{Z_2}(s)$ is non-trivial then $C_{Z_2}(s)=C_W(s)\leq Q \cap Q^x$ and then
%$Z_2=YC_W(s)=Q \cap Q^x$. However by $(i)$, $Z_2 \neq Q \cap Q^x$. Thus $s$ acts fixed-point-freely
%on $Z_2$ and so $Z_2\leq J$. Since $J \nleq Q$, we have $Z_2=Q \cap J$.

$(v)$ Suppose there was another abelian subgroup of $W$ of order $3^4$, $J_0$ say. Then $|J \cap
J_0|=3^3$ and  $J \cap J_0$ is central in $W$. This contradicts Lemma \ref{facts about W}
which says that $\mathbf{Z}(W)=Y$. It follows therefore that $J(W)=J$. Clearly $3^4$ is the largest possible order of an abelian subgroup of $S$ (else $Q$ would contain
abelian subgroups of order $3^4$). So suppose $J_1$ is an abelian subgroup of $S$ distinct from
$J$. Then $J_1 \nleq W$ and $J_1 \nleq Q$. Therefore, $S/Z$ contains three distinct abelian
subgroups  $Q/Z$, $J/Z$ and $J_1/Z$. We must have that $S=QJ=QJ_1$. Hence,  $(Q/Z) \cap (J/Z)$ and
$(Q/Z) \cap (J_1/Z)$ both have order nine and are both central in $S/Z$. We must have that $Q/Z
\cap J/Z =Q/Z \cap J_1/Z=Z_2/Z$. Thus $Y \leq Z_2 \leq J_1$ and so $J_1\leq C_S(Y)$ which we have seen is not possible.  Thus $J=J(S)$. In
particular, $J$ is a normal subgroup of $S$ of index nine and so $Y=W' \leq S' \leq Q \cap J$.

$(vi)$ Finally, we have seen that $L\<t\>/W\cong \GL_2(3)$ and so $Wt$ is conjugate to $Wst$ and since $W$ is a $3$-group, an element of $L$ conjugates $t$ to $st$.
\end{proof}

We now choose a subgroup $S \leq X \leq C_G(Z)$ such that $X/Q\cong \SL_2(3)$ as in Lemma \ref{HN-AnotherEasyLemma}.

\begin{lemma}
$Q/Z=\<C_{Q/Z}(S)^{X/Q}\>$ and $S/Q$ acts quadratically on $Q/Z$.
\end{lemma}
\begin{proof}
First observe that since $X/Q\cong \SL_2(3)$ and there is no central chief factor of $X/Q$ on
$Q/Z$, any  proper $X/Q$-submodule of $Q/Z$ is necessarily a natural $X/Q$-module. Let $Z < V< Q$
such that $V/Z$ is an $X/Q$-submodule and is therefore a natural module. Thus $S/Q$ acts
non-trivially on $V/Z$. In particular this means $V/Z\neq \mathbf{Z}(S/Z)=C_{Q/Z}(S)$. So $\mathbf{Z}(S/Z)$ is not contained in any
proper $X$-invariant subgroup of $Q$. Thus $Q=\<C_{Q/Z}(S)^{X/Q}\>$.

Now, $J=J(S)$ is abelian and normalized by $S$ and so $[Q,J]\leq J$ and then $[Q,J,J]=1$. Now $J \nleq Q$
(as $Q$ has no abelian subgroups of order $3^4$) and so $S/Q=JQ/Q$ and therefore
$[Q/Z,JQ/Q,JQ/Q]=1$ and so $S/Q$ acts quadratically on $Q/Z$.
\end{proof}

We have now satisfied the conditions of Lemma \ref{Parker-Rowley-SL2(q)-splitting} and so  we have
the following results.
\begin{lemma}\label{Lemma N1-N4}
\begin{enumerate}[$(i)$]
\item $Q/Z$ is a direct product of natural $X/Q$-modules.
 \item There are exactly four $X$-invariant
subgroups $N_1,N_2,N_3,N_4< Q$ properly containing $Z$ such that for $i \neq j$, $N_i \cap N_j=Z$.
 \item $N_i
\cap J$ has order nine for each $i$ and $S'=J \cap Q=\<N_i \cap J|1\leq i \leq 4\>$.
 \item For some $i \in \{1,2,3,4\}$, $Y \leq N_i$ and $N_i$ is abelian.
 \item For each $i \in \{1,2,3,4\}$, $X$ is transitive on $N_i\bs Z$.
\end{enumerate}
\end{lemma}
\begin{proof}
Part $(i)$ follows immediately from Lemma \ref{Parker-Rowley-SL2(q)-splitting} which says that
$Q/Z$ is a direct product of natural $X/Q$-modules. Let $N_1$ and $N_2$ be the corresponding
subgroups of $Q$. View $N_1/Z$ and $N_2/Z$ as vector spaces over $\GF(3)$. Since $N_1/Z$ and
$N_2/Z$ are isomorphic as $X$-modules, it follows that there are two additional isomorphic submodules in $Q/Z$. Let $N_3$ and $N_4$ be the corresponding  normal subgroups of $Q$. Then  $N_3/Z$ and $N_4/Z$
are natural $X$-modules and for $i \neq j$, $N_i \cap N_j=Z$. This proves $(ii)$.

By Lemma \ref{lemma structure of Y and W}, if $Z_2/Z=\mathbf{Z}(S/Z)$ then $Y \leq Z_2$ and $Z_2=J \cap Q$ is elementary abelian of order $27$. Now for each $i\in \{1,2,3,4\}$,  $C_{N_i/Z}(S) \neq 1$ and so $Z_2 \cap N_i=J \cap N_i$ has order at least nine.
In fact the order must be exactly nine for were it greater then for some $i$, $N_i=Z_2$ and then
$N_i \cap N_j$ would have order at least nine for each $j \neq i$. Now for each  $i \neq j$, $N_i
\cap N_j=Z$ and so $N_i \cap J \neq N_j \cap J$ and so $Z_2=\<N_i \cap J|1\leq i \leq 4\>$.
In particular we must have (without loss of generality) that $N_1\cap J=Y$. By Lemma \ref{lemma structure of Y and W} $(v)$, $Y\leq S' \leq Q \cap J$. Suppose $S'=Y$. Then for any $2\leq i \leq 4$, $Y \nleq N_i$
and so $[N_i,S] \leq N_i \cap Y =Z$. Therefore $N_i \leq Z_2$ which is a contradiction. Thus
$Y<S'=J \cap Q$  which proves $(iii)$.

We already have that (without loss of generality) $N_1\cap J=Y$. Suppose that $N_1$ is
non-abelian. Then $C_Q(N_1)\cong N_1\cong 3_+^{1+2}$. Since $N_1$ is $X$-invariant, $C_Q(N_1)$ is
also $X$-invariant. Therefore, without loss of generality, we can assume that $N_2=C_Q(N_1)$ and so
$N_2\leq C_S(Y)$ and $N_2$ is also non-abelian with $Y \nleq N_2$.  Since $N_1 \nleq C_S(Y)$, $S=C_S(Y)N_1$
and so  we have that $S'\leq [C_S(Y),N_1] C_S(Y)' N_1' \leq (C_S(Y)
\cap N_1)Y Z=Y$ (using Lemma \ref{lemma structure of Y and W})  which is a contradiction since $S'>Y$. This proves $(iv)$.

Finally, since each $N_i/Z$ is a natural $X/Q$-module, $X$ is transitive on the non-identity elements
of $N_i/Z$. So let $Z \neq Zn \in N_i/Z$. Then $\<Z,n\>\vartriangleleft Q$ however $|C_Q(n)|=3^4$.
Therefore $n$ lies in a $Q$-orbit of length three in $Zn$. Hence every element in $Zn$ is
conjugate in $X$. Thus $X$ is transitive on $N_i \bs Z$ which completes the proof.
\end{proof}

For the rest of this section we continue the notation from Lemma \ref{Lemma N1-N4} with
$N_1,N_2,N_3,N_4$  chosen such that $Y<N_1$ and satisfying the notation set in the following lemma
also.

\begin{lemma}\label{N_1, N_2 abelian, N_3, N_4 not}
Without loss of generality we may assume that $N_1\cong N_2$ is elementary abelian and $N_3\cong N_4$ is extraspecial with $[N_3,N_4]=1$.
\end{lemma}
\begin{proof}
By Lemma \ref{Lemma N1-N4}, $N_1$ is abelian. So suppose $N_i$ is non-abelian for some $i\in
\{2,3,4\}$.  Then $C_Q(N_i)\cong N_i\cong 3_+^{1+2}$ is $X$-invariant and we may assume
$C_Q(N_i)=N_j$ for some $i \neq j\in \{2,3,4\}$. Now it follows that either $N_i$ is abelian for
every $i\in \{1,2,3,4\}$ or without loss of generality $N_1\cong N_2$  and $N_3\cong N_4$ are non-abelian.
So we assume for a contradiction that $N_2$, $N_3$ and $N_4$ are all abelian.

Since $N_1/Z$ is isomorphic as a $\GF(3)X/Q$-module to $N_2/Z$, for any $m \in N_1\bs Z$ there
is an $n \in N_2 \bs Z$ such that $Zn$ is the image of $Zm$ under a module isomorphism. It then follows
(without loss of generality) that $Znm$ is an element of $N_3/Z$ and $Zn^2m$ is an element of
$N_4/Z$. In particular $x_1:=nm \in N_3$ and $x_2:=n^2m \in N_4$. Let $g \in X$ have order
four then $Qg^2=Qt$ inverts $Q/Z$ and so
\begin{equation}\label{one}Zn^{g^2}=Zn^2 ~\mathrm{and}~ Zm^{g^2}=Zm^2.\end{equation}
Also if $Z\neq Za\in N_i/Z$ and $g$ and $h$ are elements of order four in $X$ such that
$Q\<g\>\neq Q\<h\>$ then $N_i/Z=\<Za^g,Za^h\>$ and so $N_i=Z\<a^g,a^h\>$.

So consider $[x_1,x_2^g]$. We calculate the following using commutator relations and using that all
commutators are in $Z$ and therefore central.

\[\begin{array}{rclr}
    [x_1,x_2^g]&=&[nm,{(n^2)}^g
    m^g]&\:\\
    \;&=&[n,m^g][m,m^g][n,{(n^2)}^g][m,{(n^2)}^g]&\;\\
    \;&=&[n,m^g][m,{(n^2)}^g]& (\mathrm{since}~N_1~\mathrm{and}~N_2~ \mathrm{are~abelian})\\
    \;&=&[n,m^g][m,{n}^g]^2&\;\\
    \;&=&([n,m^g][m,{n}^g]^2)^g&(\mathrm{since~commutators~are~central~in~}X)\\
    \;&=&[n^g,m^{g^2}][m^g,n^{g^2}]^2&\;\\
    \;&=&[n^g,m^2][m^g,n^2]^2&(\mathrm{by~Equation~}\ref{one})\\
    \;&=&[n^g,m]^2[m^g,n]&\;\\
    \;&=&[{(n^2)}^g,m][m^g,n]&\;\\
   \;&=&[{(n^2)}^g,m][m^g,m][{(n^2)}^g,n][m^g,n]&(\mathrm{since}~N_1~\mathrm{and}~N_2~ \mathrm{are~abelian})\\
    \;&=&[{(n^2)}^g m^g,nm]&\;\\
    \;&=&[x_2^g,x_1].&\;\\
\end{array}\]Thus $[x_1,x_2^g]=[x_1,x_2^g]\inv$ and so
$[x_1,x_2^g]=1$. This holds for any element of order four in $X$. Thus $mn\in N_3$
commutes with $N_4=Z\<(n^2m)^g,(n^2m)^h\>$ where $g$ and $h$ are elements of order four as above.
Furthermore this argument works for any element of $N_3\bs Z$ and so $[N_3,N_4]=1$. However this
contradicts our assumption that $N_3$ and $N_4$ are abelian.
\end{proof}

\begin{lemma}\label{HN-a new class of three}
For $i\in \{3,4\}$, elements in $N_i\bs Z$ are not conjugate into $Z$. In particular,  there are 12
elements of order three in $S'$ which are not $G$-conjugate into $Z$.
\end{lemma}
\begin{proof}
Let $\{i,j\}=\{3,4\}$ and let $a\in N_i \bs Z$. Since $a\in Z_2$, $|C_S(a)|=3^5$.  Suppose that $a \in Z^G$. Then we must have that $C_S(a)=O_3(C_G(\<a,Z\>))$ and we must similarly have that $\<a,Z\>=C_S(a)'$.
Now $S=C_S(a)N_j$ and so $S' \leq C_S(a)' N_i' (C_S(a)\cap N_i)\leq \<a,Z\>$ which is a contradiction. Thus $a \notin Z^G$.
Every element in
$N_i \bs Z$ is conjugate to $a$ and therefore no element in $N_i \bs Z$ is conjugate into $Z$.

Furthermore,
by Lemma \ref{Lemma N1-N4} $(iii)$, we see that $S'=J \cap Q$ contains twelve  elements of order three
which are not conjugate into $Z$. These are contained in $N_3 \cap J=N_3 \cap S'$ and $N_4 \cap J=N_4 \cap S'$.
\end{proof}

\begin{lemma}\label{C-i's-orders and derived subgroups}
\begin{enumerate}[$(i)$]
\item Let $i \in \{1,2,3,4\}$ and set $S_i:=C_S(J \cap N_i)$ then $|S_i|=3^5$ and $|\mathbf{Z}(S_i)|=9$.
\item $S_1'=\mathbf{Z}(S_1)=J\cap N_1=Y$, $S_2'=\mathbf{Z}(S_2)=J \cap N_2$, $S_3'=\mathbf{Z}(S_4)=J
\cap N_4$ and $S_4'=\mathbf{Z}(S_3)=J \cap N_3$.
\end{enumerate} In
particular $S_i \neq S_j$ for each $i \neq j$.
\end{lemma}
\begin{proof}
By Lemma \ref{Lemma N1-N4}, $|J \cap N_i|=9$ for each $i\in\{1,2,3,4\}$ and since $J$ is elementary abelian of order $81$, $J
\leq S_i$. Hence $S_i\geq \<J,C_Q(J \cap N_i)\>$. Since $C_Q(J \cap N_i)$ has order $3^4$ and
is non-abelian, $|\<J,C_Q(J \cap N_i)\>|\geq 3^5$. Moreover, since $|S|=3^6$ and
$\mathbf{Z}(S)=Z$ has order three, it follows that $S_i=\<J,C_Q(J \cap N_i)\>$ has order $3^5$.
Now for each $i \in \{1,2,3,4\}$, $\mathbf{Z}(S_i)\leq Q$ otherwise $S=Q\mathbf{Z}(S_i)$ and then $[S_i \cap Q,S]\leq [S_i \cap Q,Q][S_i \cap Q,\mathbf{Z}(S_i)]\leq Z$ which implies that $S_i \cap Q\leq Z_2$ which is a contradiction. Thus $\mathbf{Z}(S_i)\leq Q$ has order nine and $\mathbf{Z}(S_i)=N_i \cap J$.

Now for $i \in \{1,2,3,4\}$, we have that $Z \leq S_i'$. If $S_i'=Z$ then $Q/Z$ and $S_i/Z$ are two
distinct abelian subgroups of $S/Z$ of index three. This implies that $S/Z$ has centre of order at
least $3^3$. However by Lemma \ref{lemma structure of Y and W} $(i)$,  $\mathbf{Z}(S/Z)$ has order
nine. Thus $S_i'>Z$. Now for $i=1$, by Lemma \ref{Lemma N1-N4}, $Y= N_1\cap J$ and so
$\mathbf{Z}(S_1)=N_1 \cap J=Y$. Furthermore, for $i \in \{1,2\}$, $N_i$ is abelian and so $N_i
\leq S_i$. Therefore $S_i' \leq S' \cap N_i= J\cap N_i$ since $N_i\vartriangleleft S_i$. For
$\{i,j\} = \{3,4\}$, $[N_i,N_j]=1$ and so $N_j \leq S_i$. Therefore $S_i' \leq S' \cap N_j= J
\cap N_j$ since $N_j\vartriangleleft S_i$.
\end{proof}

\begin{lemma}\label{elements of order nine in S}
Every element of order three in $S$ lies in the set $Q \cup S_1 \cup S_2$ and the cube of every
element of order nine in $S$ is in $Z$.
\end{lemma}
\begin{proof}
By hypothesis, $Q$ has exponent three and by Lemma \ref{lemma structure of Y and W} $(vi)$, so does $J$.
So let $ g\in S $  such that $g \notin Q \cup J$. Then $g=cb$ for some $ c \in Q \bs J=Q \bs S'$ and some $b \in J \bs Q$. We calculate using the equality $c[b,c][b,c,c]=[b,c]c$ and using that $b \in J$ so
commutes with all commutators in $S'\leq J$.\begin{eqnarray*}
cbcbcb&=&c^2b[b,c]bcb \\
&=&c^2b^2[b,c]cb \\
&=&c^2b^2c[b,c][b,c,c]b \\
&=&c^2b^2cb[b,c][b,c,c] \\
&=&[c,b][b,c][b,c,c] \\
&=&[b,c,c].
\end{eqnarray*}

Since $c \in Q\bs J=Q \bs S'$, $S'\<c\>$ is a proper subgroup of $Q$ properly containing $S'$.
As $S'\cap N_i=J \cap N_i$ has order nine for each $i\in \{1,2,3,4\}$, $S'N_i$ has order $81$. Thus $S'\<c\>=S'N_i$ for
some $i\in \{1,2,3,4\}$.

If $S'\<c\>=S'N_1=C_Q(Y)$ then $cb\in C_S(Y)=W$ and $W$ has exponent three. Suppose $S'\<c\>=S'N_2$. Then
$S_2=C_S(S' \cap N_2)=J\<c\>$ and $S_2'=S' \cap N_2$ therefore
$[b,c]\in S' \cap N_2$ is central in $S'\<c\>=S'N_2$. Therefore $[b,c,c]=1$ and so $cb$ has
order three.

Now suppose $S'\<c\>=S'N_3$ (and a similar argument holds if $S'\<c\>=S'N_4$). Then
$S_4=C_S(S' \cap N_4)=J\<c\>$ and $[b,c] \in S_4'=S' \cap N_3$. Suppose $cbcbcb=[b,c,c]=1$. Then
$[b,c]$ commutes with $J\<c\>=S_4$ and so $[b,c] \in S_4' \cap \mathbf{Z}(S_4)=Z$. Thus
$S_4=J\<c\>=S'\<b,c\>$ and so $[S_4,S_4]=\<[S',c],[S',b],[c,b]\>$. However $[S',c]\leq Z$,
$[S',b]=1$  and $[c,b]\in Z$ which is a contradiction since $[S_4,S_4]=N_3\cap S'>Z$. Thus
$[b,c,c]\neq 1$ and $cb$ has order nine (no element can order $27$ since $Q$ has exponent three).
Furthermore, $(cb)^3=[b,c,c] \in [S' \cap N_4,c] \leq [Q,Q]=Z$ and so the cube of every such
element of order nine is in $Z$.
\end{proof}

\begin{lemma}\label{prelims-Centre of the C_i's}
For each $i\in\{3,4\}$, if $a \in \mathbf{Z}(S_i) \bs Z$ then $\mathbf{Z}(S_i/\<a\>)=\mathbf{Z}(S_i)/\<a\>$.
\end{lemma}
\begin{proof}
Let $\{i,j\}= \{3,4\}$ then by Lemma
\ref{C-i's-orders and derived subgroups}, we have that $S_i'=\mathbf{Z}(S_j)$ and $S_j'=\mathbf{Z}(S_i)$. So let $a \in \mathbf{Z}(S_i) \bs Z$ and suppose
$\mathbf{Z}(S_i/\<a\>)>\mathbf{Z}(S_i)/\<a\>$. Let $V\leq S_i$ such that $a \in V$ and
$\mathbf{Z}(S_i/\<a\>)=V/\<a\>$ then $|V|\geq 3^3$.  Therefore $S_i/V$ is abelian and so $S_i'\leq
V$. Therefore $[S_i',S_i] \leq \<a\>$. However $S_i$ normalizes $\mathbf{Z}(S_j)=S_i'$ and so
$[S_i',S_i] \leq \<a\>\cap S_i'= \<a\>\cap \mathbf{Z}(S_j)=1$ since $\mathbf{Z}(S_i) \cap
\mathbf{Z}(S_j)\leq N_i \cap N_j=Z$. However this implies that $S_i' \leq \mathbf{Z}(S_i)$ and so $N_j \cap
J\leq N_i \cap J$ which is a contradiction. Therefore
$\mathbf{Z}(S_i/\<a\>)=\mathbf{Z}(S_i)/\<a\>$.
\end{proof}

We fix an element of order three $a$ in $Q$ such that $a \in (N_3 \cap J) \bs Z$ and therefore $a
\notin Z^G$ by Lemma \ref{HN-a new class of three}. Let $3\mathcal{A}:=\{a^g|g \in G\}$ and
$3\mathcal{B}:=\{z^g|g \in G\}$ where $Z=\<z\>$. We show in the rest of this section that these are the only
conjugacy classes of elements of order three in $G$.

\begin{lemma}\label{HN-prelims1}
\begin{enumerate}[$(i)$]
\item $|C_S(a)|=3^5$;
\item $|a^{G} \cap Q|=|a^{C_G(Z)} \cap Q|=120$ and $|z^G \cap Q|=|z^{xC_G(Z)} \cap
 Q|+2=122$; and
\item $Q^\# \subset 3\mathcal{A} \cup 3\mathcal{B}$.
\end{enumerate}
Furthermore, in Case II, $C_G(Z)/Q$ is isomorphic to the group of shape $2^{\cdot}\sym(5)$ which has semi-dihedral Sylow $2$-subgroups and in either case $\<s,C_G(Z)'\>/Q\sim  4^{\cdot}\alt(5)$.
\end{lemma}
\begin{proof}
We have chosen $a \in N_3\cap J$ and so by Lemma \ref{C-i's-orders and derived subgroups},
$C_S(a)=C_S(\<Z,a\>)=C_S(N_3 \cap J)=S_3$ which has order $3^5$. Now let $q \in Q\bs Z$ and
consider $[{C_G(Z)}:C_{C_G(Z)}(q)]$. By Lemma \ref{HN-AnotherEasyLemma} $(ii)$, an element of order five
acts fixed-point-freely on $Q/Z$ so we have that $5\mid [{C_G(Z)}:C_{C_G(Z)}(q)]$. Let $R$ be a Sylow $2$-subgroup of $C_{C_G(Z)}(q)$. Recall that $C_G(Z)/Q$
has subgroups isomorphic to $\mathrm{Q}(8)$ with $Qt$ in the centre. The preimage of any such subgroup in $C_G(Z)$ intersects trivially with $R$ as $Qt$ inverts $Q/Z$.  So $8\mid [{C_G(Z)}:C_{C_G(Z)}(q)]$. Furthermore $q$ is not $3$-central in $C_G(Z)$
and so $3\mid [{C_G(Z)}:C_{C_G(Z)}(q)]$. Therefore $[{C_G(Z)}:C_{C_G(Z)}(q)]$ is a multiple of 120.
Now there exists $z^x \in Q \bs Z$  which lies in a ${C_G(Z)}$-orbit in $Q$ of length at least 120
and also there exists $a \in Q$ which is not conjugate to $z$ and lies in a ${C_G(Z)}$-orbit in $Q$
of length at least 120. Since $a$ is not conjugate to $z^x$, these orbits are distinct. Thus
$|a^{G} \cap
 Q|=|a^{C_G(Z)} \cap Q|=120$ and $|z^G \cap Q|=|z^{xC_G(Z)} \cap Q|+2=122$.

We may now observe that in Case II, when $N_G(Z)/Q\sim 4^{\cdot}\sym(5)$, every subgroup of index two must contain an involution centralizing $z^x \in Q\bs Z$. In particular, $t$ can not be the unique involution in any such index two subgroup. It follows then that $C_G(Z)/Q$ is isomorphic to the group of shape $2^{\cdot}\sym(5)$ which has semi-dihedral Sylow $2$-subgroups as claimed.

The final comment in the statement of this lemma is clear in Case I so suppose we are in Case II. Recall that $Js$ lies in the centre of $N_G(J)/J$ and that $s$ centralizes $S/J$. We have that $\<s,C_G(Z)'\>/Q$ has shape $4^{\cdot}\alt(5)$ or $2^{\cdot}\sym(5)$. Thus a Sylow $2$-subgroup of the normalizer of $S$ is isomorphic to $C_4 \times C_2$ or $\dih(8)$ respectively. If $\<s,C_G(Z)'\>/Q \sim 2^{\cdot}\sym(5)$ then a Sylow $2$-subgroup of the normalizer of $S$ must act faithfully on $S/J=QJ/J\cong Q/Q \cap J$ as the centre of the dihedral group is a conjugate of $t$ and so inverts $S/J$. Thus $s$ cannot lie in such a subgroup and we may conclude that $\<s,C_G(Z)'\>/Q\sim 4^{\cdot}\alt(5)$.
\end{proof}

\begin{lemma}\label{HN-prelim-element of order four normalizing S}
\begin{enumerate}[$(i)$]
\item  $|C_J(t)|=|C_S(t)|=3^2$ and $t$ inverts $S/J$.
\item In Case I, we have that $|N_G(S) \cap C_G(Z)|=3^62^2$ and $|N_G(S)|=3^62^3$.
\item In Case II, we have that $|N_G(S) \cap C_G(Z)|=3^62^3$ and $|N_G(S)|=3^62^4$.
\item There exists an element of order four $e \in N_G(S) \cap C_G(Z)$ such that $e^2=t$ and $e$ does not
normalize $Y$.
\end{enumerate}
\end{lemma}
\begin{proof}
We have that $C_Q(t)=Z$
and so $t$ inverts $Q/Z$. We also have that $t$
centralizes $S/Q=QJ/Q\cong J/(J \cap Q)$. Since $C_Q(t)=Z$, we see using coprime action and an isomorphism theorem that
$C_{J/(J \cap Q)}(t)=C_J(t)(J \cap Q)/(J \cap Q)\cong C_J(t)/Z$ and so $|C_J(t)|=3^2$. We also see that $t$ inverts
$Q/(Q \cap J)\cong QJ/J=S/J$ which proves $(i)$.

Now, in Case I, the normalizer of a Sylow $3$-subgroup has
order $2^23$ with a cyclic Sylow $2$-subgroup and in Case II, it has order $2^33$. Recall that $s\in N_G(Y)$ inverts $Z$ and normalizes $S \leq N_G(Y)$. Thus  $(ii)$ and $(iii)$ follow immediately. Furthermore, in either case, we may choose
an element of order four $e \in C_G(Z)$ that squares to $t$ and normalizes $S$.  Suppose $e$
normalizes $Y$. Then $e^2=t$ centralizes $Y$ which is impossible. This completes the proof.
\end{proof}

\begin{lemma}\label{HN-prelims2}
\begin{enumerate}[$(i)$]
\item  $J^\# \subseteq 3\mathcal{A} \cup 3\mathcal{B}$.
\item  $N_2^\# \subseteq 3\mathcal{B}$, $|C_J(t) \cap 3\mathcal{A}|=|C_J(t) \cap 3\mathcal{A}|=4$ and $C_{C_S(Y)}(s)^\# \subseteq 3\mathcal{A}$.
\item Every element of order three in $S$ is in the set $3\mathcal{A}\cup 3\mathcal{B}$.
\item For every $q \in Q$ there exists $P \in \syl_3(C_G(Z))$ such that
$q \in J(P)$.

\end{enumerate}
\end{lemma}
\begin{proof}
$(i)$ Since $N_G(Y)/C_G(Y)\cong \GL_2(3)$ and $J$ is characteristic in $C_G(Y)$ and inverted by $C_G(Y)s$, we have that $J/Y$ is a natural $N_G(Y)/C_G(Y)$-module. Hence there are four $N_G(Y)$-images of $S'$ in $J$ with pairwise intersection equal to $Y$. By Lemma \ref{HN-prelims1}, $Q^\#\subseteq 3\mathcal{A} \cup 3\mathcal{B}$.
Therefore $[S,S]^\# \subseteq 3\mathcal{A} \cup 3\mathcal{B}$ which implies that $J\bs\{1\} \subseteq
3\mathcal{A} \cup 3\mathcal{B}$.

$(ii)$ We have that for $i \in \{1,2,3,4\}$, by Lemma \ref{Lemma N1-N4} $(v)$, $X$ is transitive on
$N_i\bs Z$ and so either $N_i\bs Z\subseteq 3\mathcal{A}$ or $N_i\bs Z \subseteq 3\mathcal{B}$. By
Lemma \ref{HN-prelim-element of order four normalizing S} $(iii)$, there exists $e\in N_G(S)$ such
that $Y^e\neq Y$. Since $e$ normalizes $S$, $e$ normalizes $Z_2=\<J \cap N_i \mid i \in \{1,2,3,4\}\>$. Therefore $Y^e=N_i \cap J$ for some $i \in \{2,3,4\}$. We have that  $N_i\bs
Z \subseteq 3\mathcal{A}$ for $i=3,4$ and so $Y^e=N_2 \cap J$. Thus $N_2^\# \subseteq 3
\mathcal{B}$. Notice now that $C_{Z_2}(st)$ has order 9 and is a complement to $Z$ in $Z_2$. Therefore  $|C_J(st) \cap 3\mathcal{A}|=|C_J(st) \cap 3\mathcal{A}|=4$. By Lemma \ref{lemma structure of Y and W} $(vi)$, $t$ is conjugate to $st$ by an element of $N_G(Y)\leq N_G(J)$ and so the same count holds for $t$.

Now there are five conjugates of $X$ in $C_G(Z)$ and therefore five images of $N_1$
and of $N_2$ in $C_G(Z)$ (since if $N_i$ was normal in two distinct conjugates of $X$ then $N_i$
would be normal in $C_G(Z)$). For each $i\in \{1,2\}$, $N_i\bs Z$ contains $24$ conjugates of $z$.
Since $Q\bs Z$ contains 120 conjugates of $Z$, there exists $i \in \{1,2\}$ and $g \in C_G(Z)$ such
that $Y\leq N_i^g\vartriangleleft X^g$ and $N_i^g\neq N_1$.  Now consider $C_Q(Y)$ which is
normalized by $s$ (as $s$ normalizes $Q$ and $Y$). By Lemma \ref{lemma structure of Y and W} $(iii)$,  $C_{C_S(Y)}(s) \leq Q \cap Q^x$ has order three. Now there are four proper subgroups
of $C_Q(Y)$ properly containing $Y$. These include $Q \cap Q^x$, $S'$, $N_1$ and $N_i^g$ (we can not yet exclude the possibility that $Q \cap Q^x=N_1$ or $N_i^g$). We have
that $s$ normalizes at least two subgroups: $S'\neq Q \cap Q^x$ (since $S'=J \cap Q$ and using \ref{lemma structure of Y and W} $(i)$ and $(iv)$). Suppose that $s$ normalizes
$N_1$ and $N_i^g$. If $s$ inverts $N_1$ then $N_1\leq [C_S(Y),s]=J$ which is a
contradiction (as $|N_1 \cap J|=9$). Therefore $N_1=YC_{C_S(Y)}(s)=Q \cap Q^x$ and by the same
argument $N_i^g =Q \cap Q^x$ which is a contradiction since $N_i^g \neq N_1$. Therefore at least
one of $N_1$ and $N_i^g$ is not normalized by $s$. We assume that $N_1^s \neq N_1$ (and the same
argument works if $N_i^{gs} \neq N_i^g$) and so the four proper subgroups
of $C_Q(Y)$ properly containing $Y$ are $Q \cap Q^x$, $S'$, $N_1$ and $N_1^s$. Now consider $|C_Q(Y) \cap 3\mathcal{A}|$. Since $Q/N_1$
is a natural $X/Q$-module, there are four $X$-conjugates of $C_Q(Y)$ in $Q$ intersecting at $N_1$.
Each must contain exactly 120/4=30 conjugates of $a$. Thus $|C_Q(Y) \cap 3 \mathcal{A}|=30$.
Clearly $N_1\cap 3 \mathcal{A}=N_1^s \cap 3\mathcal{A}=\emptyset$ and $|S'\cap 3 \mathcal{A}|=12$
by Lemma \ref{HN-a new class of three}. Therefore we have $|Q \cap Q^x \cap 3\mathcal{A}|=18$. In
particular this implies $C_{C_S(Y)}(s)^\# \subseteq 3 \mathcal{A}$.

$(iii)$ By Lemma \ref{elements of order nine in S}, every element of order three in $S$ lies in $Q
\cup C_S(N_1 \cap J) \cup C_S(N_2 \cup J)$ and the cube of every element of nine is in $Z$.
Recall that $N_1 \cap J=Y$ and since $N_2^\# \subseteq 3\mathcal{B}$ and $C_G(Z)$
is transitive on  $Q \cap 3\mathcal{B} \bs Z$,  $N_1 \cap J$ is conjugate in $C_G(Z)$ to $N_2
\cap J$. Therefore $S_2=C_S(N_2 \cap J)$ is conjugate to $C_S(Y)=S_1$. Now, by Lemma \ref{facts about W}, $C_S(Y)/(Q \cap Q^x)$ is a natural $\SL_2(3)$-module and so there are four $N_G(Y)$-conjugates of
$C_Q(Y)$ in $C_S(Y)$ and this accounts for every element of $C_S(Y)$. Since $C_Q(Y)^\# \subseteq Q^\# \subseteq
3\mathcal{A} \cup 3 \mathcal{B}$, $C_S(Y)^\# \subseteq 3\mathcal{A} \cup 3 \mathcal{B}$ and therefore
every element of order three in $S$ is in $3\mathcal{A} \cup 3 \mathcal{B}$.

$(iv)$ Since $z^x,a \in J=J(S)$ and every element in $Q\bs Z$ is ${C_G(Z)}$-conjugate
to one of these, every element in $Q$ lies in the Thompson subgroup of a Sylow $3$-subgroup of
${C_G(Z)}$.
\end{proof}

\begin{lemma}\label{HN-centralizer of a does not normalize J}
$C_G(a)\nleq N_G(J)$.
\end{lemma}
\begin{proof}
By Lemma \ref{HN-prelims2} $(ii)$ and $(iv)$,  there exists $g \in Q \cap Q^x \cap 3\mathcal{A}$
and there exists $R\in \syl_3(C_G(Z))$ such that $g \in J(R)$. The same lemma applied to $C_G(Z^x)$
says that there exists $P\in \syl_3(C_G(Z^x))$ such that $g \in J(P)$. If $Q \cap Q^x \leq J(R)$
then $Y \leq J(R)\leq C_G(Y)$. Hence $J(R)=J({C_G(Y)})=J$ (see Lemma \ref{lemma structure of Y and W} $(v)$) however
$Q \cap Q^x \nleq J$ (by Lemma \ref{lemma structure of Y and W} $(iii)$ since $1 \neq C_{C_G(Y)}(s)\leq Q \cap Q^x$ but
$J$ is inverted by $s$). Therefore $C_R(g)=J(R)(Q \cap Q^x)$ and similarly $C_P(g)=J(P)(Q \cap
Q^x)$.

Suppose $J(P)=J(R)$. Then $J(R)$ is normalized by $\<Q,Q^x\>=L$.  However $O_3(L)={C_S(Y)}$ and so $g \in J(R)=J({C_S(Y)})=J$ which is a contradiction and so $J(P)\neq
J(R)$. This implies that $C_G(g)$ has two distinct Sylow $3$-subgroups with distinct Thompson
subgroups. Since $a$ is conjugate to $g$, it follows that $C_G(a) \nleq N_G(J)$.
\end{proof}

\begin{lemma}\label{HN-element of order four in CG(Z)}
Let $ A \in \syl_2(C_G(Z))$ such that $t \in A$ and suppose that $f \in A$ such that $f^2=t$. Then $Z \in \syl_3(C_G(f))\cap
\syl_3(C_G(A))$.
\end{lemma}
\begin{proof}
We have that $C_Q(A)\leq C_Q(f)=Z$ since $f^2=t$ and $C_Q(t)=Z$. In either case for the structure of $C_G(Z)$, we have that every element of order four in $A$ (which is isomorphic to either $\mathrm{SD}_{16}$ or $\mathrm{Q}(8)$) lies in the subgroup of $A$ isomorphic to $\mathrm{Q}(8)$ which in turns lies in $O^2(C_G(Z))$. Thus it follows from the structure of $2^{\cdot}\alt(5)$ that no element of order three in $C_G(Z)/Q$ centralizes $Qf$. Therefore, by coprime action, we have that $Z \in \syl_3(C_G(f))\cap \syl_3(C_G(A))$.
\end{proof}

\begin{lemma}\label{HN-normalizer of J1}
We have that $[N_G(J):C_{N_G(J)}(a)]=48$ and $[N_G(J):C_{N_G(J)}(Z)]=32$. Furthermore, $|N_G(J)|=3^62^7$ or $3^62^8$ in Case I and II respectively.
\end{lemma}
\begin{proof}
Since $J/Y$ is a natural ${N_G(Y)}/{C_G(Y)}$-module,  $J$ contains four
$N_G(Y)$-conjugates of $S'$ with pairwise intersection $Y$. By Lemma \ref{Lemma N1-N4}, $S'=\<N_i \cap
S'\mid 1 \leq i \leq 4\>$. Since the conjugates of $z$ lie in $N_1 \cup S'$ and $N_2 \cup S'$,
$|S' \cap 3\mathcal{B}|=8+6=14$ and so $|J\cap 3\mathcal{B}|=8+(4*6)=32$. Therefore, by Lemma
\ref{conjugation in thompson subgroup},  $[N_G(J):C_{N_G(J)}(z)]=32$.  Now by Lemma
\ref{HN-prelim-element of order four normalizing S},  $|C_{N_G(S)}(z)|=3^62^2$ in Case I and $|C_{N_G(S)}(z)|=3^62^3$ in Case II. Note that $C_{N_G(S)}(z)\leq C_{N_G(J)}(z)\leq C_G(z) \cap N_G(QJ)=C_{N_G(S)}(z)$ and so $C_{N_G(S)}(z)= C_{N_G(J)}(z)$. Thus $|N_G(J)|=3^62^7$ or $3^62^8$ respectively. Since $J^\# \subseteq  3\mathcal{A} \cup  3\mathcal{B}$, $|J \cap
3\mathcal{A}|=48$ and so $[N_G(J):C_{N_G(J)}(a)]=48$.
\end{proof}

\begin{lemma}\label{HN-J is self-centralizing}
If $r$ is any involution in $C_G(Z) \bs Qt$, then $C_Q(r)\cong [Q,r] \cong 3_+^{1+2}$. Furthermore, we have that $C_G(J)=J$ and $J$ normalizes no non-trivial $3'$-subgroup of $G$.
\end{lemma}
\begin{proof}
We may assume that $\<t,r\> \leq C_G(Z)$ is elementary abelian of order four. By coprime action, $Q=\<C_Q(t),C_Q(r),C_Q(tr)\>$. It follows from the fact that $C_Q(t)=Z$ that $C_Q(r),C_Q(tr)>Z$. By the three subgroup lemma we have that $[[Q,r],C_Q(r)]=1$ and it therefore follows that  $[C_Q(r)\cong [Q,r]\cong 3_+^{1+2}$ as claimed. So now suppose that an involution in $C_G(Z)$ centralizes $J$ then $Z_2=C_Q(r)$ is elementary abelian and it follows then that $C_G(J)=J$.
If $N$ is a $3'$-subgroup of $G$ normalized by $J$ then by coprime action, $N=\<C_N(y):y in Y^\#\>$. Since $Y \leq O_3(C_G(y))$ for each $y \in Y^\#$, we have that $[C_N(y),Y]\leq C_N(y) \cap O_3(C_G(y))=1$. Thus $N \leq C_G(Y)$ and in particular, $N$ normalizes $J(O_3(C_G(Y)))=J(W)=J$. Thus $[J,N] \leq N \cap J=1$. We thus have that $N=1$.
\end{proof}

Recall that a group $H$ is said to be $3$-soluble of length one if $H/O_{3'}(H)$ has a normal Sylow $3$-subgroup which is to say that $H=O_{3',3,3'}(H)$.

\begin{lemma}\label{HN-normalier of J-order of the O2}
We have that $O_2({N_G(Y)}/J) \leq O_2(N_G(J)/J) \cong 2_+^{1+4}$ and $N_G(J)/J$ is $3$-soluble of length one.
\end{lemma}
\begin{proof}
Set $K:=N_G(J)$ and $\bar{K}=K/J$. Then $\bar{K}$ has order $3^22^7$ or $3^22^8$ and $\bar{S}\in
\syl_3(\bar{K})$. Clearly $O_3(K)\leq S$ so recall Lemma \ref{C-i's-orders and derived subgroups}. If  $O_3(K)>J$ then it is clear from the order that $O_3(K)\neq S$ and so we must have $O_3(K)=S_i$ for some $i\in \{1,2,3,4\}$. Therefore $K$ normalizes $\mathbf{Z}(S_i)$. However, this leads to a contradiction since $K$ is transitive on $J \cap 3\mathcal{A}$ and on $J \cap 3\mathcal{B}$ and so we have that $O_3(K)=J$.

By Burnside's $p^\a q^\b$-Theorem \cite[4.3.3, p131]{Gorenstein}, $\bar{K}$ is solvable. Let $N$ be
a subgroup of $K$ such that  $J\leq N$ and $\bar{N}=O_2(\bar{K})$. Then $\bar{N} \neq 1$ since
$\bar{K}$ is solvable and $O_3(\bar{K})=1$. Recall that $s$ inverts $J$ and so $\bar{s} \in
\mathbf{Z}(\bar{K})$, in particular, $\bar{s} \in \bar{N}$. Moreover $\bar{N}$ is the Fitting subgroup of $\bar{K}$, $F(\bar{K})$, and so
by \cite[6.5.8]{stellmacher}, $C_{\bar{K}}(\bar{N})\leq \bar{N}$. If any element in $\bar{S}$
centralizes $\bar{N}/\Phi(\bar{N})$ then by a theorem of Burnside \cite[5.1.4, p174]{Gorenstein}, such an element
centralizes $\bar{N}$ and so is the identity. Therefore $\bar{S}$ acts faithfully on
$\bar{N}/\Phi(\bar{N})$ and so by calculating the order of a Sylow $3$-subgroup in $\GL_n(2)$ for
$n=1,2,3$ we see that $|\bar{N}/\Phi(\bar{N})|\geq 2^4$. Moreover, since $\bar{s}$ is central in
$\bar{K}$, we have that $\bar{s}\in \Phi(\bar{N})$ and so $|\bar{N}| \geq 2^5$. We use Lemma \ref{HN-prelim-element of order four normalizing S}
$(iii)$ to find $e\in N_G(S)$ such that $e^2=t$ and $e$ does not normalize $Y$. Since $t$ inverts $\bar{S}$, by Lemma
\ref{HN-prelim-element of order four normalizing S} $(i)$, $\<\bar{e}\>\cap \bar{N}=1$. So, when $|\bar{K}|=3^22^7$ we have that
$|\bar{N}|\leq 2^5$ and so we have that $|\bar{N}|=2^5$ and then that $\bar{K}=\bar{N}\bar{S}\<\bar{e}\>$ is $3$-soluble of length one. So suppose instead that $|\bar{K}|=3^22^8$ and let $P \in \syl_2(N_{C_G(Z)}(S))$ then $|P|=2^3$. We have seen that a subgroup of order four in $P$ intersects trivially with $N$ and so $|P\cap N| \leq 2$. Suppose for a contradiction that $R:=P \cap N$ has order two. We have that $[S/J, RJ/J]\leq S/J \cap N/J=1$ and so $R$ acts trivially on $S/J\cong Q/(Q \cap J)$. Clearly $QR/Q\neq Q\<t\>/Q=\mathbf{Z}(C_G(Z)/Q)$ and so by Lemma \ref{HN-J is self-centralizing}, $[C_Q(R)\cong [Q,R]\cong 3_+^{1+2}$. However we have seen that $R$ centralizes $Q/Z_2$ and so $Z_2=[Q,R]$ which is a contradiction. Thus we again have that  $|\bar{N}|=2^5$ and $\bar{K}=\bar{N}\bar{S}\bar{P}\>$ is $3$-soluble of length one.

Recall that $L=\<Q,Q^x\>\leq N_G(Y)$ and $W=O_3(L)=C_L(Y)$ and $L/W\cong \SL_2(3)$. It follows then that $L/J \cong 3 \times \SL_2(3)$ and so there exists a group $J \leq A\leq K$  such that $\bar{A} \cong \mathrm{Q}(8)$ is normalized by $\bar{S}$ and thus necessarily contained in $\bar{N}$. Recall that we have $e\in N_G(S)\leq K$ such that $e^2=t$ and $e$ does not normalize $Y$. If $\bar{A}^{\bar{e}}=\bar{A}$ then $\bar{L}^{\bar{e}}=\bar{L}$ and it would follow that $Y^e=Y$. Thus if we set $\bar{B}=\bar{A}^{\bar{e}}$ then we may apply Lemma \ref{prelims-extraspecial 2^5 in GL(4,3)} to see that $\bar{N}\cong 2_+^{1+4}$ which completes the proof.
\end{proof}

\begin{lemma}\label{HN-normalizer of J2}
$C_S(s)=\<\alpha_1,\alpha_2\>\cong 3 \times 3$ where $\alpha_1,\alpha_2 \in 3\mathcal{A}$ and there exist
$\<\alpha_1,\alpha_2\>$-invariant subgroups  $\mathrm{Q}(8) \cong X_i\leq C_G(s) \cap C_G(\alpha_i)$ for $i \in\{1,2\}$ such that $s \in X_i$
and $[X_1,X_2]=1$.
\end{lemma}
\begin{proof}
Consider $D:=C_{N_G(J)}(s)\cong C_{N_G(J)}(s)J/J=C_{N_G(J)/J}(s)=N_G(J)/J$. This is a group of
order $2^73^2$ or $2^83^2$ in which $O_2(D)\cong 2_+^{1+4}$. Since $s$ centralizes $S/J\cong Q/(Q\cap J)$ (which is elementary abelian), we see that $P:=C_S(s)$ is elementary abelian of order nine. By Lemma \ref{HN-prelims2} $(ii)$, $C_{C_S(Y)}(s)^\# \subseteq 3 \mathcal{A}$ so let
$\<\alpha_1\>:=C_{C_S(Y)}(s)\leq P$ and by Lemma \ref{lemma structure of Y and W} $(iii)$, $\<\alpha_1\>\lhd C_L(s)\cong 3. \SL_2(3)$. This extension is split and thus a direct product by  Gasch\"{u}tz's Theorem as $P$ is elementary abelian. Thus $\alpha_1$ commutes with a group $X_1\leq C_L(s) $ isomorphic to $\mathrm{Q}(8)$ which is normalized by $P$.

Recall that using Lemma \ref{HN-prelim-element of order four normalizing
S} $(iii)$ there is an element of order four $e\in C_G(Z)$ which normalizes $S$ (and therefore $J$ and $J\<s\>$) but not $Y$ and so ${C_S(Y)} \neq
{C_S(Y)}^e$ and $J\leq C_S(Y^e)\leq S$. Since $s$ centralizes $S/J$, we have $\alpha_1^e=:\alpha_2\in C_{{C_S(Y^e)}}(s)$ and $P=\<\alpha_1,\alpha_2\>$. Moreover, $\<\alpha_2\> \lhd C_{L^e}(s)\cong 3 \times \SL_2(3)$. Let $X_2\leq C_{L^e}(s)$ isomorphic to $\mathrm{Q}(8)$ which is normalized by $P$. Now $D$ is $3$-soluble of length one, and since for $i \in \{1,2\}$, $\<X_i,P\>\cong 3 \times \SL_2(3)$, it follows that $X_i\leq O_2(D)\cong 2_+^{1+4}$. Since $2_+^{1+4}$ contains just two subgroups isomorphic to $\mathrm{Q}(8)$, we have that $X_1X_3\cong 2_+^{1+4}$ and $[X_1,X_2]=1$.
\end{proof}

\begin{lemma}
\begin{enumerate}[$(i)$]
\item In Case I, $C_G(a)\cong 3 \times \alt(9)$ and $N_G(\<a\>)$ is isomorphic to the diagonal subgroup of index two in $\sym(3)\times \sym(9)$.
\item In Case II,  $C_G(a)\cong 3 \times \sym(9)$ and $N_G(\<a\>)\cong\sym(3)\times \sym(9)$.
\end{enumerate}
\end{lemma}
\begin{proof}
We will apply Theorem \ref{princeSym9} in Case I to $N_G(\<a\>)/\<a\>$ and in Case II to $C_G(\<a\>)/\<a\>$ to see that each is isomorphic to $\sym(9)$. Let $N_a:=N_G(\<a\>)$, $C_a:=C_G(\<a\>)$,  $S_a:=C_S(a)\in \syl_3(N_a)$ and  $\bar{N_a}:=N_a/\<a\>$.

We first restrict ourselves to Case I. Consider $C_{\bar{N_a}}(\bar{Z})$. If $g \in N_a$ and $\bar{Z}^{\bar{g}}=\bar{Z}$, then $[Z,g]=1$. This is clear since $\<Z,a\> \cap 3\mathcal{B}\subset Z$.   Recall that $t$ inverts $Q/Z$ and so by swapping $a$ with some appropriate conjugate from $\<Z,a\>$, we may assume that $t$ inverts $a$. By Lemma \ref{HN-prelims1}, $|Q \cap 3\mathcal{A}|=120$ and $C_G(Z)$ is transitive on the
set.  We have therefore that
$[C_G(Z):C_{C_G(Z)}(a)]=120$. Hence $|C_G(Z) \cap N_a|=3^52$. Therefore $C_{\bar{N_a}}(\bar{Z})$ has order $3^42$ with Sylow $3$-subgroup $\bar{S_a}$ and Sylow $2$-subgroup $\bar{\<t\>}$.

Now $\bar{J}\vartriangleleft \bar{C_{\bar{N_a}}(\bar{Z})}$ and $\bar{J}$ is elementary abelian of order $27$. Consider $\mathbf{Z}(\bar{S_a})$. Since $a \in N_3 \cap J$, we may apply  Lemma
\ref{prelims-Centre of the C_i's} to say that
$\mathbf{Z}(\bar{S_a})=\bar{\mathbf{Z}(S_a)}=\bar{Z}$ which has order three. Clearly $\bar{t}$ commutes with $Z$ but not $\bar{S_a}$. Thus we may apply Lemma \ref{Prelims-sym9} to see that $C_{\bar{N_a}}(\bar{Z})\cong C_{\sym(9)}(\hspace{0.5mm} (1,2,3)(4,5,6)(7,8,9) \hspace{0.5mm} )$.

So it remains to show that $\bar{J}$ normalizes no non-trivial $3'$-subgroup of $\bar{N_a}$. However this follows from Lemma \ref{HN-J is self-centralizing}. Hence we may apply Theorem \ref{princeSym9} to see that either $N_a\leq N_G(J)$ or  $\bar{N_a} \cong \sym(9)$. Thus we use Lemma \ref{HN-centralizer of a does not normalize J} to see that $\bar{N_a} \cong \sym(9)$. It follows of course that $C_G(a)/\<a\>\cong \alt(9)$ and using  \cite[33.15, p170]{Aschbacher}, for example, we see
that the Schur Multiplier of $\alt(9)$ has order two. Therefore $C_G(a)$ splits over $\<a\>$ and so $C_G(a)\cong 3 \times \alt(9)$. To see the structure of the normalizer we need only observe that an involution $s$ inverts $J$ and
therefore inverts $a$ whilst acting non-trivially on $O^3(C_G(a))$. Therefore since
$\aut(\alt(9))\cong \sym(9)$, we see that $N_G(\<a\>)$ is isomorphic to the diagonal subgroup of index two in $\sym(3)\times \sym(9)$.

Now in Case II, we consider $C_{\bar{C_a}}(\bar{Z})$. Arguing as before, if $g \in C_a$ and $\bar{Z}^{\bar{g}}=\bar{Z}$, then $[Z,g]=1$. In this case, we conclude from
$[C_G(Z):C_{C_G(Z)}(a)]=120$ that $|C_G(Z) \cap C_a|=3^52$ and so $C_{\bar{C_a}}(\bar{Z})$ has order $3^42$ with Sylow $3$-subgroup $\bar{S_a}$ and a Sylow $2$-subgroup $\bar{\<r\>}$ say, where $r$ is an involution in $C_G(Z)$. If $[\bar{r},\bar{S_a}]=1$ then $[r,S_a]=1$. By Lemma \ref{HN-J is self-centralizing}, $C_Q(r)\cong 3_+^{1+2}$ however, $Q \cap J \leq Q \cap S_a\leq C_Q(r)$ gives us a contradiction. Thus $[\bar{r},\bar{S_a}]\neq 1$ and we may again apply Lemma \ref{Prelims-sym9} to see that $C_{\bar{C_a}}(\bar{Z})\cong C_{\sym(9)}(\hspace{0.5mm} (1,2,3)(4,5,6)(7,8,9) \hspace{0.5mm} )$. Of course we again have that  $\bar{J}$ normalizes no non-trivial $3'$-subgroup of $\bar{C_a}$ so we may apply Theorem \ref{princeSym9} to see that $\bar{C_a} \cong \sym(9)$. It follows then that $C_G(a)\cong 3 \times \sym(9)$ and $N_G(\<a\>)\cong \sym(3)\times \sym(9)$.
\end{proof}

\section{The Structure of the Centralizer of $t$}\label{HN-Section-CG(t)}

We now have sufficient information concerning the $3$-local structure of $G$ to determine the
centralizer of $t$ and to show that in one case $G$ has a non-trivial $2$-quotient. We set $H:=C_G(t)$, $P:=C_J(t)$ and $\bar{H}:=H/\<t\>$. We will show that
$H$\ has shape $2_+^{1+8}.(\alt(5) \wr 2)$ (possibly extended by a cyclic group of order two) and so we must first show that $H$ has an extraspecial subgroup
of order $2^9$. We then show that $H$ has a subgroup, $K$, of the required shape and then finally
we apply a theorem of Goldschmidt to prove that $K=H$. Along the way we gather several results
which will be useful in Section \ref{HN-Section-CG(u)}.

\begin{lemma}\label{HN-conjugates in P}
In Case I,  $C_H(Z)\cong 3 \times  2^{.}\alt(5)$ and $N_H(Z) \cong 3 : 4^{.}\alt(5)$ and $C_H(P)=P\<t\>$.
In Case II, $C_H(Z) \cong 3 \times  2^{.}\sym(5)$ and $N_H(Z) \cong 3 : 4^{.}\sym(5)$ and $C_H(P)\cong 3^2 \times 2^2$. Furthermore,
$|P \cap 3\mathcal{A}|=|P \cap 3\mathcal{B}|=4$  and $P \in \syl_3(H)$.
\end{lemma}
\begin{proof}
By coprime action and an isomorphism theorem, we have that $C_{C_G(Z)/Q}(t)\cong C_{C_G(Z)}(t)/C_Q(t)$ and $C_{N_G(Z)/Q}(t)\cong C_{C_N(Z)}(t)/C_Q(t)$. By
Lemma \ref{HN-prelim-element of order four normalizing S}, $|P|=9$ and since $P\leq J$ is elementary abelian, $P$ splits
over $Z$. Thus $C_{C_G(Z)}(t)$ splits over $Z$ by Gasch\"{u}tz's Theorem and so
$C_{H}(Z)$ and $N_H(Z)$ are as claimed.

Lemma \ref{HN-prelims2} gives us that $|P \cap 3\mathcal{A}|=|P \cap 3\mathcal{B}|=4$ and then it is immediate that $N_H(P)\leq N_H(Z)$ and so $P\in \syl_3(H)$.
\end{proof}

We fix notation such that $P=\{1,z_1,z_1^2,z_2,z_2^2,a_1,a_1^2, a_2,a_2^2\}$ where $z_1 \in Z$, $P \cap
3\mathcal{B}=\{z_1,z_1^2,z_2,z_2^2\}$ and $P \cap 3\mathcal{A}=\{a_1,a_1^2, a_2,a_2^2\}$.

\begin{lemma}\label{HN-Normalizer H of P}
Let $\{i,j\}=\{1,2\}$ then $P \cap O^3(C_G(a_i))=\<a_j\>$ and $N_H(P)/C_H(P)\cong
\mathrm{Dih}(8)$ acts transitively on $3\mathcal{A}\cap P$ and $3\mathcal{B}\cap P$.
\end{lemma}
\begin{proof}
By Lemma \ref{HN-conjugates in P}, $|P\cap 3 \mathcal{A}|=|P\cap 3
\mathcal{B}|=4$ and so it is clear that $N_H(P)/C_H(P)$ is isomorphic to a subgroup of $\mathrm{Dih}(8)$. Observe that every element of order three in $O^3(C_G(a_i))\cong \alt(9)$ or $\sym(9)$ is conjugate to its
inverse. Therefore an element in $O^3(C_G(a_i))$ inverts $P \cap O^3(C_G(a_i))$ and so we must have that element inverting $a_j$ and permuting $\<z_1\>$ and $\<z_2\>$. Thus $P \cap
O^3(C_G(a_i))=\<a_j\>$.
Furthermore an element of order four in $N_H(Z)$ inverts $Z$ whilst centralizing $P/Z$. Hence an element in $N_H(Z)$ permutes $\<a_1\>$ and $\<a_2\>$. We have that
$s$ inverts $P$ and so we have that $N_H(P)$ is transitive on $3\mathcal{A}\cap P$ and $3\mathcal{B}\cap P$.
\end{proof}

\begin{lemma}\label{HN-images in alt9}
Let $\{i,j\}=\{1,2\}$.
\begin{enumerate}[$(i)$]
\item The image in  $O^3(C_G(a_i))$ of elements of cycle type $3$ and $3^2$ are in $3\mathcal{A}$ and those of cycle type $3^3$ are in $3\mathcal{B}$.
In particular,  $a_j \in P \cap O^3(C_G(a_i))$ has cycle type $3^2$.

\item $t$ has cycle type $2^4$ and is not $G$-conjugate to involutions of cycle type $2^2$ in $O^3(C_G(a_i))$.

\item In Case II when  $O^3(C_G(a_i))\cong \sym(9)$, even involutions are not $G$-conjugate to odd involutions.
\end{enumerate}
\end{lemma}
\begin{proof}
We have that $C_G(a_i)\cong 3 \times \alt(9)$ or $3 \times \sym(9)$ and so $|P\cap O^3(C_G(a_i))|=3$. Consider
representatives for the three conjugacy classes of elements of order three in $\alt(9)$. An element of cycle type $3$ must clearly be in $3\mathcal{A}$ and an element of cycle type $3^3$ is the cube of an element of order nine and so by Lemma \ref{elements of order nine in S} must be in $3\mathcal{B}$. Consider now the image of $P \cap O^3(C_G(a_i))$ (which we have seen is equal to $\<a_j\>$) in $\alt(9)$. If it is conjugate to $\<(1,2,3)\>$ then $P$ commutes with a subgroup isomorphic to $3\times 3 \times \alt(6)$. However $z \in P$ and $C_G(z)$ has no such
subgroup. So, since elements of cycle type $3^3$ are in $3\mathcal{B}$, we must have that the image in $\alt(9)$ of $P \cap O^3(C_G(a_i))$ is conjugate to
$\<(1,2,3)(4,5,6)\>$.

We see easily that the image of $t$ is an even permutation since when $O^3(C_G(a_i))\cong \sym(9)$, the odd involutions commute with a group of order nine and so in $C_G(a_i)$ they centralize a  group of order 27. Let $r$ be an involution of cycle type $2^2$. Then $r$ commutes with a single three cycle $x$ say in $O^3(C_G(a_i))$. Thus $r \in C_G(\<a_i,x\>)\cong 3^2 \times \alt(6)$ or $3^2\times \sym(6)$. Clearly then $\<a_i,x\>^\# \subset 3\mathcal{A}$ and so $r$ commutes with no conjugate of $Z$. Thus $r$ is not $G$-conjugate to $t$.

Now suppose we are in Case II and so $O^3(C_G(a_i))\cong \sym(9)$. Continue to let $r \in C_G(\<a_i,x\>)\cong 3^2 \times \sym(6)$ have cycle type $2^2$. There is an element of order four $f$ in  $C_G(\<a_i,x\>)$ which squares to $r$. Hence for every $y \in \<a_i,x\>$, $f \in C_G(y)$ and so the image of $r$ in $O^3(C_G(y))$ is an even involution and thus of type $2^2$. Thus $C_G(r) \cap C_G(y)$ has Sylow $3$-subgroups of order nine. It follows that $C_G(r)$ has Sylow $3$-subgroups of order nine. Now any odd involution in $O^3(C_G(a_i))$ commutes with a group of order 27 and so $r$ is not $G$-conjugate to any odd involution in $O^3(C_G(a_i))$.
\end{proof}

Set $2\mathcal{A}$ to be all $G$-conjugates of an involution whose image in $O^3(C_G(a_1))$ has cycle type $2^2$ and $2\mathcal{B}=\{t^g|g \in G\}$. We now introduce some further notation by first fixing an injective homomorphism from
$N_G(\<a_i\>)$ into $\sym(12)$ such that $O^3(C_G(a_i))$ maps into $\sym(\{1,..,9\})$ and $a_i$
maps to $(10,11,12)$ ($\{i,j\}=\{1,2\}$). We define subgroups and elements of $G$ by their image.

\begin{notation1}\label{HN-Alt9notation}
Let $\{i,j\}=\{1,2\}$.
\begin{enumerate}[$\bullet$]
\item $a_i\mapsto (10,11,12)$.
\item $a_j \mapsto (1,3,5)(2,4,6)$.
 \item $t \mapsto (1,2)(3,4)(5,6)(7,8)$.
 \item $Q_i \mapsto \<(1,2)(3,4)(5,6)(7,8), (1,3)(2,4)(5,8)(6,7),
 (1,5)(3,8)(2,6)(7,4),
 (1,2)(3,4), (3,4)(5,6)\>$.
 \item $r_i \mapsto (1,3)(2,4)$.
 \item When $i=1$, $Q_1>E\mapsto \<(1,2)(3,4)(5,6)(7,8),(1,3)(2,4)(5,8)(6,7),(1,5)(3,8)(2,6)(7,4)\>$.
 \item When $i=2$, $Q_2\ni u\mapsto (1,2)(3,4)$ and $Q_2>F\mapsto \<(1,2)(3,4),(3,4)(5,6)\>$.
\end{enumerate}
\end{notation1}

We observe the following by calculating directly in the image of $N_G(\<a_i\>)$ in $\sym(12)$.
\begin{lemma}\label{HN-alt9 observations}
\begin{enumerate}[$(i)$]
\item In Case I, $C_H(a_i)\sim 3 \times (2_+^{1+4}:\sym(3))$.  In Case II, $C_H(a_i)\sim 3 \times (2_+^{1+4}:(2 \times\sym(3)))$.
In either case, $Q_i\cong 2_+^{1+4}$ is normal in $C_H(a_i)$ with $r_i \in C_H(a_i)\bs Q_i$.
 \item $2 \times 2 \times 2\cong E\vartriangleleft H \cap O^2(C_G(a_1))$ and there exists $\GL_3(2)\cong C\leq C_G(a_1)$ such
 that $a_2 \in C$ and $C$ is a complement to $C_{C_G(a_1)}(E)$ in $N_{C_G(a_1)}(E)$. Furthermore, $N_G(E) \cap C_G(P)=\<t,P\>$.
 \item If $\<t\> <V <Q_i$ such that $V \vartriangleleft C_H(a_i)$ then $V$ is elementary
 abelian.
 \item $C_{C_H(a_i)}(Q_i)=\<t,a_i\>$.
 \item $C_{C_H(a_1)}(E)=\<E,a_1\>$.
 %\item $C_{C_H(a_1)}([E,P])=C_{N_H(\<a_1\>)}([E,P])$.
 \item $s,t,st \in 2 \mathcal{B}$.
 %\item If $Q_i \leq T \in \syl_2(N_G(\<a_i\>))$ then $Q_i$ is characteristic in $T$.
\end{enumerate}
\end{lemma}
\begin{proof}
These can mostly be checked by direct calculation in the permutation group.

To prove $(vi)$ we observe that having fixed the image of $a_j$ we see that the image of $J\cap O^3(C_G(a_i))$ is $\<(1,3,5),(2,4,6),(7,8,9)\>$ and since $s$ inverts $J$ we may assume the image of $s$ is $C_G(a_i)$-conjugate to $(1,3)(2,4)(7,8)(10,11)$. Thus it becomes clear that $s$ is conjugate to $st$. Now by Lemma \ref{lemma structure of Y and W} $(vi)$, $t$ is conjugate to $st$ and therefore to $s$ also.
\end{proof}

For the sake of simplifying language in the following lemma, we define the following for $g \in G$, ${\reflectbox{\rm{N}}_P}(g):=\{M\leq C_H(g) \mid (|M|,3)=1,~[M,P]\leq M,~ C_M(P)\leq \<t\> \}$.
\begin{lemma}\label{HN-3'-subgroups of centralizers}
Let $i\in \{1,2\}$, $\reflectbox{\rm{N}}_P(z_i)=\{1,\<t\>, A_i,B_i\}$ where $A_i\cong B_i\cong \\mathrm{Q}(8)$ are distinct $2$-subgroups of $C_G(z_i)$
with $z_j \in \<A_i,B_i\>=C_H(z_i)'\cong 2^{.}\alt(5)$. Meanwhile, $M \in \reflectbox{\rm{N}}_P(a_i)$ only if $M\leq Q_i$.
\end{lemma}
\begin{proof}
We have that $C_H(a_i)\sim 3 \times 2_+^{1+4}:\sym(3)$ or $3 \times 2_+^{1+4}:(2 \times \sym(3))$ which are subgroups of $3 \times \alt(9)$ and $3\times \sym(9)$ respectively. It is clear in the first case that if $M$ is any normal $3'$-subgroup of $C_H(a_i)$ then $M\leq Q_i$. In the second case we must check within our permutation group that any such $M$ with $C_M(P)\leq \<t\>$ must satisfy $M \leq Q_i$.

We have that $C_H(z_i)\cong 3 \times 2^{.}\alt(5)$ or $3 \times 2^{.}\sym(5)$. Let $M$ be a $3'$-subgroup of $C_H(z_i)$ that is normalized by $P$. It is clear that $5 \nmid |M|$ so $M$ must be a $2$-group. Now $C_H(z_i)$ has Sylow $2$-subgroups isomorphic to $\mathrm{Q}(8)$ or $\mathrm{SDih}(16)$. Since $C_M(P) \leq \<t\>$, we must have that $M\leq\<t\>$ or $M \cong \mathrm{Q}(8)$ and $MP\cong 3 \times \SL_2(3)$. Notice that $P$ is involved in precisely two subgroups of $C_H(z_i)$ isomorphic to $MP\cong 3 \times \SL_2(3)$. We define $A_i$ and $B_i$ to be the two distinct $2$-radical subgroups. It then follows that  $\<A_i,B_i\>=C_H(z_i)'\cong 2^{.}\alt(5)$ and since an element of order four centralizes $z_i$ and inverts $P \cap \<A_i,B_i\>$, we must have that $z_j \in P \cap \<A_i,B_i\>$.
\end{proof}

We continue the notation for the $P$-invariant subgroups from  the previous lemma. The subgroups $\{A_i,B_i\}$ and $Q_j$ for $i,j \in \{1,2\}$ play key roles in this section as our building blocks for $H$.

\begin{lemma}\label{HN-swapping a_i's and z_i's-2}
Let $\{i,j\}=\{1,2\}$. The following hold.
\begin{enumerate}[$(i)$]
\item  $N_H(P) \cap C_H(a_i)$ acts transitively on the set $\{\<z_1\>,\<z_2\>\}$.
\item $N_H(P) \cap C_H(z_i)$ acts transitively on the set $\{\<a_1\>,\<a_2\>\}$.
\item $N_H(P)$ acts as $\dih(8)$ on $\{A_1,B_1, A_2,B_2\}$ with blocks of imprimitivity $\{A_i,B_i\}$. In particular, $N_H(P) \cap N_H(\<a_i\>)$ acts transitively.
\end{enumerate}
\end{lemma}
\begin{proof}
By Lemma \ref{HN-Normalizer H of P}, $N_H(P)/C_H(P)\cong \mathrm{Dih}(8)$ and $N_H(P)$ is
transitive on $P \cap 3\mathcal{A}$ and $P \cap 3\mathcal{B}$ which both have order four and so $(i)$ and $(ii)$ are clear.

Now by Lemma \ref{HN-3'-subgroups of centralizers}, $N_H(P)$ acts imprimitively on the set $\{A_1,B_1,A_2,B_2\}$ swapping $\{A_1,B_1\}$ with $\{A_2,B_2\}$. Recall that $N_H(\<z_2\>)\sim 3:4^{.}\alt(5)$ or $3:4^{.}\sym(5)$. In either case an element of order four, $g$ say, inverts $z_2$ whilst centralizing $\<A_2,B_2\>$. Therefore $g \in C_H(z_1)$. This element necessarily permutes $A_1$ and $B_1$ since Sylow $2$-subgroups of $C_H(Z)$ are either quaternion of order eight of semi-dihedral of order 16. Thus $N_H(P)$ acts transitively and imprimitively on $\{A_1,B_1, A_2,B_2\}$ and contains a tranposition and so $(iv)$ follows. If the subgroup $N_H(P) \cap N_H(\<a_i\>)$ acts as a non-transitive subgroup of $\dih(8)$ then it preserves each $\{A_i,B_i\}$ but therefore normalizes $\<z_1\>$ and $\<z_2\>$ which we have seen is not the case and so we have $(iii)$.
\end{proof}

The following lemma is a key step in determining the structure of $H$ since it proves that $H$
contains a subgroup which is extraspecial of order $2^9$.
\begin{lemma}\label{HN-Q_i's}
Let $\{i,j\}=\{1,2\}$ then $Q_i \cap Q_j=\<t\>$ and $O_{3'}(C_G(Q_i))=Q_j$. In particular $\<t\>$ is
the centre of a Sylow $2$-subgroup of $G$ and $Q_1Q_2\cong 2_+^{1+8}$ with $C_G(Q_1Q_2)=\<t\>$ and $C_{Q_1Q_2}(z_i)=\<t\>$.
\end{lemma}
\begin{proof}
Let $\{i,j\}=\{1,2\}$. Recall Lemma
\ref{HN-alt9 observations} $(iv)$ which says that $C_{C_H(a_i)}(Q_i)=\<t,a_i\>$. Therefore
$C_{\bar{H}}(\bar{Q_i})$ has a self-centralizing element of order three and so we may use Theorem \ref{Feit-Thompson}.

Notice also that $\<a_i\>\in \syl_3(C_G(Q_i))$ and $N_H(P) \cap N_H(\<a_i\>)$ normalizes $C_G(Q_i)$. By Lemma \ref{HN-swapping a_i's and z_i's-2} $(iii)$,  $N_H(P) \cap N_H(\<a_i\>)$ is transitive on
$\{A_1,A_2,B_1,B_2\}$. Thus no group from this set commutes with $Q_i$ else we would have $z_1 \in \<A_2,B_2\>\leq C_G(Q_i)$ which is a contradiction.

Set $N=O_{3'}(C_G(Q_i))$. Then $N$ is normalized by $P$ and by the comment above we see that $C_N(z_1)=C_N(z_2)=\<t\>$. Moreover $C_{N}(a_i)=\<t\>$ and so by coprime action we
have,
\[N=\<C_N(z_1),C_N(z_2),C_N(a_1),C_N(a_2)\>=C_N(a_j).\]

By Lemma \ref{HN-3'-subgroups of centralizers}, since $C_N(P) \leq \<t\>$, we have that $N\leq Q_j$. Suppose that $\<t\> < N< Q_j$. Since $a_i$ acts fixed-point-freely on $N/\<t\>$, $|N|=2^3$. By Lemma \ref{HN-alt9 observations} $(iii)$, $N$ is elementary abelian. Now by Lemma \ref{HN-alt9 observations} $(vi)$, $s$ is conjugate to $t$ in $G$. Recall Lemma \ref{HN-normalizer of J2}. This,
together with the fact that $P=\<a_1,a_2\> \in \syl_3(H)$, implies that for $k\in\{1,2\}$ there
exists a $P$-invariant subgroup $\mathrm{Q}(8)\cong X_k\leq C_H(a_k)$  with $[X_1,X_2]=1$. Now by Lemma
\ref{HN-3'-subgroups of centralizers}, $X_i \leq Q_i$ and  $X_j \leq Q_j$.  We have that $X_j$ and $N$ are both $P$-invariant and furthermore we have that $X_j\cong \mathrm{Q}(8)$ where as $N$ is elementary abelian. Therefore $|X_j \cap N|=2$ and so $Q_j=NX_j$. Similarly, $Q_i=O_2(C_G(Q_j))X_i$. Therefore $X_j$
commutes with $Q_i$ which is a contradiction.

So we have that either $N=Q_i$ or $N=\<t\>$. Assume the latter for a contradiction.

In Case I, we have that $N_G(\<a_i\>)$ is the diagonal subgroup of index two in $\sym(3)\times \sym(9)$. It follows that $C_G(Q_i)$ has a self-normalizing Sylow $3$-subgroup and so a normal $3$-complement. Therefore, $C_G(Q_i)=N \<a_i\>=\<t\>\times \<a_i\>$. Hence $N_G(Q_i) \leq N_G(\<a_i\>)$ has Sylow $2$-subgroups of order $2^7$ isomorphic to a Sylow $2$-subgroup of $\sym(9)$. We check that $Q_i$ is characteristic in such a $2$-group to see that $N_G(\<a_i\>)$ contains a Sylow $2$-subgroup of $G$. Therefore there exists $g\in G$ such that $A_1^g \in N_G(\<a_i\>)$. However, using Lemma \ref{HN-element of order four in CG(Z)} we now get a contradiction since no element of order four in $A_1$ commutes with an element of $3\mathcal{A}$. Thus in Case I we have that $O_{3'}(C_G(Q_i))=Q_j$.

In Case II we instead have that $N_G(\<a_i\>)\cong \sym(3)\times \sym(9)$ and using  Theorem \ref{Feit-Thompson} we see that $C_G(Q_i)/\<t\> \cong \alt(5)$,  $\PSL_2(7)$ or $\sym(3)$.  Notice that an element of order three in $P$ must therefore commute with $C_G(Q_i)$. That element must be in $\<a_j\>$. However, $t\in C_G(a_j)$ and it follows from Lemma \ref{HN-images in alt9} that $t$ is $2$-central in $C_G(a_j)\cong 3 \times \sym(9)$. A $2$-central involution in $\sym(9)$ does not commute with subgroups isomorphic to $\alt(5)$, $\PSL_2(7)$ or their double covers. It follows that we must have $C_G(Q_i)/\<t\>\cong \sym(3)$. We now again have that $N_G(Q_i)$ is contained in $N_G(\<a_i\>$ and we get a contradiction as before.

Hence we have in both cases that $N=Q_j$ and also that $C_G(Q_i)/N$ acts faithfully on $N/\<t\>$. So $[Q_1,Q_2]=1$ and furthermore $Q_1 \cap Q_2 \leq C_{Q_1}(Q_1)=\<t\>$ and so we get that $Q_1Q_2\cong 2_+^{1+8}$ and clearly $C_{Q_1Q_2}(z_i)=\<t\>$. Now let
$Q_1Q_2\leq T\in \syl_2(G)$ then $\mathbf{Z}(T)\leq C_T(Q_1Q_2)\leq C_T(Q_1) \cap C_T(Q_2)$. Since $C_G(Q_i)/N$ acts faithfully on $N/\<t\>$, it is clear that $C_T(Q_1) \cap C_T(Q_2)=\<t\>$. Hence
$\mathbf{Z}(T)=C_G(Q_1Q_2)=\<t\>$.
\end{proof}

Set ${Q_{12}}:=Q_1Q_2\cong 2_+^{1+8}$ and recall that in Notation \ref{HN-Alt9notation} we defined
$E\leq C_G(a_1)$ such that $t \in E\trianglelefteq C_H(a_1)$ is elementary abelian of order eight. We now consider $C_G(E)$ and $N_G(E)$.

%\begin{lemma}\label{HN-Describing E}
%We have  $t\in E \trianglelefteq  C_H(a_1)$, $C_G(E)/E$ has a nilpotent normal $3$-complement on which $a_1$
%acts fixed-point-freely. Furthermore, $N_G(E)/C_G(E)\cong \GL_3(2)$ where the extension is split and there is
%a complement to $C_G(E)$ in $C_H(a_1)$ containing $a_2$.
%\end{lemma}
%\begin{proof}
%By Lemma \ref{HN-alt9 observations} $(v)$, $C_{C_H(a_1)}(E)=C_{N_H(\<a_1\>)}(E)=\<E,a_1\>$ and so
%by Burnside's normal $p$-complement Theorem  (Theorem \ref{Burnside-normal p complement}), $C_G(E)$
%has  a normal $3$-complement, $M$ say and $C_M(a_1)=E$ which implies that $a_1$ acts
%fixed-point-freely on $M/E$. A theorem of Thompson says that $M/E$ is nilpotent and therefore $M$ is nilpotent. Also by Lemma \ref{HN-alt9 observations} $(ii)$, there exists a
%complement to $C_G(E)$ in $C_H(a_1)$ containing $a_2$.
%\end{proof}

\begin{lemma}\label{HN-info on centralizer of E}
\begin{enumerate}[$(i)$]
\item We have that $C_G(E)/O_2(C_G(E))\cong C_3$ in Case I and  $C_G(E)/O_2(C_G(E))\cong\sym(3)$ in Case II.
\item Without loss of generality (on choices of $A_i$)  we may assume that $O_{2}(C_G(E))=\<E,Q_2,A_1,A_2\>$ which is normalized by $P$.
\item $N_G(E)/C_G(E)\cong \GL_3(2)$ where the extension is split and there is a complement to $C_G(E)$ in $C_H(a_1)$ containing $a_2$.
\item $\<Q_{12},A_1,A_2\>$ is a $2$-group which is normalized by $P$.
\end{enumerate}
\end{lemma}
\begin{proof}
By Lemma \ref{HN-alt9 observations} $(v)$, $C_{C_H(a_1)}(E)=\<E,a_1\>$ and so
$C_G(E)/E$ satisfies Theorem \ref{Feit-Thompson}. Set $N=O_{3'}(C_G(E))$ then $C_N(a_1)=E$ and $a_1$ acts
fixed-point-freely on $N/E$. A theorem of Thompson says that $N/E$ is nilpotent and therefore $N$ is nilpotent. By Theorem \ref{Feit-Thompson}, $C_G(E)/N\cong C_3$, $\sym(3)$, $\alt(5)$ or $\PSL_2(7)$ (in which case $N=E$). Also by Lemma \ref{HN-alt9 observations} $(ii)$, there exists a
complement to $C_G(E)$ in $N_G(E)$ which in particular is a subgroup of $C_H(a_1)$ containing $a_2$.

Since $P$ normalizes $N$, we may apply coprime action to see that
\[N=\<C_N(z_1),C_N(z_2),C_N(a_1),C_N(a_2)\>.\] Since $C_N(P)\leq C_E(a_2)=\<t\>$, we use Lemma \ref{HN-3'-subgroups of centralizers} to see that $N$ is generated by $2$-groups and as $N$ is nilpotent, $N$ is a $2$-group.

We have that $E\leq Q_1$ and by Lemma \ref{HN-Q_i's}, $[Q_1,Q_2]=1$ and $Q_1 \cap Q_2=\<t\>$ so $Q_2 \cap E=\<t\>$ and $Q_2 \cap N >\<t\>$ and so has order $2^3$ or $2^5$.
Notice also that this implies that $N$ does not split over $E$. Let $g\in N_G(E)\cap C_G(a_1)$ be an element of order seven. Then $g$
acts fixed-point-freely on $E$. If $[N/E,g]=1$ then $N=C_N(g)\times E$ which is a contradiction.
Thus $[N/E,g]\neq 1$ and so $|N/E|\geq 2^3$. Since $a_1$ acts fixed-point-freely on $N/E$ and
preserves $[N/E,g]$, we have $|[N/E,g]| \geq 2^6$.

If $z_1$ and $z_2$ act fixed-point-freely on $N/E$ then $N=Q_2E$ and so
$|N/E|=2^2$ or $2^4$ which we have seen is not the case. Therefore at least one of $C_{N/E}(z_1)$ and
$C_{N/E}(z_2)$ is non-trivial.  Since $E\trianglelefteq  C_H(a_1)$  we may apply Lemma
\ref{HN-swapping a_i's and z_i's-2} $(i)$ which says that $N_H(P) \cap C_H(a_1)$ acts transitively
on the set $\{\<z_1\>,\<z_2\>\}$. Therefore $C_{N/E}(z_1)$ and
$C_{N/E}(z_2)$ are both non-trivial. So we may assume, without loss of generality, that $A_1\leq N$
and $A_2 \leq N$ and so
$N=\<E,Q_2\cap N,A_1,A_2\>$.

Now suppose that $C_G(E)/N\cong \alt(5)$. Since $N_G(E)/C_G(E)\cong \GL_3(2)$, we have that $N_G(E)\cong \alt(5) \times \GL_3(2)$. We have that $Na_2$ is an element of order three in $N_G(E)/N$ and $C_{N_G(E)/N}(Na_2)$ contains a subgroup isomorphic to $\alt(5)$. By coprime action,
\[C_{N_G(E)/N}(a_2)= C_{N_G(E)}(a_2)N/N\cong C_{N_G(E)}(a_2)/C_N(a_2)\] which is a $2$-group of order 8 or 32 extended by $\alt(5)$ which does not exist in $\sym(9)$.  Hence  $C_G(E)/N\cong \sym(3)$ or $C_3$ and it follows that $N=\<E,,A_1,A_2\>$. Additionally since $Q_1$ normalizes $E$ and so normalizes $N$, we see that $\<Q_{12},A_1,A_2\>$ is a $2$-group which is clearly normalized by $P$.
\end{proof}
We continue the notation from this lemma for the rest of this section such that $A_1$ and $A_2$ commute with $E$. Set $K:=N_G(Q_{12})\leq H$. We show in the rest of this section that $K=H$.

\begin{lemma}\label{HN-three things about K}
\begin{enumerate}[$(i)$]
\item $N_H(P)\leq K$ and  $|N_H(P)Q_{12}/Q_{12}|=3^22^3$ in Case I or $3^22^4$ in Case II.

\item Suppose that $v\in K$ such that $Q_{12}v$ is an involution which inverts ${Q_{12}}z_i$ for some $i \in \{1,2\}$.
Then $C_{\bar{{Q_{12}}}}(v)=[\bar{Q_{12}},v]$ has order $2^4$.

\item $C_{Q_{12}}(A_1)\neq C_{Q_{12}}(A_2)$.

\item For $i \in \{1,2\}$, $C_H(a_i)\leq K$.

\item For $i \in \{1,2\}$, $C_H(z_i)\leq K$.

\item $\<A_1,B_1\>$ commutes with $\<A_2,B_2\>$ modulo $Q_{12}$ and in particular, $\<A_1,A_2\>Q_{12}/Q_{12}\cong 2^4$.

\end{enumerate}
\end{lemma}
\begin{proof}
$(i)$ First observe that $N_H(P)$ acts on the set $\{a_1,a_2,a_1^2,a_2^2\}=P \cap 3\mathcal{A}$ and
therefore it preserves ${Q_{12}}=Q_1Q_2$ so $N_H(P) \leq K$. Clearly $N_{Q_{12}}(P)$ commutes with $P$ and $\<t\> \leq C_{Q_{12}}(P)\leq Q_1 \cap Q_2=\<t\>$ so the order of the quotient is clear given Lemmas \ref{HN-conjugates in P} and \ref{HN-Normalizer H of P}.

$(ii)$ Observe that $\bar{Q_{12}}$ is elementary abelian and $\bar{Q_{12}}\bar{v}$ has order two and
inverts $\bar{Q_{12}}\bar{z_i}$ which has order three. Therefore we may use Lemma
\ref{Prelims-centralizers of invs on a vspace which invert a 3} and since
$|C_{\bar{Q_{12}}}(z_i)|=1$, we have that $|C_{\bar{{Q_{12}}}}(v)|\leq 2^4$. Of course we always have that
$|C_{\bar{{Q_{12}}}}(v)|\geq 2^4$ and so we get equality.

$(iii)$ Suppose that $C_{Q_{12}}(A_1)=C_{Q_{12}}(A_2)$. Recall that $N_H(P)\sim 3:4^{.}\alt(5)$ acts as $\dih(8)$ on $\{A_1,B_1,A_2,B_2\}$. Therefore there exists $g\in N_H(P)$ permuting $A_1$ and $B_1$ and fixing $A_2$ and $B_2$. Hence
\[C_{Q_{12}}(A_1)=C_{Q_{12}}(A_2)=C_{Q_{12}}(A_2)^g=C_{Q_{12}}(A_1)^g=C_{Q_{12}}(B_1).\]
Therefore $E \leq C_{Q_{12}}(A_1)=C_{Q_{12}}(\<A_1,B_1\>)\leq C_{Q_{12}}(z_2)$. This is a contradiction.

$(iv)$ This is clear since
$C_H(a_i)=Q_iN_{C_H(a_i)}(P)$ and so $C_H(a_i)\leq K$.

$(v)$ By Lemma \ref{HN-info on centralizer of E}, $T:=\<Q_{12},A_1,A_2\>$ is a $2$-group which is normalized by $P$.  We consider $N_T(Q_{12})\leq K$. Since $T$ is normalized by $P$, we apply coprime action to see that \[N_T(Q_{12})=\<C_{N_T(Q_{12})}(r)\mid r \in P^\#\>.\] We have that $T\leq N_G(E)$ and by Lemma \ref{HN-alt9 observations} $(ii)$, $C_T(P)=\<t\>$ and so we may use Lemma \ref{HN-3'-subgroups of centralizers}. Since $Q_{12}$ is normalized by $N_H(P)$ which is transitive on $\{A_1,B_1,A_2,B_2\}$ (by Lemma \ref{HN-swapping a_i's and z_i's-2} $(iii)$), it is clear that $A_1 \nleq Q_{12}$. Thus $N_T(Q_{12})>Q_{12}$. Now we use Lemma \ref{HN-3'-subgroups of centralizers} to see that for $j \in \{1,2\}$ $C_{N_T(Q_{12})}(a_j)=Q_j$ and to see that for some $i \in \{1,2\}$,  $C_{N_T(Q_{12})}(z_i)\in \{A_i,B_i\}$. However we again apply Lemma \ref{HN-swapping a_i's and z_i's-2} $(iii)$ to see that since one of $A_i$ or $B_i$ is in $K$ and $N_H(P) \leq K$ is transitive on $\{A_1,B_1,A_2,B_2\}$, $\<A_1,B_1,A_2,B_2\> \leq K$. Moreover, $C_H(P)\leq K$ and so for $i \in \{1,2\}$, $\<A_i,B_i,C_H(P)\>=C_H(z_i)\leq K$.

$(vi)$ Let $\hat{K}:=K/Q_{12}$ Since $\hat{T}=\<\hat{A_1},\hat{A_2}\>$ is a $2$-subgroup of $\hat{K}$. We apply the same coprime action arguments again to $\hat{T}$ and $\mathbf{Z}(\hat{T})$ to see that $\hat{T}$ is elementary abelian of order 16. Thus $[\hat{A_1},\hat{A_2}]=1$. However we have seen that we can choose an element of order four $g \in N_G(Z)$ that fixes $A_2$ whilst permuting $\{A_1,B_1\}$. Thus $1^g=[\hat{A_1},\hat{A_2}]^g=[\hat{B_1},\hat{A_2}]$. So $[\<\hat{A_1},\hat{B_1}\>,\hat{A_2}]$. Repeating these arguments we get to $[\<\hat{A_1},\hat{B_1}\>,\<\hat{A_2},\hat{B_2}\>]=1$.
\end{proof}

%
%
%
%\begin{lemma}\label{HN-involutions in O^2(K)/Q}
%Let $\bar{v}\in \bar{\<Q_{12},A_1,B_1,A_2,B_2\>}\bs \bar{{Q_{12}}}$ be an involution. Then $|C_{\bar{{Q_{12}}}}(\bar{v})|=2^4$. If $Q_{12}v\in Q_{12}\<A_i,B_i\>/Q_{12}$ then $v$ is an element of
%order four squaring to $t$ and $C_H(v)$ contains a conjugate of $Z$. Otherwise $Q_{12}v$ is diagonal in which case $v\in 2\mathcal{B}$.
%%%%%%%%%%%%%%%%%%%%%%%%%%%%%%%%%%%%%%%%%%%%and $|C_{\bar{K}}(\bar{v})|=2^9$.
%\end{lemma}
%\begin{proof}
%We have that $Q_{12}\<A_1,A_2,B_1,B_2\>/Q_{12} \cong \alt(5) \times \alt(5)$ which has two conjugacy classes of involutions. Any such involution inverts a conjugate of $Q_{12}z_1$ and therefore by Lemma \ref{HN-three things about K} $(ii)$, $|C_{\bar{{Q_{12}}}}(v)|=2^4$ and then by Lemma \ref{lem-conjinvos}, every involution in
%$\bar{{Q_{12}}}\bar{v}$ is conjugate to $\bar{v}$. We may choose an element of order four, $f_i\in
%C_H(z_i)$ with $f_i^2=t$. Then $\bar{{Q_{12}}}\bar{f_i}$ is an involution and represents every involution in one of the conjugacy classes. Now a diagonal involution has the form $\bar{f_1f_2}$ and we may choose it such that $f_1f_2$  inverts $\bar{P}$. Notice that $f_1f_2$ is an involution in $N_H(Z)$ and that $\bar{\<f_1f_1,A_1,B_1\>}\cong 2 \times \alt(5)$. It therefore follows from Lemma \ref{HN-prelims1} that $f_1f_2$ is conjugate to $s$ so lies in $2\mathcal{B}$.
%\end{proof}

\begin{lemma}\label{HN-alt5 and char 2}
Set $D=\<N_H(P),C_H(p)\mid p \in P^\#\>$.
\begin{enumerate}[$(i)$]
\item In Case I, $D/Q_{12}$ is isomorphic to $\alt(5)\wr 2$. In Case II, $D/Q_{12}$ is isomorphic to a subgroup of $\sym(5) \wr 2$ of shape $(\alt(5)\wr 2).2$.
\item A Sylow $2$-subgroup of $D/Q_{12}$ has a unique elementary abelian subgroup of order 16.
\item If If $Q_{12}v$ is an involution in $(DQ_{12})'/Q_{12}\cong \alt(5)\times \alt(5)$, then either $v$ has order four and squares to $t$ with $C_H(v)$ containing a conjugate of $Z$ or $Q_{12}v$ is diagonal in which case $v\in 2\mathcal{B}$ and $[\bar{Q_{12}},v]=C_{\bar{Q_{12}}}(v)$ has order $2^4$.
\end{enumerate}
\end{lemma}
\begin{proof}
Set $\hat{K}=K/Q_{12}$ then $\hat{D}=\<\hat{A_1},\hat{A_2},\hat{B_1},\hat{B_2}, \hat{N_H(P)}\>$.  We have seen that $\<\hat{A_1},\hat{A_2},\hat{B_1},\hat{B_2}\>$ is isomorphic to  $\alt(5)\times \alt(5)$. It is normalized by $\hat{N_H(P)}$ which has order $3^22^3$ or $3^22^4$ by Lemma \ref{HN-three things about K} $(i)$. Thus we have that either $\hat{D}\sim (\alt(5)\times \alt(5)).2$ which in fact is isomorphic to $\alt(5) \wr 2$ since an element in $N_H(P)$ swaps $\<A_1,B_1\>$ and $\<A_2,B_2\>$. Otherwise we have that $\hat{D}\sim (\alt(5)\times \alt(5) ):(2 \times 2)$. Note that $\hat{D}$ is a subgroup of the automorphism of $\alt(5)\times \alt(5)$ otherwise $\<C_{\hat{D}}(\hat{P}),\hat{A_1}\>$ contains an abelian subgroup of $C_{\hat{H}}(\hat{Z})$ of order $2^3$ which is not possible. So $\hat{D}$ is an index two subgroup of $\sym(5)\wr 2$. It is not isomorphic to $\sym(5) \times \sym(5)$ as an element in $N_H(P)$ swaps $\<A_1,B_1\>$ and $\<A_2,B_2\>$. It remains to calculate in the remaining two groups that a Sylow $2$-subgroup has a unique elementary abelian subgroup of order 16. This proves $(i)$ and $(ii)$.

Now $\hat{D}'$ has two conjugacy classes of involutions: diagonal and non-diagonal.  If $\hat{v}$ is an involution in $\hat{D}'$ then it inverts a conjugate of $\hat{z_1}$ and therefore by Lemma \ref{HN-three things about K} $(ii)$, $C_{\bar{{Q_{12}}}}(v)=[\bar{Q_{12}},v]$ has order $2^4$. Moreover $\bar{{Q_{12}}}\bar{v}$ is an involution and by Lemma \ref{lem-conjinvos}, every involution in
$\bar{{Q_{12}}}\bar{v}$ is conjugate to $\bar{v}$. Thus if $\hat{w}=\hat{v}$ the $v$ is conjugate to $w$ or to $tw$. We may choose an element of order four, $f_i\in
\<A_i,B_i\>)\cap N_H(P)$ with $f_i^2=t$. Then $\hat{f_i}$ represents every non-diagonal involution in $\hat{D}'$. Now $\hat{f_1f_2}$ is a diagonal involution and $f_1f_2\in N_H(P)$. In fact $f_1f_2$ is in $N_H(Z)$ and inverts $P$. We can see in $N_H(\<a_1\>)$, for example, that no element of order four inverts $P$ and so $f_1f_2$ is an involution.  Moreover, $\hat{\<f_1f_2,A_1,B_1\>}=\hat{\<f_2,A_1,B_1\>}\cong 2 \times \alt(5)$ so $\<f_1f_2,A_1,B_1\>\cong 4^{.}\alt(5)$. It therefore follows from Lemma \ref{HN-prelims1} that $f_1f_2$ is conjugate to $s$. Thus if $\hat{v}$ is an involution in $\hat{D}'$ then $v$ is either an element of order four conjugate to $f_i$ or $f_i^3$ or is conjugate to $s$ or $st$ and so in $2\mathcal{B}$ which proves $(iii)$.
\end{proof}

\begin{lemma}\label{HN-centralizer of F}
% $C_G(E)\leq K$ and $C_G(E)$ is normal in a Sylow $2$-subgroup of $H$.
Assume that we are in Case I. For $V<E$ such that $t \in V\cong 2 \times 2$, $C_G(V)\leq \<C_H(p)\mid p \in P^\#\>$.
\end{lemma}
\begin{proof}
Notice that $E=\<t\>\times [E,P]$ and that the image of $[E,P]$ in $O^3(C_G(a_1))$ is $\<(1, 5)(2, 6)(3, 8)(4, 7),(1, 8)(2, 7)(3, 5)(4, 6)\>$. Since we are in Case I, we have that $N_G(\<a_1\>)$ is the diagonal subgroup of index two in $\sym(3) \times \sym(9)$. We calculate the centralizer in $\sym(9)$ and $\alt(9)$ of such a fours group to see it has order 32. Thus $C_G([E,P])$ has a Sylow $3$-subgroup $\<a_1\>$ with centralizer $\<E,a_1\>$ equal to its normalizer. Hence $C_G([E,P])$ has a normal $3$-complement, $N$ say. It is clear that $N$ is normalized by $P$ with $C_N(P)=\<t\>$ and contains $\<C_{Q_1}([E,P]),Q_2,A_1,A_2\>$ which has order $2^{12}$. Now by coprime action, $N=\<C_N(a_1),C_N(a_2),C_N(z_1),C_N(z_2)\>$. It is clear that for $i \in \{1,2\}$, $A_i$ is a maximal $3'$-subgroup of $C_G(z_i)$ so $C_N(z_i)=A_i$. Also $Q_2$ is a maximal $3'$-subgroup of $C_G(a_2)$ normalized by $P$ and we have calculated that $|C_N(a_1)|=32$. Thus $N=\<C_N(a_1),O_2(C_G(E))\>$. Now $N_G(E)$ is transitive on fours subgroups of $E$ but normalizes $O_2(C_G(E))$ and so there exists $g \in C_G(a_1)$ such that $t \in V=[E,P]^g$ and so $N^g=O_{3'}(C_G(V))$ with $N^g=\<C_{N^g}(a_1),O_2(C_G(E))\>$ and we see that $C_G(V)\leq \<C_H(p)\mid p \in P^\#\>$.
\end{proof}

\begin{lemma}\label{HN-InvOrbs}
\begin{enumerate}[$(i)$]
\item $K$ has a subgroup $K_0$ such that $K_0/Q_{12}\cong \alt(5)\wr 2$ acts faithfully on $Q_{12}$.
\item Every involution in $Q_{12}$ is in $2\mathcal{A}\cup 2 \mathcal{B}$ and $K_0$ is transitive on $Q_{12} \cap 2\mathcal{A}$ and on $(Q_{12}\bs \<t\>)\cap 2 \mathcal{B}$. The orbit lengths are 120 and 150 respectively.
\item Diagonal subgroup of $K_0$ isomorphic to $\alt(5)$ either centralizes a subgroup of $Q_{12}$ isomorphic to $C_4 \times C_2$ containing an involution in $2\mathcal{A}$ or contain an element of order five acting fixed-point-freely on $\bar{Q_{12}}$.
\end{enumerate}
In particular, in Case I, $K_0=K$.
\end{lemma}
\begin{proof}
Again we set $\hat{K}=K/Q_{12}$. We have seen in Lemma \ref{HN-images in alt9} and Notation \ref{HN-Alt9notation}, that $Q_1$ contains non-conjugate involutions from $2\mathcal{A}$ and from $2\mathcal{B}$. We have seen that $K$ has a subgroup, $K_0$ say, such that $\hat{K_0}$  is isomorphic to $\alt(5) \wr 2$. We consider the action of this group on $Q_{12}$. The action is clearly faithful as $C_G(Q_{12})=\<t\>$.

Now for $\{i,j\} = \{1,2\}$, $\<\hat{A_i},\hat{B_i}\>\cong \alt(5)$ acts on $\bar{Q_{12}}$ with an element of order three, $\hat{z_j}$, acting fixed-point-freely. It therefore follows that $\bar{Q_{12}}$ is a sum of two natural $\GF(2)\alt(5)$-modules. In particular, an element of order five in $\<\hat{A_i},\hat{B_i}\>$ acts fixed-point-freely on $\bar{Q_{12}}$.

Now $\hat{K_0}$ contains two further conjugacy classes of subgroups isomorphic to $\alt(5)$; the diagonal subgroups. One of these lies in a  $\sym(5)$ the other in an $\alt(5) \times 2$ and furthermore $\hat{K_0}$ contains two further conjugacy classes of subgroups of order five. Let $F$ be a Sylow $5$-subgroup of $\hat{K_0}$ then by coprime action $Q_{12}=\<C_{Q_{12}}(f):f \in F^\#\>$. Since $C_G({Q_{12}})=\<t\>$, no element in $F$ acts trivially. However $F$ has six subgroups of order five and for one of these, $\<f\>$ say we must have $\<t\> < C_{Q_{12}}(f) < Q_{12}$ and $Q_{12}=[Q_{12},f]C_{Q_{12}}(f)$ where $[Q_{12},f]$ and $C_{Q_{12}}(f)$ are both extraspecial and intersect at $\<t\>$. Moreover, for any other element $f'\in F\bs \<f\>$ we have that $C_{Q_{12}}(f')\cap C_{Q_{12}}(f)=\<t\>$ so $C_{Q_{12}}(f)$ has an automorphism of order five. Of course $[Q_{12},f]$ has an automorphism $f$ of order five also. It follows that $C_{Q_{12}}(f)\cong [Q_{12},f]\cong 2_-^{1+4}$. Moreover, exactly two of the subgroups of $F$ of order five act non-trivially on $\bar{Q_{12}}$ and so it follows that one class of diagonal $\alt(5)$ subgroups of $\hat{K_0}$ contain such an element of order five and the other contains a fixed-point-free element of order five.

Recall that $E\leq Q_1$ commutes with $\<A_1,A_2\>$ so consider an involution, $v$ say in $E$ distinct from $t$. Suppose that $v$ is fixed by an element of order five as well as $\<A_1,A_2\>$ in $\hat{K_0}$. Then it follows that $C_{\hat{K_0}}(v)$ contains $\hat{K_0}'$ which is a contradiction. Thus $v$ lies in a $\hat{K_0}$-orbit which is a multiple of 25. Clearly $v$ does not commute with $P$ and so $v$ lies in an orbit which is a multiple of 3. Also $v$ is conjugate to $vt$ in $Q_{12}$ and so $|v^{K_0}|\geq 150$.

Now, let $D$ be a diagonal subgroup of $\hat{K_0}$ isomorphic to $\alt(5)$ and let $F$ be a Sylow $5$-subgroup of $D$. We choose $D$ such that $Q_{12}a_2$ generates a Sylow $3$-subgroup of $D$. Note that $F$ acts non-trivially on a subgroup of $\bar{Q_{12}}$ of order $2^4$ and that $[Q_{12},F]\cong 2_-^{1+4} \ncong 2_+^{1+4}\cong [Q_{12},a_2]=Q_1$. In fact $|[Q_{12},F]\cap Q_1|\leq 2^3$ because $[Q_{12},F]$ is centralized by an element of order five. Thus, if $|[Q_{12},F]\cap Q_1|=2^4$ then $|[Q_{12},F]\cap E|=2^2$ but we have seen no element of $E\bs\<t\>$ commutes with an element of order five.  Thus,  $|[Q_{12},F]\cap Q_1|\leq 2^3$. Now suppose that  $|C_{Q_{12}}(F)\cap C_{Q_{12}}(a_2)|=|C_{Q_{12}}(F)\cap Q_2|=4$ then, as a $D$-module, $\bar{Q_{12}}$ has no submodule which is a sum of two trivial modules. However this implies that $\< [\bar{Q_{12}},F],Q_1 \>$ is a sum of a 4-dimensional and a trivial module and so has order $2^5$ however that means that $|[Q_{12},F]\cap Q_1|\geq 2^4$ which is a contradiction. Thus  $|C_{Q_{12}}(F)\cap Q_2|= 8$ (it can be no larger without containing a conjugate of $v\in E\bs\<t\>$). Since $2_-^{1+4}$ has $2$-rank $2$, $C_{Q_{12}}(F)\cap Q_2$ has two rank at most 2. If it had $2$-rank 1 then it would necessarily be isomorphic to $\mathrm{Q}(8)$. However, $Q_2$ has just two subgroups isomorphic to $\mathrm{Q}(8)$ both of which are normalized by $a_1$. Any subgroup of $C_{Q_{12}}(F)\cap Q_2$ normalized by $a_1$ is normalized by $\hat{\<a_1,D\>}=\hat{K_0}'$ which is again a contradiction. Thus $C_{Q_{12}}(F)\cap Q_2\cong 4 \times 2$. This implies that an involution in $Q_2\bs \<t\>$ is centralized by $D$. Call this involution $w$ and observe that $w \notin v^{K_0}$. Now $\hat{K_0}$ acts faithfully on $Q_{12}$ and so $C_{\hat{K_0}}(w)=D$ or is a maximal subgroup of $\hat{K_0}$ of shape $2 \times \alt(5)$ or $\sym(5)$. In particular, $\bar{w}$ lies in a $\hat{K_0}$-orbit of length a multiple of 60. Thus $|w^{K_0}|\geq 120$. Now $Q_{12}$ has 270 involutions and so every involution lies in $v^{K_0} \cup w^{K_0}$. Since $Q_{12}$ contains representatives from $2\mathcal{A}$ and $2 \mathcal{B}$ we must have that  $w^{K_0}=Q_{12} \cap 2\mathcal{A}$ and  $v^{K_0}=Q_{12}\cap 2 \mathcal{B}$.

Finally, in Case I, by Lemma \ref{HN-alt5 and char 2}, $K_0$ contains $C_H(p)$ for each $p \in P^\#$ and by Lemma \ref{HN-centralizer of F}, an involution in $Q_{12}\bs \<t\>$ has centralizer contained in $K_0$ so we may conclude that $K=K_0$.
\end{proof}

\begin{lemma}\label{HN-new-sylow2}
$N_H(E) \leq K$  and contains a Sylow $2$-subgroup of $G$. In Case I, a Sylow $2$-subgroup has order $2^{14}$. In Case II, it has order $2^{15}$ and is self-normalizing with derived subgroup contained in $Q_{12}A_1A_2$.
\end{lemma}
\begin{proof}
It follows from Lemma \ref{HN-info on centralizer of E} that $C_G(E) \leq K$. So we consider
$N_H(E)$. By Lemma \ref{HN-info on centralizer of E}, there exists a complement, $C$, to $C_G(E)$ in $N_G(E)$
such that $C \leq C_G(a_1)$. Now, by Dedekind's Modular Law, $N_G(E) \cap
H=C_G(E)C \cap H=C_G(E)(C \cap H)$. Furthermore, $\sym(4)\cong C \cap H \leq C_H(a_1)\leq K$ by Lemma
\ref{HN-three things about K} $(iv)$. Thus $N_H(E) \leq K$.

In Case I, we have seen that $K/Q_{12}\cong \alt(5)\wr 2$ and so $K$ has Sylow $2$-subgroups of order $2^{14}$. To see these are Sylow $2$-subgroups of $H$ we must show that $Q_{12}$ is characteristic in any such. First we consider Case II.

We have seen in Lemma \ref{HN-info on centralizer of E} that $O_2(C_G(E))=\<E,Q_2,A_1,A_2\>$ and $C_G(E)/O_2(C_G(E))\cong\sym(3)$. We have seen also that $A_1$ and $A_2$ commute modulo $Q_{12}$, however it is clear that they must in fact commute modulo $EQ_2$. Thus $|O_2(C_G(E))|=2^{11}$ and $|N_H(E)|=2^{15}3^2$. Since $N_H(E)$ normalizes $Q_{12}$, it is clear that $Q_1Q_2A_1A_2\unlhd N_H(E)$ and has order $2^{13}$.  It follows that $Q_{12}A_1A_2=O_2(N_H(E))$ and since $C \cap H$ is complement commuting with $a_1$, we additionally see that $N_H(E)/O_2(N_H(E))\cong \sym(3) \times \sym(3)$.

Now set $\hat{K}=K/Q_{12}$ and consider $M:=N_{K}({Q_{12}A_1A_2})\geq N_H(E)$. This has ${P}$ as a Sylow $3$-subgroup. We have seen in Lemma \ref{HN-swapping a_i's and z_i's-2} $(iii)$ that $N_H(P)$ acts as $\dih(8)$ on $\{A_1,B_1,A_2,B_2\}$. Thus the subgroup of $N_H(P)$ which preserves $\{A_1,A_2\}$ has index four. Using Lemmas \ref{HN-conjugates in P} and \ref{HN-Normalizer H of P} we therefore see that $|N_M(P)|=3^22^3$. We have, using Lemma \ref{HN-conjugates in P}, that a Sylow $2$-subgroup of $C_H(P)$ is elementary abelian and if an involution, $r$ say, in $C_H(P)$ normalizes $A_1$ and $A_2$ then it must be in $A_1$ and $A_2$ otherwise $\<A_1,r\>$ is a $2$-subgroup of $C_G(Z)$ normalized by $P$ which is not possible. Thus $C_M(P)=\<t\>$. Now we may apply Lemma \ref{HN-3'-subgroups of centralizers} together with coprime action to argue that $Q_{12}A_1A_2=O_{3'}(M)$ and so $\hat{A_1A_2}=O_{3'}(\hat{M})$. Now $|N_{\hat{M}}(\hat{P})|=3^22^2$  with elementary abelian Sylow $2$-subgroups (as seen in $N_H(E)$). We also see that  $C_M(p)\leq N_M(P)$ for any $p \in P^\#$. Thus $\hat{M}/\hat{A_1A_2}$ satisfies Theorem \ref{Hayden} and we may conclude that $\hat{M}/\hat{A_1A_2}$ has a normal Sylow $3$-subgroup. Thus $M=Q_{12}A_1A_2N_M(P)=N_H(E)$. Let $T$ be a Sylow $2$-subgroup of $N_H(E)$. Then $\hat{T}$ is a Sylow $2$-subgroup of the subgroup $\hat{D}$ in Lemma \ref{HN-alt5 and char 2} and therefore $\hat{A_1A_2}$ is characteristic in $\hat{T}$. Hence $T$ is a Sylow $2$-subgroup of $K$ and has order $2^{15}$.

We now show that in both Cases I and II that if $T\in \syl_2(N_H(E))$ then $Q_{12}$ is characteristic in $T$ to conclude that $T$ is a Sylow $2$-subgroup of $H$. We continue the notation that $\hat{K}=K/Q_{12}$ and use Lemma \ref{Prelims 2^8 3^2 Dih(8)} by considering the action of $\hat{T}$ on $\bar{Q_{12}}$. Now any involution in $\hat{A_1A_2}$ inverts a conjugate of $Z$ and so by Lemma \ref{HN-three things about K} $(ii)$ has centralizer of order $2^4$ in $\bar{Q_{12}}$. Now if $R$ is any elementary abelian normal subgroup of $\hat{T}$ then $R \cap \hat{A_1A_2}$ has order at least two and so for $R$ of order $2,4,8$ we have satisfied the requirements of Lemma \ref{Prelims 2^8 3^2 Dih(8)}. It remains to check that if $|R|=2^4$ then $|C_{\bar{Q_{12}}}(R)|\leq 2^3$. However we have seen in Lemma \ref{HN-alt5 and char 2} that such an $R$ must be conjugate to $\hat{A_1A_2}$ and now we may use Lemma \ref{HN-three things about K} $(iii)$ to see that $C_{\bar{Q_{12}}}(A_1)\neq C_{\bar{Q_{12}}}(A_2)$ and since each  $C_{\bar{Q_{12}}}(A_i)$ has centralizer in $\bar{Q_{12}}$ of order at most $2^4$,  we can conclude that  $|C_{\bar{Q_{12}}}(A_1A_2)|\leq 2^3$. Hence Lemma \ref{Prelims 2^8 3^2 Dih(8)} gives us that $Q_{12}$ is characteristic in $T$ and since $T\in \syl_2(N_G(K))$, we must have that $T \in \syl_2(H)$ and then by Lemma \ref{HN-Q_i's}, $T \in \syl_2(G)$.

Finally it is clear that $T' \leq Q_{12}A_1A_2$ and since $Q_{12}$ is characteristic in $T$, $N_G(T)\leq K$ and therefore, $\hat{A_1A_2} ~\mathrm{char}~ \hat{T} \unlhd \hat{N_G(T)}$. Thus $N_G(T)\leq N_H(E)$ and so, in Case II, is self-normalizing as claimed.

%Now we have that $T'\leq Q_{12}A_1A_2$. Now, recall that $r_1$ from Notation \ref{HN-Alt9notation}. The image of $r_1$ in $O^3(C_G(a_1))$ is $(1,3)(2,4)$ and $r_1\in 2\mathcal{A}$. Since any involution in  $Q_{12}A_1A_2\bs Q_{12}$ is in $2\mathcal{B}$ by Lemma \ref{HN-involutions in O^2(K)/Q}, we see that no $K$-conjugate of $r_1$ lies in $T'$. In the case that $C_G(a_i)\cong 3 \times \sym(9)$, we could choose an involution, $v$ say, in  $O^3(C_G(a_1))$ whose image is the transposition $(7,8)$ say. Then by Lemma \ref{HN-images in alt9}, $v$ and $vr$ are not in $2\mathcal{A}$ or $2\mathcal{B}$ and so no $K$-conjugate of $\<r_1,v\>$ intersects non-trivially with $T'$. Now we may apply Gr\"{u}n's Theorem to see that a Sylow $2$-subgroup of $K'$ is contained in $Q_{12}A_1A_2$.
\end{proof}

\begin{lemma}\label{HN-Indextwo}
In Case II, $G$ has proper normal subgroup $G_0$ which satisfies the hypotheses in Case I.
\end{lemma}
\begin{proof}
In Case II we have that $O^3(C_G(a_1))\cong \sym(9)$. We choose an involution, $r$ say, in  $O^3(C_G(a_1))$ whose image is the transposition $(7,8)$ and so $r$ is in the subgroup of $K$ described in Lemma \ref{HN-alt5 and char 2}. Let $T$ be a Sylow $2$-subgroup of $N_H(E)$. Then by Lemma \ref{HN-new-sylow2}, $T \in \syl_2(G)$ and $T'\leq Q_{12}A_1A_2$. Now, using Lemma \ref{HN-alt5 and char 2} $(iii)$ and Lemma \ref{HN-InvOrbs}, we see that every involution in $T'$ lies in $2\mathcal{A} \cup 2\mathcal{B}$. By Lemma \ref{HN-images in alt9}, $r$ is not in  $2\mathcal{A}$ or $2\mathcal{B}$ and so no $G$-conjugate of $r$ lies in $T'$. Now we may apply Gr\"{u}n's Theorem to see that a Sylow $2$-subgroup of $G'$ is equal to $\<N_G(T)',T \cap R' \mid R \in \syl_2(G)\>$. Hence $r \notin G'$ and $G/G'$ has even order. Now by Lemma \ref{HN-prelims1}, it is clear that $\<A_1,A_2,s\>\leq Q'$ and so we see that $G$ has a proper normal subgroup $G_0$ such that $N_{G_0}(Z)\sim 3^{1+4}: 4^{.}\alt(5)$. We must check that $Z$ is conjugate to $Z^x$ in $G_0$ however this is clear as $Z$ and $Z^x$ are conjugate in $\<Q,Q^x\>\leq G_0$.
\end{proof}

In light of Lemma \ref{HN-Indextwo}, we may simplify our working significantly by assuming from now on that we are in Case I only and so by Lemma \ref{HN-InvOrbs}, $K/Q_{12}\cong \alt(5) \wr 2$.

\begin{lemma}\label{HN-K is strongly 3 embedded}
$K$ is strongly $3$-embedded in $H$.
\end{lemma}
\begin{proof}
Let $h \in H$ and $y\in K \cap K^h$ be an element of order three. By Lemma \ref{HN-three things about K}, the centralizer in $H$ of every element of order three in $K$ is contained in $K$. Thus $C_H(y) \leq K \cap K^h$.
Therefore $K \cap K^h$ contains a Sylow $3$-subgroup of $H$. So assume $P\leq K \cap K^h$. Then
${Q_{12}}=O_2(K)=\prod_{p\in P^\#}O_2(C_H(p))=O_2(K^h)={Q_{12}}^h$. Therefore $h \in
N_G({Q_{12}})=K$ and so $K=K^h$.
\end{proof}

%\begin{lemma}\label{HN-2's which invert P}
%Any $2$-element that inverts $P$ is in $2\mathcal{B}$.
%\end{lemma}
%\begin{proof}
%Consider $N_G(\<a_1\>)$ which is isomorphic to the diagonal subgroup of $\sym(3)\times \sym(9)$. An
%involution in $N_G(\<a_1\>)$ that inverts $a_1$ has an image in $N_G(\<a_1\>)/\<a_1\>$ which is
%odd. By Lemma \ref{HN-images in alt9}, we may identify the image of $P$ in $N_G(\<a_1\>)/\<a_1\>$
%with a cyclic subgroup generated by an element of cycle type $3^2$. Let $v$ be an involution
%inverting $P$. Then the image of $v$ in $\sym(9)$ has cycle type $2^3$ and therefore centralizes an
%element of cycle type $3^3$. Such an element of order three is the cube of an element of order nine
%in $\sym(9)$. Thus, in $N_G(\<a_1\>)$ we have that $v$ inverts $P$ whilst centralizing an element
%of order three in $3\mathcal{B}$ (since elements in $3\mathcal{A}$ are not cubes of elements of
%order nine in $G$ by Lemma \ref{elements of order nine in S}). Any involution which commutes with
%an element of order three in $3\mathcal{B}$ is necessarily conjugate to $t$ and so in
%$2\mathcal{B}$.
%\end{proof}

Recall we fixed an involution $r_1 \in C_H(a_1)$ in Notation \ref{HN-Alt9notation}.
\begin{lemma}\label{HN-Transfer-O^2(H) is proper}
$r_1$ is not in $O^2(H)$. In particular, $H\neq O^2(H)$ and $O^2(H) \cap K\sim 2_+^{1+8}.(\alt(5)\times \alt(5))$.
\end{lemma}
\begin{proof}
Given the cycle type of the images of $r_1$ and $t$ in $\alt(9)\cong O^3(C_G(a_1))$ and by Lemma \ref{HN-images in alt9}, we see that $r_1$ is not conjugate to $t$ in $G$ however the product $r_1t$
is conjugate to $t$ in $O^3(C_G(a_1))$ and therefore $r_1$ is not conjugate to $r_1t$ in $G$.

Observe that $r_1$ inverts $a_2$ therefore $r_1\notin {Q_{12}}$. Since  $r_1$ centralizes $a_1$ whilst inverting $a_2$, we have that $r_1$ permutes
$\<z_1\>$ and $\<z_2\>$ and therefore permutes $\<A_1,B_1\>$ and $\<A_2,B_2\>$ and so $r_1\notin O^2(K)$. Let
$T\in \syl_2(K)$ such that $r_1\in T$ and suppose that for some $h \in H$,
$r_1^h\in O^2(K) \cap T$. Suppose that $r_1^h \in {Q_{12}}$. Then $\<r_1^h,t\>\vartriangleleft
{Q_{12}}$ but is not central in ${Q_{12}}$ as $Q_{12}$ is extraspecial. Therefore $r_1^h$ is
conjugate to $r_1^ht=(r_1t)^h$ in ${Q_{12}}$ and so $r_1$ is conjugate to
$r_1t$ which is a contradiction. So $r_1^h \notin {Q_{12}}$. %Suppose $3\mid |C_K(r_1^h)|$. Let $D\in
%\syl_3(C_K(r_1^h))$ then $D^\# \subset 3\mathcal{A}$ and $D^{h\inv}\syl_3(C_{K^{h\inv}})(r_1)$. By
%Sylow's Theorem we may find $g \in C_H(r_1)$ such that $D^{h\inv g} \in \syl_3(C_K(r_1))$ Since $K$ is
%a strongly $3$-embedded subgroup of $H$, $K^{h\inv g}=K$. It is clear that $K=N_H(K)$ (as
%$Q=O_2(K)$) and so $g\inv h \in K$. However $r_1 \notin O^2(K)$ and so $r_1^{g\inv h}\notin O^2(K)$
%which is a contradiction since $r_1^{g\inv h}=r_1^h \in O^2(K)$. Therefore $3\nmid |C_K(r_1^h)|$.
So consider $Q_{12} \neq {Q_{12}}r_1^h$. By Lemma \ref{HN-alt5 and char 2}, either $r_1^h
\in 2 \mathcal{B}$ or has order four. However $r_1$ is an involution and is not conjugate to $t$ in
$G$ and so we have a contradiction.

Thus no $H$-conjugate of $r_1$ lies in $T \cap O^2(K)$ which is a maximal subgroup of $T\in
\syl_2(H)$. By Thompson Transfer, $r_1 \notin O^2(H)$ and so $H \neq
O^2(H)$. Since $[K:O^2(K)]=2$, we must have $O^2(K)=O^2(H) \cap K\sim 2_+^{1+8}.(\alt(5)\times
\alt(5))$.
\end{proof}

\begin{lemma}\label{HN-orbits of elements in Q}
Let $f \in {Q_{12}}\bs \<t\>$. Then either $f$ has order four or one of the following occurs.
\begin{enumerate}[$(i)$]
 \item $f\in 2\mathcal{B}$, $C_H(f)\leq K$ has order
 $2^{13}3$ and $\bar{f}$ is $2$-central in $\bar{K}$.
  \item $f\in 2\mathcal{A}$, $|C_K(f)|=2^{11}35$ and
  $C_{K}({f}){Q_{12}}/{Q_{12}}\cong \alt(5) \times 2$ or $\sym(5)$.
 \end{enumerate}
\end{lemma}
In particular, $K$ acts irreducibly on $\bar{Q_{12}}$, $C_H(f) \cap 3\mathcal{A} \neq 1$ and if
$\bar{f} \in \mathbf{Z}(\bar{T})$ then $f\in 2\mathcal{B}$ and $C_H(f) \leq K$.
\begin{proof}
Lemma \ref{HN-InvOrbs} tells us that every involution in $Q_{12}\bs \{t\}$ lies in one of two $K$-conjugacy  classes. Meanwhile, using Lemma \ref{HN-centralizer of F}  we see that if such an involution $f \in 2\mathcal{B}$ then $C_H(f)\leq K$ and lies in a $K$-orbit of length 150 which means that $|C_H(f)|=2^{13}3$ and so $\bar{f}$ is $2$-central in $\bar{K}$. If  $f \in 2\mathcal{A}$ then $f$ commutes with a diagonal subgroup of $K/Q_{12}$ isomorphic to $\alt(5)$ and lies in a $K$-orbit of length 120. It follows from the structure of the maximal subgroups of $\alt(5)\wr 2$ that $C_{K/Q_{12}}(f)\cong 2 \times \alt(5)$ or $\sym(5)$.

We now suppose that $f \in Q_{12}$ has order four. In Lemma \ref{HN-InvOrbs} we saw that an element of order four in $Q_{12}$ also commutes with a diagonal subgroup of $K/Q_{12}$ isomorphic to $\alt(5)$ and so lies in a $K$-orbit of length 120 or 240. Suppose there is more than on $K$-orbit of elements of order four. Any $K$-orbit has length a multiple of 30 (because $Q_{12}$ is non-abelian and because there exist elements of order three and five which act fixed-point-freely on $\bar{Q_{12}}$). However no orbit can have length 30 or 60 because $K/Q_{12}$ has no subgroups of order $2^535$ or $2^435$. Since $Q_{12}$ has 240 involutions we have either one orbit of length 240 or two of length 120. In particular, no element of order four is $2$-central in $\bar{K}$.

Now if $f$ is any element in $Q_{12}$ then $1\neq \bar{f} \in \bar{{Q_{12}}}$ commutes with an element of order three in
$\bar{K}$. Since each $z_i$ acts fixed-point-freely on $\bar{{Q_{12}}}$, we have that $\bar{f}$ is
centralized by a conjugate of ${Q_{12}}a_i$. Therefore $f$ commutes with a conjugate of $a_i$.
Furthermore we observe that if $f$ has order four or $f\in 2\mathcal{A}$ then $\bar{f}$ is not
$2$-central in $\bar{K}$ whereas if $f\in 2\mathcal{B}$ then $\bar{f}$ is $2$-central in $\bar{K}$.
Finally, suppose that $ W<Q_{12}$ with $t \in W\vartriangleleft K$. Then $W$ must be a union of
$K$-orbits. However the $K$-orbits on $Q_{12}$ have lengths in $\{1, 150,120,240\}$ and no union of
orbits  is a power of $2$ greater than $2$ and less than $2^9$. Thus $K$ acts irreducibly on
$\bar{Q_{12}}$.
\end{proof}

\begin{lemma}\label{HN-Q=Q^h}
Let $h \in H$. If $({Q_{12}} \cap {Q_{12}}^h)\bs \<t\>$ contains an involution in $2\mathcal{B}$
then ${Q_{12}}={Q_{12}}^h$.
\end{lemma}
\begin{proof}
We may suppose that for some $1 \neq \bar{f} \in \mathbf{Z}(\bar{T})$,  $\bar{f} \in \bar{{Q_{12}}} \cap
\bar{{Q_{12}}}^h$. By Lemma \ref{HN-orbits of elements in Q}, $f\in 2\mathcal{B}$ and $C_H(f) \leq
K$ and also $C_H(f) \leq K^h$. However this implies that $3\mid |K \cap K^h|$ and so $K=K^h$ and
${Q_{12}}={Q_{12}}^h$ by Lemma \ref{HN-K is strongly 3 embedded}.
\end{proof}

\begin{lemma}
Let $T \in \syl_2(K)$. Then $\bar{{Q_{12}}}$ is strongly closed in $\bar{T}$ with respect to $\bar{H}$.
\end{lemma}
\begin{proof}
Let $1\neq \bar{f}\in \bar{{Q_{12}}}$ such that $\bar{f} \in \bar{T^h}\bs \bar{{Q_{12}^h}}$ for
some $h \in H$. Since $f \in Q_{12}\leq O^2(K)\leq O^2(H)$, we must have that $f\in O^2(K^h)=O^2(H)
\cap K^h$. By Lemma \ref{HN-alt5 and char 2} $(iii)$ applied to $K^h$,  either $f$ is an element of
order four squaring to $t$ and commuting with a conjugate of $Z$ or $f\in 2\mathcal{B}$ and
$C_{\bar{{Q_{12}}}^h}({f})=[\bar{{Q_{12}}}^h,{f}]$ has order $2^4$.

Suppose first that $f$ has order four. Then $f^2=t$ and ${Q_{12}}^hf$ is an involution in
$O^2(K/{Q_{12}})^h$. By Lemma \ref{HN-alt5 and char 2} $(iii)$, $C_G(f)$ contains a conjugate of
$Z$ and then by Lemma \ref{HN-element of order four in CG(Z)}, a Sylow $3$-subgroup of $C_G(f)$ is
conjugate to $Z$. However, by Lemma \ref{HN-orbits of elements in Q},  since $f \in Q_{12}$, $C_H(f) \cap 3\mathcal{A}
\neq 1$ which is a contradiction.

So we suppose instead that $f$ is an involution. Then $Q_{12}^hf$ is a diagonal involution with $f \in 2 \mathcal{B}$ and it follows from Lemma \ref{HN-InvOrbs} $(iii)$ that $f$ commutes with an element of order four in $Q_{12}^h$.  Now, by Lemma
\ref{HN-orbits of elements in Q}, $C_H(f)\leq K$. The element of order four in $Q_{12}^h$ commuting with $H$ is necessarily in $Q_{12}$ else we have a contradiction as before. Let $D,V\leq K^h$ such that $\bar{D}:=C_{\bar{K^h}}({f})$ and $\bar{V}:=C_{\bar{{Q_{12}^h}}}({f})$. Then we have seen that $V \cap Q_{12}>\<t\>$. By Lemma \ref{lem-conjinvos}, since $\bar{V}=[\bar{Q_{12}},f]$, $|C_{\bar{K}^h}({f})|=|\bar{V}||C_{\bar{K^h}/\bar{Q_{12}^h}}({f})|=2^9$. Thus $\bar{Q_{12}^h D}$ is a Sylow $2$-subgroup of $\bar{K}^h$.

Since $1\neq \bar{V} \cap \bar{Q_{12}}\unlhd \bar{D}$, we get that $\bar{V}\cap \bar{Q_{12}}\cap \mathbf{Z}(\bar{D})\neq 1$. However, $\bar{V}\cap \mathbf{Z}(\bar{D})\leq \mathbf{Z}(\bar{Q_{12}^hD})$ the preimage of which contains only involutions in $2\mathcal{B}$. Therefore $V \cap Q_{12}\leq Q_{12}^h \cap Q_{12}$ contains involutions in $2\mathcal{B}$ distinct from $t$. This contradicts Lemma \ref{HN-Q=Q^h}. Thus $\bar{{Q_{12}}}$ is strongly closed in
$\bar{T}$ with respect to $\bar{H}$.
\end{proof}

\begin{lemma}
$K=H$.
\end{lemma}
\begin{proof}
Assume for a contradiction that $K<H$ then ${Q_{12}} \ntriangleleft H$. Consider $O_{3'}(H)$. By
Lemma \ref{HN-orbits of elements in Q}, the only proper non-trivial subgroup of $Q_{12}$ which is normalized by $K$ is $\<t\>$. So we have that $O_{3'}(H)\cap K\leq O_{3'}(H)\cap Q_{12}=\<t\>$. Since $O_{3'}(H)$
is normalized by $P$, by coprime action, $O_{3'}(H)$ is generated by elements commuting with
elements of $P^\#$. However by Lemma \ref{HN-three things about K}, for every $p \in P^\#$, $C_H(p) \leq K$. Therefore $O_{3'}(H)\leq K$ and so
$O_{3'}(H)=\<t\>$.

Set $M:=\<{Q_{12}}^H\>\trianglelefteq H$ then $M \leq O^2(H)$. Moreover $O_{3'}(M)\leq O_{3'}(H)$
and so $O_{3'}(M)=\<t\>$.  Therefore we have $P \cap M\neq 1$. Now $M \cap K$ is a normal subgroup of $K$ and contained in $O^2(H)\cap K=O^2(K)$. Hence $M \cap K=O^2(K)$.

Set $N:=O_{2'}(M)$. If $N$ is $3'$ then $N \leq O_{3'}(H) =\<t\>$ and so $N=1$. Otherwise $P \cap
N\neq 1$ and then $[P\cap N,{Q_{12}}] \leq N \cap {Q_{12}}=1$ which is a contradiction as $C_G(Q_{12})\leq Q_{12}$. Therefore
$O_{2'}(M)=1$. Now, since $P \leq M\vartriangleleft H$, $H=MN_H(P)$ by a Frattini argument and so
$M=\<Q_{12}^H\>=\<Q_{12}^{N_H(P)M}\>=\<Q_{12}^M\>$ since $N_H(P)\leq K=N_G(Q_{12})$. Finally, we
may apply Theorem \ref{goldschmidt} to $\bar{M}=\<\bar{{Q_{12}}}^M\>$. As required, we have that
for $T\in \syl_2(K)$,  $\bar{{Q_{12}}}$ is strongly closed in $\bar{T}$ with respect to
$\bar{H}$. Hence $\bar{{Q_{12}}}$ is strongly closed in $\bar{M} \cap \bar{T}$ with respect
to $\bar{M}$. We have also that $O_{2'}(\bar{M})=1$. Thus $\bar{{Q_{12}}}=O_2(\bar{M})\Omega(\bar{T} \cap \bar{M})$. Since ${Q_{12}}$ is
not a Sylow $2$-subgroup of $M \leq O^2(H)$ we may find $e \in (M \cap T)\bs {Q_{12}}$. Then by
Lemma \ref{HN-alt5 and char 2} $(iii)$, ${Q_{12}}e$ contains either involutions or elements of
order four squaring to $t$. In either case $\bar{{Q_{12}}}\bar{e} \cap \Omega(\bar{T} \cap
\bar{M})\neq 1$ and so ${Q_{12}} \nleq \Omega(\bar{T} \cap \bar{M})$. This contradiction proves
that $H=K$.
\end{proof}

\section{The Structure of the Centralizer of $u$}\label{HN-Section-CG(u)}

We continue to assume that we are in Case I only. We now know the structure of the centralizer of an involution in $G$-conjugacy class $2\mathcal{B}$
and so we must determine the structure of the centralizer of an involution in $2\mathcal{A}$.
We continue notation from Section \ref{HN-Section-CG(t)}. Recall that in Notation \ref{HN-Alt9notation} we fixed an  involution $u \in Q_2\leq C_G(a_2)$ and we defined
$2\mathcal{A}$ to be the conjugacy class of involutions in $G$ containing $u$. By Lemma \ref{HN-images in alt9}, $2\mathcal{A}\neq 2 \mathcal{B}$. Let $L:=C_G(u)$ and
$\wt{L}=L/\<u\>$ and we continue to set $H=C_G(t)$ and $\bar{H}=H/\<t\>$. We will show that $L\sim
(2^{\cdot}\HS):2$ and so we must identify that $\wt{L}$ has an index two subgroup isomorphic to the
sporadic simple group $\HS$. We first show that $\wt{L}$ has a subgroup $2\times \sym(8)$ and later
that the centre of this subgroup does not live in $O^2(\wt{L})$. We will use the information we
have about $C_G(t)=H$ and $N_G(E)$ to see the structure of some $2$-local subgroups of $\wt{L}$.
Once we have used extremal transfer to find the index two subgroup of $\wt{L}$ we are then able to
use this $2$-local information to apply a theorem due to Aschbacher \cite{AschbacherHS} to
recognize $\HS$. The Aschbacher result requires us to find $2$-local subgroups of shape $(4 *2_+^{1+4}).\sym(5)$ and $(4^3). \GL_3(2)$

Recall using  Notation \ref{HN-Alt9notation} that $u \in F \leq Q_2 \leq C_G(a_2)$ and that $a_1$
normalizes $F$.

\begin{lemma}\label{HN-fours groups centralize A8}
$C_G(F)\cong 2\times 2 \times \alt(8)$ with $C_G(F)>C_G(u) \cap C_G(a_1)\cong \alt(8)$ and
$C_{\wt{L}}(\wt{F})\cong 2 \times \sym(8)$. Moreover if $F_0$ is any fours subgroup of $C_G(a_2)$
such that  $F_0^\# \subseteq 2\mathcal{A}$ then $C_G(F_0)\cong C_G(F)$.
\end{lemma}
\begin{proof}
Set $M:=C_G(F)$. First observe that $F \leq O^3(C_G(a_2))\cong \alt(9)$ and the image of $F^\#$ in
$\alt(9)$ consists of involutions of cycle type $2^2$. Notice also that $\alt(9)$ has two classes
of such fours groups with representatives $\<(1,2)(3,4),(1,3)(2,4)\>$ and
$\<(1,2)(3,4),(3,4)(5,6)\>$. These subgroups of $\alt(9)$ have respective centralizers isomorphic
to $2^2 \times \alt(5)$ and $2^2 \times \sym(3)$ and respective normalizers
 $(\alt(4) \times \alt(5)):2$ and $\sym(4) \times \sym(3)$.

Given the image of $F$ in $O^3(C_G(a_2))$, we have that $M\cap C_G(a_2)\cong 3 \times 2^2 \times
\sym(3)$. Let $R \in \syl_3(M \cap C_G(a_2))$ such
that $\<R,a_1\>$ is a Sylow $3$-subgroup of $N_G(F)$. Then $a_2 \in R$ and  $\<R,a_1\>$ is abelian and $R^\# \subseteq
3\mathcal{A}$ since no element of order three in $3\mathcal{B}$ commutes with a fours group.
Therefore by the earlier argument for each $r \in R^\#$, $C_G(r) \cap M\cong 3\times 2^2
\times \alt(5)$ or $3\times 2^2 \times \sym(3)$.

Consider $M \cap C_G(a_1)$ which is isomorphic to a subgroup of $\alt(9)\cong O^3(C_G(a_1))$. Notice that $F$ does not commute
with $O^3(C_G(a_1))$ (for we would then have $F$ commuting with an element of $3\mathcal{B}$ and such elements do not commute with a fours group) and so $M \cap C_G(a_1)$ is a proper subgroup of $O^3(C_G(a_1))$. By Lemma
\ref{HN-Q_i's}, we  have that $F\leq Q_2$ commutes with $Q_1\leq C_G(a_1)$. Also $F$ commutes with
$R\leq C_G(a_1)$ and so $|M \cap C_G(a_1)|$ is a multiple of $2^53^2$. Moreover $M \cap C_G(a_1)$ contains the subgroup $Q_1\<a_2\>\sim 2^{1+4}_+.3$.

We check the maximal subgroups of $\alt(9)$ (see \cite{atlas}) to see that $M \cap C_G(a_1)$ is isomorphic to either a subgroup of $\alt(8)$ or the diagonal subgroup of index two in $\sym(5) \times \sym(4)$.
The latter possibility leads to a Sylow $2$-subgroup of order $2^5$ with centre of order four which
is impossible as $2_+^{1+4}\cong Q_2\leq M\cap C_G(a_1)$. So $M\cap C_G(a_1)$ is isomorphic to a
subgroup of $\alt(8)$. Suppose it is isomorphic to a proper subgroup of $\alt(8)$. We again check
the maximal subgroups of $\alt(8)$ (\cite{atlas}) to see that $M\cap C_G(a_1)$ is isomorphic to a
subgroup of $N_{\alt(8)}(\<(1,2)(3,4),(1,3)(2,4),(5,6)(7,8),(5,7)(6,8)\>)\sim 2^4:(\sym(3) \times
\sym(3))$. This subgroup can be seen easily in $\GL_4(2)$ as the subgroup of matrices of shape
\[\left(\begin{array}{cccc}
\ast & \ast & 0 & 0 \\
\ast & \ast & 0 & 0 \\
\ast & \ast & \ast & \ast \\
\ast & \ast & \ast & \ast
\end{array}\right).\] We calculate in this group that an extraspecial subgroup of order $2^5$ is not normalized by an element of order three. Therefore  $M \cap C_G(a_1)$ is not isomorphic to a subgroup of this matrix group. Thus $M\cap C_G(a_1)\cong \alt(8)$. In particular $M$ has a subgroup isomorphic to
$2^2 \times \alt(8)$.

Now we have that for every $r \in R^\#$, $C_M(r)\cong 3\times 2^2 \times \sym(3)$ or
$3\times 2^2 \times \alt(5)$. Now $R\leq C_M(a_1)\cong \alt(8)$ and so $R \in
\syl_3(C_M(a_1))$.  Moreover, $\alt(8)$ has two conjugacy classes of elements of order three. So we
may set $R=\{1,a_2,a_2^2,a_3,a_3^2,b_1,b_1^2,b_2,b_2^2\}$ where $a_2$ is conjugate to $a_3$ in
$C_M(a_1)$ and $b_1$ is conjugate to $b_2$ in $C_M(a_1)$ such that $C_{C_M(a_1)}(b_i)\cong 3 \times
\alt(5)$ ($i \in \{1,2\}$) and $C_{C_M(a_1)}(a_j)\cong 3 \times \sym(3)$ ($j \in \{2,3\}$). Now we
already have that $C_M(a_3) \cong C_M(a_2)\cong 3 \times 2^2 \times \sym(3)$ and we have two
possibilities for the structure of the other $3$-centralizer. Therefore we must have that
$C_{M}(b_i)\cong 3 \times 2 \times 2 \times \alt(5)$. Now by coprime action $C_{M/F}(Fb_i)\cong
C_M(b_i)/F$ and $C_{M/F}(Fa_i)\cong C_M(a_i)/F$. Hence we may apply Corollary
\ref{Cor-ParkerRowleyA8} to $M/F$ to say that $M/F\cong \alt(8)$. Therefore $M\cong 2^2
\times \alt(8)$.

Consider $N_L(F)$. We have seen that $N_G(F)/M\cong \sym(3)$ and so $[N_L(F):M]=2$. It follows that
$N_L(F)/F\cong 2 \times \alt(8)$ or $\sym(8)$. For $b_1\in R$, $C_{M}(b_1)\cong 3 \times 2^2
\times \alt(5)$ and so  $C_{N_L(F)}(b_1)\sim 3 \times (2^2 \times \alt(5)):2$ and
$C_{N_L(F)}(b_1)/F\sim 3 \times \sym(5)$ which is not a subgroup of $\alt(8) \times 2$. Thus we
must have that $N_L(F)/F\cong \sym(8)$ and so $C_{\wt{L}}(\wt{F})\cong 2 \times \sym(8)$.

Now let $F_0\leq C_G(a_2)$ have image $\<(1,2)(3,4),(1,3)(2,4)\>$ in $\alt(9)\cong O^3(C_G(a_2))$. Then
$C_G(a_2) \cap C_G(F_0)\cong 3 \times 2^2 \times \alt(5)$. Now recall that $R\in \syl_3(M)$
and $M=C_G(F)\cong 2^2 \times \alt(8)$ and so there exists $r \in R^\#$ such that $M \cap
C_G(r)\cong 3 \times 2^2 \times \alt(5)$. Since every element in $R^\#$ is conjugate in $G$,
we have that $F_0$ is conjugate to $F$ in $G$. Thus $C_G(F_0)\cong C_G(F)$.
\end{proof}

Recall from Notation \ref{HN-Alt9notation} that $r_2$ is an involution in $O^3(C_G(a_2))$ which is
conjugate to $u$ and $r_2u$. In light of Lemma \ref{HN-fours groups centralize A8}, the following
result is a calculation in a group isomorphic to $2 \times 2 \times \alt(8)$.
\begin{lemma}\label{HN-HS-centralizer of r,t,u}
$O_2(C_{H \cap L}(r_2))$ has order $2^6$.
\end{lemma}
\begin{proof}
It is clear from Notation \ref{HN-Alt9notation} that $\<r_2,u\>^\# \subseteq 2\mathcal{A}$. Set
$F_0:=\<r_2,u\>$ then by Lemma \ref{HN-fours groups centralize A8}, $C_G(F_0)\cong 2\times 2 \times
\alt(8)$. Notice also from Notation \ref{HN-Alt9notation} that $t \in C_G(F_0) \cap C_G(a_2)\cong 3
\times 2^2 \times \alt(5)$ which has an abelian subgroup containing $t$ isomorphic to $3
\times 2^4$. Consider $\<F_0,t\> \cap C_G(F_0)'$ (of course
$C_G(F_0)'\cong \alt(8)$) which has order two. If $\<F_0,t\> \cap C_G(F_0)'$ is $2$-central in
$C_G(F_0)'$ then $a_2 \in C_G(F_0)' \cap C_G(t)$ is isomorphic to the subgroup of $\alt(8)$ of
shape $2_+^{1+4}.\sym(3)$. However this implies that $C_G(\<F_0,t\>) \cap C_G(a_2)\cong 2 \times 2
\times 2 \times 3$ which is not the case. Thus $\<F_0,t\> \cap C_G(F_0)'$ is not $2$-central in
$C_G(F_0)'$ and so $C_G(F_0)' \cap C_G(t)$ is isomorphic to a subgroup of $\alt(8)$ of shape $(2^2\times \alt(4)) :2$. Thus $C_{H \cap L}(r_2)\sim 2^2\times (2^2 \times
\alt(4)) :2$ and so the order of the $2$-radical is clear.
\end{proof}

\begin{lemma}\label{HN-HS-Sylow 2}
$H\cap L$ contains a Sylow $2$-subgroup of $L$ which has order $2^{11}$ and centre $\<t,u\>$.
\end{lemma}
\begin{proof}
Let $S_u$ be a Sylow $2$-subgroup of $C_L(t)$. We have that $u \in Q_2 \leq Q_{12}$ and since $u
\in 2\mathcal{A}$,  we may apply Lemma \ref{HN-orbits of elements in Q} to see that
$|C_{{H}}({u})|=2^{11}.3.5$. Therefore $|S_u|=2^{11}$. Now, $u \in Q_2$ and $[Q_1,Q_2]=1$ (by Lemma
\ref{HN-Q_i's}) so  we have that $Q_1\leq C_{O_2(H)}(u) \leq S_u$. Moreover, $\mathbf{Z}(S_u)\leq
C_{S_u}(Q_1) \leq Q_2$. Therefore $\mathbf{Z}(S_u)\leq \mathbf{Z}(C_{Q_2}(u))=\<t,u\>$ since
$Q_2$ is  extraspecial of order $2^5$. Hence $\mathbf{Z}(S_u)=\<t,u\>$. Since $\<t,u\>\leq Q_{12}$
and $Q_{12}$ is extraspecial, $u$ is conjugate to $ut$ in $Q_{12}$. Therefore $N_G(\<t,u\>)\leq
C_G(t)$. So let $S_u \leq T_u\in \syl_2(H \cap L)$ then $N_{T_u}(S_u) \leq N_{L}(\<t,u\>) \leq H
\cap L$. Thus $S_u$ is a Sylow $2$-subgroup of $L$.
\end{proof}

\begin{lemma}
$(H \cap L)/({Q_{12}} \cap L) \cong \sym(5)$.
\end{lemma}
\begin{proof}
Using Lemma \ref{HN-orbits of elements in Q} we have that $C_H(u)/C_{Q_{12}}(u)\cong \alt(5) \times
2$ or $\sym(5)$. We suppose for a contradiction that $(H \cap L) /(Q_{12} \cap
L)=C_H(u)/C_{Q_{12}}(u)\cong C_H(u){Q_{12}}/{Q_{12}}\cong 2 \times \alt(5)$. Now set
$V:=C_{Q_{12}}(u)$ then $|V|=2^8$ and $\bar{V}$ is normalized by $C_H(u)/V\cong 2 \times \alt(5)$.

Recall from Notation \ref{HN-Alt9notation} that $r_2$ is an involution in $O^3(C_G(a_2))$ and from
Lemma \ref{HN-alt9 observations} that $r_2\in C_H(a_2)\bs  Q_2$. Since $[r_2,a_2]=1$,
$[Vr_2,Va_2]=1$ and therefore $Vr_2\in \mathbf{Z}(C_H(u)/V)$. In particular, $C_{\bar{V}}(r_2)$ is
preserved by $O^2(C_H(u)/V)\cong \alt(5)$. Since $Va_2$ acts non-trivially on $\bar{V}$,
$O^2(C_H(u)/V)$ acts non-trivially. This is to say that there exists a non-central
$O^2(C_H(u)/V)$-chief factor of $V$. Moreover, this chief factor has order $2^4$.

Now Lemma
\ref{HN-HS-centralizer of r,t,u}  gives us that $|O_2(C_{H \cap L}(r_2))|=2^6$. Clearly $C_V(r_2)$ is a normal $2$-subgroup of $C_{H \cap L}(r_2)$ as is $\<r_2\>\nleq V$. Therefore $|C_V(r_2)|\leq 2^5$ and so
$|C_{\bar{V}}(r_2)|=2^4$ or $2^5$. Suppose first that $|C_{\bar{V}}(r_2)|=2^4$ then $\bar{u} \in
C_{\bar{V}}(r_2)$ is normalized by $O^2(C_H(u)/V)$ and so $C_{\bar{V}}(r_2)$ is necessarily a sum of
trivial $O^2(C_H(u)/V)$-modules. Moreover $\bar{V}/C_{\bar{V}}(r_2)$ has dimension three and is
therefore also a sum of trivial $O^2(C_H(u)/V)$-modules. This is a contradiction.

So suppose instead that $|C_{\bar{V}}(r_2)|=2^5$. Then $|[\bar{V},r_2]|=2^2$. Furthermore $[\bar{V},r_2]$  is preserved by $O^2(C_H(u)/V)$. Thus
$[\bar{V},r_2]$ is a sum of two trivial $O^2(C_H(u)/V)$-modules. Since $[\bar{V},r_2]\leq
C_{\bar{V}}(r_2)$, it follows that $C_{\bar{V}}(r_2)$ is also
a sum of trivial $O^2(C_H(u)/V)$-modules as is $\bar{V}/C_{\bar{V}}(r_2)$. Again this gives us a
contradiction. Hence we may conclude that $C_H(u)/C_{Q_{12}}(u)\cong C_H(u){Q_{12}}/{Q_{12}}\cong
\sym(5)$.
\end{proof}
%
%\begin{lemma}\label{HN-HS-element of order five}
%Let $(L \cap {Q_{12}})e \in (L \cap H)/(L \cap {Q_{12}})$ have order five then
%$C_{Q_{12}}(e) \cong [{Q_{12}},e] \cong 2_{-}^{1+4}$.
%\end{lemma}
%\begin{proof}
%By Lemma \ref{HN-Q_i's}, $C_G({Q_{12}})\leq {Q_{12}}$ and so  $e$ acts non-trivially on $Q_{12}$ and since $(L \cap {Q_{12}})e $ has order five, $e$ describes an automorphism of $Q_{12}$ of order five. We have that $e$ centralizes $u$ and so $C_{Q_{12}}(e)>\<t\>$. Hence by Lemma \ref{prelims-extraspecial and a coprime aut}, $C_{Q_{12}}(e)$ and
%$[{{Q_{12}}},e]$ are both extraspecial with intersection equal to $\<t\>$ and product equal to $Q_{12}$. Since
%$e$ acts fixed-point-freely on $[\bar{{Q_{12}}},e]$, we have that $|[\bar{{Q_{12}}},e]|=2^4$. Thus $C_{Q_{12}}(e)$ and $[{{Q_{12}}},e]$ are both extraspecial of order $2^5$. Since
%$[{{Q_{12}}},e]$ has an automorphism of order five, $[{{Q_{12}}},e]\cong 2_-^{1+4}$ follows from
%Lemma \ref{extraspecial outer automorphisms}. Finally, since ${Q_{12}}$ is extraspecial of plus
%type, we have that $C_{Q_{12}}(e)\cong 2_-^{1+4}$.
%\end{proof}

\begin{lemma}\label{HN-HS-1-finding 4*4*4}
There exists an element of order four $d\in C_{Q_2}(u)$ such that $d^2=t$ and $4 \times 2\cong
{\<d,u\>}\vartriangleleft {H \cap L}$.
\end{lemma}
\begin{proof}
This is clear once we recall Lemma \ref{HN-InvOrbs} $(iii)$ which tells us that diagonal subgroups of $H/Q_{12}$ isomorphic to $\alt(5)$ which centralize an involution in $\bar{Q_{12}}$, in fact centralize a subgroup of $Q_{12}$ isomorphic to $C_4 \times C_2$. We have that $(H\cap L)Q_{12}/Q_{12}\cong (H \cap L)/({Q_{12}} \cap L) \cong \sym(5)$ contains such a subgroup. Finally we must check that the element of order four in in $Q_2$ however this is clear as $a_2\in H \cap L$ and $C_{Q_{12}}(a_2)=Q_2$.
\end{proof}

\begin{lemma}\label{HN-HS-complement to A}
There exists a complement $C\cong \GL_3(2)$ to $C_L(E)$ in $N_L(E)$ such that $EC \leq C_G(F)$.
\end{lemma}
\begin{proof}
Recall that $u \in F \leq Q_2$ and by Lemma \ref{HN-fours groups centralize A8}, $2 \times 2 \times
\alt(8)\cong C_G(F)> C_G(u) \cap C_G(a_1)\cong \alt(8)$. Notice that $E \leq Q_1 \leq C_G(F)$ since $[Q_1,Q_2]=1$. Notice also that $t\in C_G(u)\cap C_G(a_1)$.
From notation \ref{HN-Alt9notation}, the image of $t$ in $\alt(9)\cong O^3(C_G(a_1))$ is
$(1,2)(3,4)(5,6)(7,8)$ and so clearly $t$ lies in exactly one subgroup of $O^3(C_G(a_1))$
isomorphic to $\alt(8)$.  By Lemma \ref{HN-alt9 observations}, $O^3(C_G(a_1))$ contains a complement,
$C$ say, to $C_G(E)$ in $N_G(E)$. Moreover the image of $EC$ in $O^3(C_G(a_1))$ lies in a subgroup
isomorphic to $\alt(8)$ containing $t$. Therefore $EC\leq C_G(u)\cap C_G(a_1)\leq C_G(F)$.
\end{proof}

\begin{lemma}\label{HN-HS-index two subgroup}
We have that $L=FO^2(L)$ with $[L:O^2(L)]=2$, $\wt{N_{O^2(L)}(E)}\cong
4^3:\GL_3(2))$ and $\wt{C_{O^2(L)}(t)}\sim 2_+^{1+4}*4. \sym(5)$.
\end{lemma}
\begin{proof}
By Lemma \ref{HN-alt9 observations} $(v)$, a Sylow $3$-subgroup of $C_G(E)/E$ is self-centralizing
in $G$ and therefore $3 \nmid |C_G(u) \cap C_G(E)|$. Hence $C_L(E)$ is a $2$-group. Notice that $|C_L(E)|\leq 2^8$ as, by Lemma \ref{HN-HS-complement to A}, a complement to $C_G(E)$ in $N_G(E)$ isomorphic to $\GL_3(2)$ is in $L$ and by Lemma \ref{HN-HS-Sylow 2}, $|L|_2=2^{11}$.  By Lemma \ref{HN-HS-1-finding 4*4*4},
there exists an element of order four $d\in C_{Q_2}(u)$ such that $d^2=t \in E$ and
${\<d,u\>}\vartriangleleft H \cap L$ which implies that $4\cong \<\wt{d}\>\vartriangleleft \wt{H
\cap L}$. Moreover, $d \in Q_2\leq C_G(E)$ and so $\<\wt{d}\>\vartriangleleft \wt{N_L(E) \cap H}$. So consider $\<\wt{d}^{N_L(E)}\>\leq \wt{C_L(E)}$. Since $\wt{d}^2=\wt{t}\in
\wt{E}$ and $N_L(E)$ is transitive on $E^\#$, we clearly have at least seven conjugates of
$\<\wt{d}\>$ in $\wt{N_L(E)}$. Moreover since $\<\wt{d}\>\vartriangleleft \wt{C_L(E)}$,
the $N_L(E)$-conjugates of $\<\wt{d}\>$ in $\wt{C_L(E)}$ pairwise commute and generate a $2$-group of order at most $2^7$. Thus
$\wt{N_L(E)}\vartriangleright \<\wt{d}^{{N_L(E)}}\>\cong 4 \times 4 \times 4$. Let $A$ be the preimage in $L$ containing $u$. Now by Lemma \ref{HN-HS-complement to A} there exists a complement,
$C$, to  $C_L(E)$ in $N_L(E)$. Moreover, $\wt{C}$ acts non-trivially on $\wt{A}$ and so
$\wt{A}\wt{C}\sim 4^3:\GL_3(2)$. Recall that  $u \in F \leq Q_2$
and by Lemma \ref{HN-HS-complement to A}, $F \leq C_L(E)$. Therefore $\wt{F}$ normalizes $\wt{AC}$.
Furthermore, by Lemma \ref{HN-fours groups centralize A8}, $C_{\wt{L}}(\wt{F})\cong 2 \times
\sym(8)$. In particular, $\wt{F} \nleq \wt{A}$ and $[\wt{A}, \wt{F}] \neq 1$. By Lemma
\ref{HN-HS-complement to A}, $[C,F]=1$. Thus $\wt{ACF}\sim 4^3:(2 \times \GL_3(2))$. Since $C_L(E)$ is a $2$-group, we must have that $\wt{ACF}
= \wt{N_L(E)}$.  We now apply Lemma \ref{Prelims-4*4*4 transfer} to $\wt{L}$ to  say that
$O^2(\wt{L})\neq \wt{L}$ and clearly $[\wt{L}:O^2(\wt{L})]=2$. Notice that $u \in O^2(L)$ because $u$ in $\alt(9)\cong O^3(C_G(a_2))$ and recalling Notation \ref{HN-Alt9notation}, $u\mapsto (1,2)(3,4)$, we see that an element of order four
with image $(1,3,2,4)(5,6)$ squares to $u$. Thus $[{L}:O^2({L})]=2$. Let $L_0=O^2(L)$. It follows that  $\wt{N_{L_0}(E)}\cong
4^3:\GL_3(2))$.

Since $F\leq Q_{12} \cap L$ and $F \nleq L_0$, $Q_{12} \cap L_0<Q_{12} \cap L$. Since $[Q_1,a_2]=Q_1$, and so $Q_1 \leq L'
\leq L_0$. Also $\<d,u\>\cong 4 \times 2$ is normal in $L \cap H$ and clearly $d\in L_0$. Thus $(2_+^{1+4} \ast 4) \times 2\sim
Q_1\<d,u\>=Q_{12} \cap L_0$.  Now $(H \cap L_0)/(Q_{12} \cap O^2(L))\cong \sym(5)$ follows from an
isomorphism theorem since
\[\frac{H \cap L_0}{Q_{12} \cap L_0}=\frac{H \cap L_0}{(H \cap L_0) \cap (Q_{12} \cap L)}\cong  \frac{(H \cap
L_0)(Q_{12} \cap L)}{Q_{12} \cap L}=\frac{H \cap L}{Q_{12} \cap L}\cong \sym(5).\] Thus
$C_{L_0}(t)$ has $2$-radical, $Q_{12} \cap L_0\cong 2_+^{1+4}*4\times 2$ with quotient
$\sym(5)$.
\end{proof}

\begin{lemma}
$L\cong 2^{\cdot}\HS :2$.
\end{lemma}
\begin{proof}
We must prove that $O^2(\wt{L})$ satisfies the hypotheses of  Theorem \ref{Aschbacher-HS}   to
recognize the sporadic simple group $\HS$. Now we have that $\wt{t}$ is an involution in
$\wt{L_0}$ and since $ut$ is not conjugate to $t$, we have that $C_{\wt{L_0}}(\wt{t})=\wt{(C_G(t) \cap L_0)} \sim 2_+^{1+4}*4.\sym(5)$.

Suppose that $g \in O^2(L)$ and $\wt{x}$ normalizes
$\wt{E}$. Then $g$ normalizes $E\<u\>$. Since $N_L(E)$ is transitive on $E^\#$ and we have seen
that $tu \in 2 \mathcal{A}$, we have that $\{eu|e \in E^\#\}\subseteq 2\mathcal{A}$. Therefore
$E\<u\> \cap 2\mathcal{B}=E^\#$. Hence $g$ normalizes $E$. Thus $N_{\wt{O^2(L)}}(\wt{E})=\wt{N_{O^2(L)}(E)}\cong
4^3:\GL_3(2))$.

We therefore apply Theorem \ref{Aschbacher-HS} to $\wt{O^2(L)}$ to see that $\wt{O^2(L)}\cong \HS$ and so $O^2(L)\cong 2^{.}\HS$ or $2\times \HS$. We have seen that $L$ does not split
over $\<u\>$ and we have also seen that $L=FO^2(L)$. Thus we must have that $L\cong 2^{\cdot}\aut(\HS)\cong 2^{\cdot}\HS :2$.
\end{proof}

\begin{lemma}
In Case I, $G \cong \HN$.
\end{lemma}
\begin{proof}
We have that $G$ is a finite group with two involutions $u$ and $t$ and $L=C_G(u) \sim (2^{\cdot}\HS) :
2$. Also $C_G(t)\sim 2_+^{1+8}.(\alt(5)\wr 2)$ and $O_2(H)=Q_{12}$ and by Lemma \ref{HN-Q_i's},
$C_G(Q_{12})\leq Q_{12}$. Thus, by Theorem \ref{Segev-HN}, $G
\cong \HN$.
\end{proof}

Now we recall Case II and Lemma \ref{HN-Indextwo} to see that in Case II, $G$ has proper subgroup $G_0$ of even index and we have proved that $G_0\cong \HN$. By a Frattini argument, $G=G_0N_G(S)$ ($S \in \syl_3(G)$) and so it follows using Lemma \ref{HN-Q_i's} that $[G:G_0]=2$. We check finally that $G\ncong 2\times \HN$ however we can use, for example, that $C_G(J)\leq J$ (Lemma \ref{HN-J is self-centralizing}). Thus, in Case II, $G \cong \mathrm{Aut}(\HN)$.

\section{Acknowledgements}
The article was completed with support from the Heilbronn Institute for Mathematical Research.

\bibliographystyle{plain}
\bibliography{mybibliography}

\end{document}